\documentclass[a4paper]{amsart}
\usepackage{amsmath}
\usepackage{amssymb}
\usepackage{amsthm}  
\usepackage{amscd}
\usepackage{secdot}
\usepackage{stmaryrd}
\usepackage[dvipdfmx]{graphicx}
\usepackage[all]{xy}
\usepackage{wrapfig}  
\usepackage{color}

\usepackage{enumerate}
\usepackage{cases}
\usepackage[top=25mm,bottom=25mm,left=30mm,right=30mm]{geometry}
\usepackage{tikz} 
\usepackage{extarrows}
\usepackage{pxpgfmark}
\usepackage{multirow}    
\usepackage{comment}

\usetikzlibrary{intersections,calc,patterns,decorations.markings}

\numberwithin{equation}{section}
\theoremstyle{definition}
\newtheorem{example}{Example}[section]
\newtheorem{definition}[example]{Definition}

\theoremstyle{plain}
\newtheorem{lemma}[example]{Lemma}
\newtheorem{theorem}[example]{Theorem} 
\newtheorem{proposition}[example]{Proposition}
\newtheorem{corollary}[example]{Corollary}

\DeclareMathOperator{\add}{{\rm add}}
\DeclareMathOperator{\proj}{{\rm proj}}

\DeclareMathOperator{\Hom}{{\rm Hom}}

\DeclareMathOperator{\End}{{\rm End}}

\DeclareMathOperator{\Image}{{\rm Im}}

\DeclareMathOperator{\module}{{\rm mod}}
\DeclareMathOperator{\Db}{\mathsf{D}^{\rm b}}
\DeclareMathOperator{\Kb}{\mathsf{K}^{\rm b}}

\DeclareMathOperator{\thick}{{\rm thick}}

\DeclareMathOperator{\twosilt}{{\rm 2\mathchar`-silt}}
\DeclareMathOperator{\twotilt}{{\rm 2\mathchar`-tilt}}

\title[Brauer tree algebras have $\binom{2n}{n}$ $2$-tilting complexes]{Brauer tree algebras have $\binom{2n}{n}$ $2$-tilting complexes}
\author{Toshitaka Aoki} 
\address{T. Aoki: Graduate School of Mathematics, Nagoya University, Chikusa-ku, Nagoya, 464-8602 Japan}
\email{m15001d@math.nagoya-u.ac.jp}

\begin{document}
\maketitle
\begin{abstract}
We show that any Brauer tree algebra has precisely $\binom{2n}{n}$ $2$-tilting complexes, where $n$ is the number of edges of the associated Brauer tree. 
More explicitly, for an external edge $e$ and an integer $j\neq0$, we show that the number of $2$-tilting complexes $T$ with $g_e(T)=j$ is $\binom{2n-|j|-1}{n-1}$, where $g_e(T)$ denotes the $e$-th of the $g$-vector of $T$.
To prove this, we use a geometric model of Brauer graph algebras on the closed oriented marked surfaces and a classification of $2$-tilting complexes due to Adachi-Aihara-Chan. 
\end{abstract}

\section{Introduction}

Special biserial algebras provide an important class of representation-tame algebras, and finite dimensional symmetric special biserial algebras are precisely Brauer graph algebras.
A Brauer graph algebra is defined from a combinatorial object called a Brauer graph, a finite connected graph equipped with a cyclic permutation of the edges incident to each vertex. 
It is also known as a ribbon graph (or a fat graph) and has a canonical embedding into a closed oriented marked surface \cite{Labourie13,MS14}. 
Recently, a similar construction of algebras from ribbon graphs has been developed in several area of mathematics, such as cluster theory \cite{ABCJP} and Fukaya categories of surfaces \cite{HKK,LP,KS}. 

Brauer graph algebras corresponding to plane trees are called Brauer tree algebras.  
Tilting theory of Brauer tree algebras plays a central role in the study of modular representation theory for cyclic groups. For example, it was first shown by Rickard \cite{Rickard89st} that the class of Brauer tree algebras is closed under derived equivalent. From a point of view of mutation theory, the categorical operation called tilting mutation can be described as a combinatorial operation called flip (or Kauer move) on the associated plane tree \cite{Aihara14,Kauer98}, and all tilting complexes are obtained by iterated mutation from the initial Brauer tree algebra \cite{Aihara13}. 
%Here, we note that tilting complexes coincide with silting complexes for symmetric algebras \cite[Example 2.8]{AI}. 
Two-term tilting complexes ($2$-tilting complexes for short) are closely related to several important notions in representation theory, such as torsion classes, $\tau$-tilting modules and $t$-structures \cite{AIR,IJY,KY}. 
In \cite{AAC}, they classify all $2$-tilting complexes over Brauer graph algebras by using the notion of signed walks and prove that there are only finitely many $2$-tilting complexes over Brauer tree algebras (see also \cite{AZ}).

In this paper, we prove the following result, which determines the number of $2$-tilting complexes over an arbitrary Brauer tree algebra. 
Let $\mathbf{G}$ be a Brauer tree and $G_1$ its edge set. 
We denote by $B_{\mathbf{G}}$ the Brauer tree algebra associated to $\mathbf{G}$, by $\twotilt B_{\mathbf{G}}$ the set of isomorphism classes of basic $2$-tilting complexes for $B_{\mathbf{G}}$. For $e\in G_1$ and $j\in \mathbb{Z}$, we set 
\[
   \twotilt_e^{j} (B_{\mathbf{G}}) :=\{T\in \twotilt B_{\mathbf{G}}\mid g_e(T)=j\}, 
\]
where $g(T)=(g_e(T))_{e\in G_1}$ denotes the $g$-vector of a two-term complex $T$ (see Definition \ref{def:g-vector}), and $[1,n]:=\{1,\ldots,n\}$.

\begin{theorem} (Theorem \ref{thm:main}) \label{theorem main}
   Let $\mathbf{G}$ be a Brauer tree and $B_{\mathbf{G}}$ the Brauer tree algebra associated to $\mathbf{G}$. 
   Let $n=|\mathbf{G}|$ be the number of edges of $\mathbf{G}$.
   \begin{enumerate}
   \item[\rm (1)] For any external edge $e$ of $\mathbf{G}$ and any non-negative integer $j$, we have 
   \[
      \# \twotilt_e^j (B_{\mathbf{G}}) =\# \twotilt_e^{-j} (B_{\mathbf{G}}) = \begin{cases} 
         \binom{2n-j-1}{n-1} & \text{if $j\in [1,n]$}, \\
         0 &\text{otherwise}.
         \end{cases}
   \]  
   \item[\rm (2)] The following equation holds:
   \[
      \#\twotilt B_{\mathbf{G}} = 2\sum_{j=1}^n  \binom{2n-j-1}{n-1} = \binom{2n}{n}.
   \]
   In particular, this number depends only on the number of edges of $\mathbf{G}$. 
   \end{enumerate}
\end{theorem} 
%Note that there are no canonical bijection between $\twotilt_e^j B_{\mathbf{G}}$ and $\twotilt_e^{-j} B_{\mathbf{G}}$ in general.

Each edge $e\in G_1$ determines subtrees $\mathbf{S},\mathbf{T}$ of $\mathbf{G}$ satisfying $S_1\cup T_1 = G_1$ and $S_1\cap T_1=\{e\}$. 
%Clearly, we have $|G|=|S|+|T|-1$, where $|G|$ is the number of edges of $G$. (theorem ) 
We regard both of them as plane trees canonically. See the following figure. 
We study $2$-tilting complexes for $B_{\mathbf{G}}$ from those for Brauer tree subalgebras $B_{\mathbf{S}}$ and $B_{\mathbf{T}}$. 
The following result plays a key role in our proof of Theorem \ref{theorem main}. 
Now, we denote by $\mathrm{P}(s,t)$ the set of all lattice paths in the lattice $[1,s]\times[1,t]$ (see Definition \ref{def:latticepath}).

\[
\begin{tabular}{ccccccccc}
   \small$\mathbf{G}$ & &\small$\mathbf{S}$ & \small$\mathbf{T}$ \\
   \begin{tikzpicture}[scale=0.85,baseline=0mm]
      %\draw[red] (0,0.6)--(0,-0.7);
      %\node(al) at (0,-0.85) {\color{red}\footnotesize$\alpha_e$};
      \coordinate(a) at(0.6,0);
      \coordinate(b) at(-0.6,0);

      \draw (a)--node[fill=white,inner sep=1]{\footnotesize$e$}(b); 
      \draw (a)--($(a)+(45:0.8)$) (a)--($(a)+(0:0.8)$) (a)--($(a)+(-90:0.8)$);
      \draw (b)--($(b)+(120:0.8)$) (b)--($(b)+(180:0.8)$) (b)--($(b)+(-120:0.8)$);
      \draw[fill=white] (a)circle(0.6mm) ;
      \draw[fill=white] (b)circle(0.6mm) ;
      
      %\draw (a)circle(0.6mm);
      %\draw (b)circle(0.6mm);

      \node(xx) at (0,-1) {};
      \node(xx) at (0,1) {};
      
   \end{tikzpicture}
   &
\begin{tikzpicture}[scale=0.85,baseline=0mm]
   \node(sq) at (0,0) {$\rightsquigarrow$};
   \node(xx) at (0,-1) {};
      \node(xx) at (0,1) {};
\end{tikzpicture}
   &
   \begin{tikzpicture}[scale=0.85,baseline=0mm]
      \coordinate(bl) at (0.2,0.6);
      %\draw[fill=black] (bl)circle(0.4mm);
      %\draw[red] (bl)..controls(0.9,0.3)and(1.1,-0.2)..(0.8,-0.3);
      %\draw[red] (0.8,-0.3)..controls(0.2,-0.5)and(0.1,0.3)..(bl);
      %\node at (0.5,-0.6){\color{red}\footnotesize$\alpha_e$};
      \coordinate(a) at(0.6,0);
      \coordinate(b) at(-0.6,0);
      \draw (a)--node[fill=white,inner sep=1]{\footnotesize$e$}(b); 
      %\draw (a)--($(a)+(45:0.8)$) (a)--($(a)+(0:0.8)$) (a)--($(a)+(-90:0.8)$);
      \draw (b)--($(b)+(120:0.8)$) (b)--($(b)+(180:0.8)$) (b)--($(b)+(-120:0.8)$);
      
      \draw[fill=white] (a)circle(0.6mm) ;
      \draw[fill=white] (b)circle(0.6mm) ;
      \node(xx) at (0,-1) {};
      \node(xx) at (0,1) {};
   
   \end{tikzpicture}
   & 
   \begin{tikzpicture}[scale=0.85,baseline=0mm]
      \coordinate(bl) at (-0.2,0.6);
      %\draw[fill=black] (bl)circle(-0.4mm);
      %\draw[red] (bl)..controls(-0.9,0.3)and(-1.1,-0.2)..(-0.8,-0.3);
      %\draw[red] (-0.8,-0.3)..controls(-0.2,-0.5)and(-0.1,0.3)..(bl);
      %\node at (-0.5,-0.6){\color{red}\footnotesize$\alpha_e$};

      \coordinate(a) at(0.6,0);
      \coordinate(b) at(-0.6,0);
      \draw (a)--node[fill=white,inner sep=1]{\footnotesize$e$}(b); 
      \draw (a)--($(a)+(45:0.8)$) (a)--($(a)+(0:0.8)$) (a)--($(a)+(-90:0.8)$);
      %\draw (b)--($(b)+(120:0.8)$) (b)--($(b)+(180:0.8)$) (b)--($(b)+(-120:0.8)$);
      
      \draw[fill=white] (a)circle(0.6mm) ;
      \draw[fill=white] (b)circle(0.6mm) ;
      \node(xx) at (0,-1) {};
      \node(xx) at (0,1) {};
   
   \end{tikzpicture}
\end{tabular}
\]

\begin{theorem} (Theorem \ref{thm:gluing}) \label{theorem sub}
   In the above, for each $s \in [1,|\mathbf{S}|]$ and $t\in [1,|\mathbf{T}|]$, we have injective maps 
      \begin{eqnarray}
         \rho_{e}^{s,t} \colon& \twotilt_e^{s} (B_{\mathbf{S}}) \times \twotilt_{e}^t (B_{\mathbf{T}}) \times \mathrm{P}(s,t) &\to\ \twotilt_e^{s+t-1} (B_{\mathbf{G}}) \quad \text{and}\\
         \rho_{e}^{-s,-t} \colon& \twotilt_e^{-s} (B_{\mathbf{S}}) \times \twotilt_{e}^{-t} (B_{\mathbf{T}}) \times \mathrm{P}(s,t) &\to\ \twotilt_e^{-s-t+1} (B_{\mathbf{G}}). 
      \end{eqnarray}
   Furthermore, for each integer $j\in [1,|\mathbf{G}|]$, they provide decompositions
   \[
      \twotilt_e^j (B_{\mathbf{G}}) = \bigsqcup_{\substack{s\in [1,|\mathbf{S}|] \\ t\in [1,|\mathbf{T}|]\\ j=s+t-1}} \Image \rho_e^{s,t} \quad \text{and} \quad \twotilt_e^{-j} (B_{\mathbf{G}}) = \bigsqcup_{\substack{s\in [1,|\mathbf{S}|] \\ t \in [1,|\mathbf{T}|]\\ j=s+t-1}} \Image \rho_e^{-s,-t}.
   \]
\end{theorem}

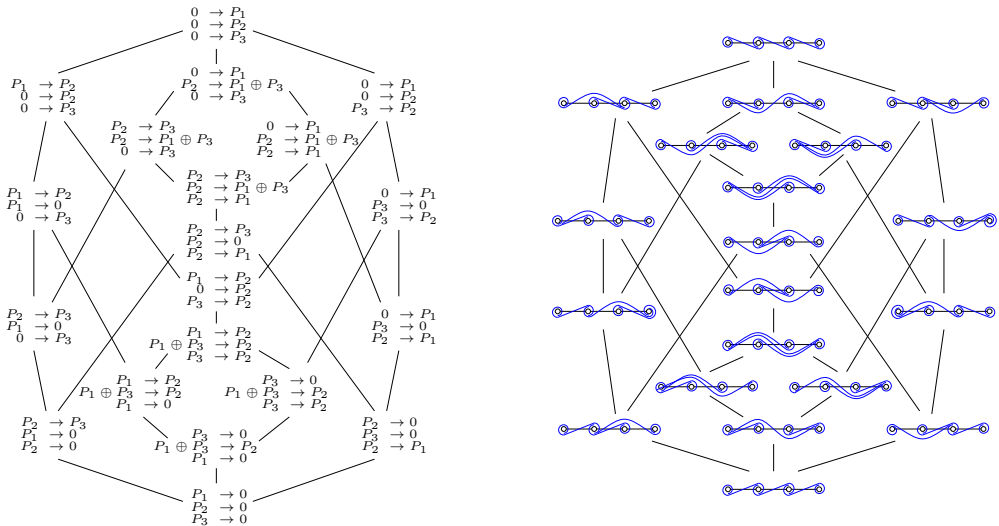
\begin{figure}[b]  
   \scalebox{0.8}{
   \begin{tabular}{ccccccccc}
   \begin{tikzpicture}[baseline=0]
      %%%%%%%%%%%%%%%%%%%%%%%%%%%%%%%%%%%%%%%%%%%
      
      \coordinate(x) at (0,0.4);
      \node(1) at($(x)+(0,0.2)$) {\tiny$\ P_2\ \to P_3$};
      \node(2) at($(x)+(0,0)$) {\tiny$\ P_2\ \to 0 \ \  $};
      \node(3) at($(x)+(0,-0.2)$) {\tiny$\  P_2\ \to P_1$};
      \coordinate(ul) at ($(x)+(-0.6,0.35)$);
      \coordinate(dr) at ($(x)+(0.6,-0.35)$);
      \draw ($(x)+(0,0.35)$)--($(x)+(0,0.55)$);
      
      %%%%%%%%%%%%%%%%%%%%%%%%%%%%%%%%%%%%%%%%%%%
      %%%%%%%%%%%%%%%%%%%%%%%%%%%%%%%%%%%%%%%%%%%
      \coordinate(x) at (0,1.3);
      \node(1) at($(x)+(0,0.2)$) {\tiny$\ P_2\ \to P_3$};
      \node(2) at($(x)+(0,0)$) {\tiny$\ \quad \quad \ P_2\ \to P_1\oplus P_3$};
      \node(3) at($(x)+(0,-0.2)$) {\tiny$\  P_2\ \to P_1$};
      \coordinate(ul) at ($(x)+(-0.6,0.35)$);
      \coordinate(dr) at ($(x)+(0.6,-0.35)$);
      \draw ($(x)+(-0.7,0.1)$)--($(x)+(-1,0.4)$);
      \draw ($(x)+(1.2,0.1)$)--($(x)+(1.5,0.4)$);
      %\draw (ul) rectangle (dr);
      %%%%%%%%%%%%%%%%%%%%%%%%%%%%%%%%%%%%%%%%%%%
      %%%%%%%%%%%%%%%%%%%%%%%%%%%%%%%%%%%%%%%%%%%
      \coordinate(x) at (-1.2,2.1);
      \node(1) at($(x)+(0,0.2)$) {\tiny$\ P_2\ \to P_3 \  $};
      \node(2) at($(x)+(0,0)$) {\tiny$\ \quad \quad P_2\ \to P_1\oplus P_3$};
      \node(3) at($(x)+(0,-0.2)$) {\tiny$\ \  0\ \to P_3$};
      \coordinate(ul) at ($(x)+(-0.6,0.35)$);
      \coordinate(dr) at ($(x)+(0.6,-0.35)$);
      %\draw (ul) rectangle (dr);
      %%%%%%%%%%%%%%%%%%%%%%%%%%%%%%%%%%%%%%%%%%%
      %%%%%%%%%%%%%%%%%%%%%%%%%%%%%%%%%%%%%%%%%%%
      \coordinate(x) at (1.2,2.1);
      \node(1) at($(x)+(0,0.2)$) {\tiny$\ \ 0\ \to P_1$};
      \node(2) at($(x)+(0,0)$) {\tiny$\ \quad \quad P_2\ \to P_1\oplus P_3$};
      \node(3) at($(x)+(0,-0.2)$) {\tiny$\  P_2\ \to P_1 \ $};
      \coordinate(ul) at ($(x)+(-0.6,0.35)$);
      \coordinate(dr) at ($(x)+(0.6,-0.35)$);
      %\draw (ul) rectangle (dr);
      %%%%%%%%%%%%%%%%%%%%%%%%%%%%%%%%%%%%%%%%%%%
      %%%%%%%%%%%%%%%%%%%%%%%%%%%%%%%%%%%%%%%%%%%
      \coordinate(x) at (0,3);
      \node(1) at($(x)+(0,0.2)$) {\tiny$\ 0\ \to P_1$};
      \node(2) at($(x)+(0,0)$) {\tiny$\ \quad \ \ P_2\ \to P_1\oplus P_3$};
      \node(3) at($(x)+(0,-0.2)$) {\tiny$\  0\ \to P_3$};
      \coordinate(ul) at ($(x)+(-0.6,0.35)$);
      \coordinate(dr) at ($(x)+(0.6,-0.35)$);
      \draw ($(x)+(-0.7,-0.1)$)--($(x)+(-1,-0.5)$);
      \draw ($(x)+(1.2,-0.1)$)--($(x)+(1.5,-0.5)$);
      \draw ($(x)+(0,0.35)$)--($(x)+(0,0.6)$);
      %\draw (ul) rectangle (dr);
      %%%%%%%%%%%%%%%%%%%%%%%%%%%%%%%%%%%%%%%%%%%
      %%%%%%%%%%%%%%%%%%%%%%%%%%%%%%%%%%%%%%%%%%%
      \coordinate(x) at (0,4);
      \node(1) at($(x)+(0,0.2)$) {\tiny$\ 0\ \to P_1$};
      \node(2) at($(x)+(0,0)$) {\tiny$\ 0\ \to P_2$};
      \node(3) at($(x)+(0,-0.2)$) {\tiny$\  0\ \to P_3$};
      \coordinate(ul) at ($(x)+(-0.6,0.35)$);
      \coordinate(dr) at ($(x)+(0.6,-0.35)$);
      \draw ($(x)+(-0.55,-0.1)$)--($(x)+(-2.6,-0.8)$);
      \draw ($(x)+(0.6,-0.1)$)--($(x)+(2.6,-0.8)$);
      %\draw (ul) rectangle (dr);
      %%%%%%%%%%%%%%%%%%%%%%%%%%%%%%%%%%%%%%%%%%%
      %%%%%%%%%%%%%%%%%%%%%%%%%%%%%%%%%%%%%%%%%%%
      \coordinate(x) at (-3,1);
      \node(1) at($(x)+(0,0.2)$) {\tiny$\ \ P_1\ \to P_2 $};
      \node(2) at($(x)+(0,0)$) {\tiny$\ P_1\ \to 0\ $};
      \node(3) at($(x)+(0,-0.2)$) {\tiny$\ \quad 0\ \to P_3$};
      \coordinate(ul) at ($(x)+(-0.6,0.35)$);
      \coordinate(dr) at ($(x)+(0.6,-0.35)$);
      \draw ($(x)+(0,0.4)$)--($(x)+(0.2,1.4)$);
      \draw ($(x)+(0,-0.4)$)--($(x)+(0,-1.5)$);
      \draw ($(x)+(0.3,-0.4)$)--($(x)+(1.5,-2.6)$);

      %\draw (ul) rectangle (dr);
      %%%%%%%%%%%%%%%%%%%%%%%%%%%%%%%%%%%%%%%%%%%
      %%%%%%%%%%%%%%%%%%%%%%%%%%%%%%%%%%%%%%%%%%%
      \coordinate(x) at (3,1);
      \node(1) at($(x)+(0,0.2)$) {\tiny$\ \quad 0\ \to P_1$};
      \node(2) at($(x)+(0,0)$) {\tiny$\ P_3\ \to 0\ $};
      \node(3) at($(x)+(0,-0.2)$) {\tiny$\ \ P_3\ \to P_2$};
      \coordinate(ul) at ($(x)+(-0.6,0.35)$);
      \coordinate(dr) at ($(x)+(0.6,-0.35)$);
      \draw ($(x)+(0,0.4)$)--($(x)+(-0.2,1.4)$);
      \draw ($(x)+(0,-0.4)$)--($(x)+(0,-1.5)$);
      \draw ($(x)+(-0.3,-0.4)$)--($(x)+(-1.5,-2.6)$);
      
      %\draw (ul) rectangle (dr);
      %%%%%%%%%%%%%%%%%%%%%%%%%%%%%%%%%%%%%%%%%%%
      %%%%%%%%%%%%%%%%%%%%%%%%%%%%%%%%%%%%%%%%%%%
      \coordinate(x) at (-2.8,2.8);
      \node(1) at($(x)+(0,0.2)$) {\tiny$\ P_1\ \to P_2 \ \ $};
      \node(2) at($(x)+(0,0)$) {\tiny$\ 0\ \to P_2$};
      \node(3) at($(x)+(0,-0.2)$) {\tiny$\  0\ \to P_3$};
      \coordinate(ul) at ($(x)+(-0.6,0.35)$);
      \coordinate(dr) at ($(x)+(0.6,-0.35)$);
      %\draw (ul) rectangle (dr);
      \draw ($(x)+(0.3,-0.4)$)--($(x)+(2.2,-3)$);
      %%%%%%%%%%%%%%%%%%%%%%%%%%%%%%%%%%%%%%%%%%%
      %%%%%%%%%%%%%%%%%%%%%%%%%%%%%%%%%%%%%%%%%%%
      \coordinate(x) at (2.8,2.8);
      \node(1) at($(x)+(0,0.2)$) {\tiny$\ 0\ \to P_1$};
      \node(2) at($(x)+(0,0)$) {\tiny$\ 0\ \to P_2$};
      \node(3) at($(x)+(0,-0.2)$) {\tiny$\  P_3\ \to P_2 \ \ $};
      \coordinate(ul) at ($(x)+(-0.6,0.35)$);
      \coordinate(dr) at ($(x)+(0.6,-0.35)$);
      \draw ($(x)+(-0.15,-0.4)$)--($(x)+(-2.1,-3)$);
      %\draw (ul) rectangle (dr);
      %%%%%%%%%%%%%%%%%%%%%%%%%%%%%%%%%%%%%%%%%%%
      %%%%%%%%%%%%%%%%%%%%%%%%%%%%%%%%%%%%%%%%%%%
      \coordinate(x) at (0,-0.4);
      \node(1) at($(x)+(0,0.2)$) {\tiny$\ P_1\ \to P_2$};
      \node(2) at($(x)+(0,0)$) {\tiny$\ \ \ 0\ \to P_2$};
      \node(3) at($(x)+(0,-0.2)$) {\tiny$\  P_3\ \to P_2$};
      \coordinate(ul) at ($(x)+(-0.6,0.35)$);
      \coordinate(dr) at ($(x)+(0.6,-0.35)$);
      \draw ($(x)+(0,-0.35)$)--($(x)+(0,-0.55)$);
      %\draw (ul) rectangle (dr);
      %%%%%%%%%%%%%%%%%%%%%%%%%%%%%%%%%%%%%%%%%%%
      %%%%%%%%%%%%%%%%%%%%%%%%%%%%%%%%%%%%%%%%%%%
      \coordinate(x) at (0,-1.3);
      \node(1) at($(x)+(0,0.2)$) {\tiny$\ P_1\ \to P_2$};
      \node(2) at($(x)+(0,0)$) {\tiny$\ P_1\oplus P_3\ \to P_2 \quad \quad \ $};
      \node(3) at($(x)+(0,-0.2)$) {\tiny$\  P_3\ \to P_2$};
      \coordinate(ul) at ($(x)+(-0.6,0.35)$);
      \coordinate(dr) at ($(x)+(0.6,-0.35)$);
      \draw ($(x)+(-0.8,-0.15)$)--($(x)+(-1,-0.4)$);
      \draw ($(x)+(0.7,-0.1)$)--($(x)+(1.2,-0.4)$);
      %\draw (ul) rectangle (dr);
      %%%%%%%%%%%%%%%%%%%%%%%%%%%%%%%%%%%%%%%%%%%
      %%%%%%%%%%%%%%%%%%%%%%%%%%%%%%%%%%%%%%%%%%%
      \coordinate(x) at (-1.2,-2.1);
      \node(1) at($(x)+(0,0.2)$) {\tiny$\ \ P_1\ \to P_2$};
      \node(2) at($(x)+(0,0)$) {\tiny$\ P_1\oplus P_3\ \to P_2 \quad \quad $};
      \node(3) at($(x)+(0,-0.2)$) {\tiny$\  P_1\ \to 0\ $};
      \coordinate(ul) at ($(x)+(-0.6,0.35)$);
      \coordinate(dr) at ($(x)+(0.6,-0.35)$);
      %\draw (ul) rectangle (dr);
      %%%%%%%%%%%%%%%%%%%%%%%%%%%%%%%%%%%%%%%%%%%
      %%%%%%%%%%%%%%%%%%%%%%%%%%%%%%%%%%%%%%%%%%%
      \coordinate(x) at (1.2,-2.1);
      \node(1) at($(x)+(0,0.2)$) {\tiny$\ P_3\ \to 0\ $};
      \node(2) at($(x)+(0,0)$) {\tiny$\ P_1\oplus P_3\ \to P_2 \quad \quad $};
      \node(3) at($(x)+(0,-0.2)$) {\tiny$\ \ P_3\ \to P_2$};
      \coordinate(ul) at ($(x)+(-0.6,0.35)$);
      \coordinate(dr) at ($(x)+(0.6,-0.35)$);
      %\draw (ul) rectangle (dr);
      %%%%%%%%%%%%%%%%%%%%%%%%%%%%%%%%%%%%%%%%%%%
      %%%%%%%%%%%%%%%%%%%%%%%%%%%%%%%%%%%%%%%%%%%
      \coordinate(x) at (0,-3);
      \node(1) at($(x)+(0,0.2)$) {\tiny$\ P_3\ \to 0$};
      \node(2) at($(x)+(0,0)$) {\tiny$\ P_1\oplus P_3\ \to P_2 \quad \ \ $};
      \node(3) at($(x)+(0,-0.2)$) {\tiny$\  P_1\ \to 0$};
      \coordinate(ul) at ($(x)+(-0.6,0.35)$);
      \coordinate(dr) at ($(x)+(0.6,-0.35)$);
      \draw ($(x)+(-0.8,0.15)$)--($(x)+(-1.2,0.5)$);
      \draw ($(x)+(0.7,0.1)$)--($(x)+(1.2,0.5)$);
      \draw ($(x)+(0,-0.35)$)--($(x)+(0,-0.6)$);
      %\draw (ul) rectangle (dr);
      %%%%%%%%%%%%%%%%%%%%%%%%%%%%%%%%%%%%%%%%%%%
      %%%%%%%%%%%%%%%%%%%%%%%%%%%%%%%%%%%%%%%%%%%
      \coordinate(x) at (0,-4);
      \node(1) at($(x)+(0,0.2)$) {\tiny$\ P_1\ \to 0$};
      \node(2) at($(x)+(0,0)$) {\tiny$\ P_2\ \to 0$};
      \node(3) at($(x)+(0,-0.2)$) {\tiny$\  P_3\ \to 0$};
      \coordinate(ul) at ($(x)+(-0.6,0.35)$);
      \coordinate(dr) at ($(x)+(0.6,-0.35)$);
      \draw ($(x)+(-0.55,0.1)$)--($(x)+(-2.6,0.8)$);
      \draw ($(x)+(0.6,0.1)$)--($(x)+(2.6,0.8)$);
      %\draw (ul) rectangle (dr);
      %%%%%%%%%%%%%%%%%%%%%%%%%%%%%%%%%%%%%%%%%%%
      %%%%%%%%%%%%%%%%%%%%%%%%%%%%%%%%%%%%%%%%%%%
      \coordinate(x) at (-3,-1);
      \node(1) at($(x)+(0,0.2)$) {\tiny$\ \ P_2\ \to P_3$};
      \node(2) at($(x)+(0,0)$) {\tiny$\ P_1\ \to 0\ $};
      \node(3) at($(x)+(0,-0.2)$) {\tiny$\ \quad 0\ \to P_3$};
      \coordinate(ul) at ($(x)+(-0.6,0.35)$);
      \coordinate(dr) at ($(x)+(0.6,-0.35)$);
      \draw ($(x)+(0,-0.4)$)--($(x)+(0.2,-1.4)$);
      %\draw ($(x)+(0,0.4)$)--($(x)+(0,1.5)$);
      \draw ($(x)+(0.3,0.4)$)--($(x)+(1.5,2.6)$);
      %\draw (ul) rectangle (dr);
      %%%%%%%%%%%%%%%%%%%%%%%%%%%%%%%%%%%%%%%%%%%
      %%%%%%%%%%%%%%%%%%%%%%%%%%%%%%%%%%%%%%%%%%%
      \coordinate(x) at (3,-1);
      \node(1) at($(x)+(0,0.2)$) {\tiny$\ \quad 0\ \to P_1$};
      \node(2) at($(x)+(0,0)$) {\tiny$\ P_3\ \to 0\ $};
      \node(3) at($(x)+(0,-0.2)$) {\tiny$\  \ P_2\ \to P_1$};
      \coordinate(ul) at ($(x)+(-0.6,0.35)$);
      \coordinate(dr) at ($(x)+(0.6,-0.35)$);
      \draw ($(x)+(0,-0.4)$)--($(x)+(-0.2,-1.4)$);
      %\draw ($(x)+(0,0.4)$)--($(x)+(0,1.5)$);
      \draw ($(x)+(-0.3,0.4)$)--($(x)+(-1.2,2.7)$);
      %\draw (ul) rectangle (dr);
      %%%%%%%%%%%%%%%%%%%%%%%%%%%%%%%%%%%%%%%%%%%
      %%%%%%%%%%%%%%%%%%%%%%%%%%%%%%%%%%%%%%%%%%%
      \coordinate(x) at (-2.8,-2.8);
      \node(1) at($(x)+(0,0.2)$) {\tiny$\ \ \  P_2\ \to P_3$};
      \node(2) at($(x)+(0,0)$) {\tiny$\ P_1\ \to 0$};
      \node(3) at($(x)+(0,-0.2)$) {\tiny$\  P_2\ \to 0$};
      \coordinate(ul) at ($(x)+(-0.6,0.35)$);
      \coordinate(dr) at ($(x)+(0.6,-0.35)$);
      \draw ($(x)+(0.2,0.4)$)--($(x)+(2.1,3)$);
      %\draw (ul) rectangle (dr);
      %%%%%%%%%%%%%%%%%%%%%%%%%%%%%%%%%%%%%%%%%%%
      %%%%%%%%%%%%%%%%%%%%%%%%%%%%%%%%%%%%%%%%%%%
      \coordinate(x) at (2.8,-2.8);
      \node(1) at($(x)+(0,0.2)$) {\tiny$\ P_2\ \to 0$};
      \node(2) at($(x)+(0,0)$) {\tiny$\ P_3\ \to 0$};
      \node(3) at($(x)+(0,-0.2)$) {\tiny$\ \ \ P_2\ \to P_1 $};
      \coordinate(ul) at ($(x)+(-0.6,0.35)$);
      \coordinate(dr) at ($(x)+(0.6,-0.35)$);
      \draw ($(x)+(-0.2,0.4)$)--($(x)+(-2.1,3)$);
      %\draw (ul) rectangle (dr);
      %%%%%%%%%%%%%%%%%%%%%%%%%%%%%%%%%%%%%%%%%%%
   \end{tikzpicture}
   & \quad \quad \quad \quad 
   \begin{tikzpicture}[baseline=0]
      %%%%%%%%%%%%%%%%%%%%%%%%%%%%%%%%%%%%%%%%%%
      \coordinate (x) at(0,0.4);
      \coordinate(a) at ($(x)+(0,0)$);
      \coordinate(b) at ($(x)+(0.5,0)$);
      \coordinate(c) at ($(x)+(1,0)$);
      \coordinate(d) at ($(x)+(1.5,0)$);
      \draw (a)--(b)--(c)--(d);
      \draw[fill=white](a)circle(0.4mm);
      \draw[fill=white](b)circle(0.4mm);
      \draw[fill=white](c)circle(0.4mm);
      \draw[fill=white](d)circle(0.4mm);

      \draw[blue] ($(d)+(-0.06,-0.07)$)..controls($(c)+(0,0.28)$)..($(b)+(0.06,-0.06)$);
      \draw[blue] ($(d)+(-0.06,-0.07)$)arc(-120:90:0.85mm)arc(90:270:0.7mm);

      \draw[blue] ($(a)+(0.03,0.08)$)..controls($(b)+(0,-0.28)$)..($(c)+(-0.06,0.06)$);
      \draw[blue] ($(a)+(0.03,0.08)$)arc(80:270:0.85mm)arc(270:450:0.7mm);

      \draw[blue]($(b)+(0,-0.1)$)--($(c)+(0,0.1)$);
      \draw[blue]($(c)+(0,0.1)$)arc(90:-90:0.85mm)arc(-90:-270:0.7mm);
      \draw[blue]($(b)+(0,-0.1)$)arc(-90:-270:0.85mm)arc(-270:-450:0.7mm);
   
      \draw ($(x)+(0.75,0.25)$)--($(x)+(0.75,0.65)$);
      
      %%%%%%%%%%%%%%%%%%%%%%%%%%%%%%%%%%%%%%%%%%%
      %%%%%%%%%%%%%%%%%%%%%%%%%%%%%%%%%%%%%%%%%%
      \coordinate (x) at(0,1.3);
      \coordinate(a) at ($(x)+(0,0)$);
      \coordinate(b) at ($(x)+(0.5,0)$);
      \coordinate(c) at ($(x)+(1,0)$);
      \coordinate(d) at ($(x)+(1.5,0)$);
      \draw (a)--(b)--(c)--(d);
      \draw[fill=white](a)circle(0.4mm);
      \draw[fill=white](b)circle(0.4mm);
      \draw[fill=white](c)circle(0.4mm);
      \draw[fill=white](d)circle(0.4mm);

      \draw[blue] ($(d)+(0,-0.1)$)arc(-90:90:0.85mm)arc(90:270:0.7mm);

      \draw[blue] ($(d)+(-0.07,-0.03)$)..controls($(c)+(0,0.28)$)..($(b)+(0.04,-0.08)$)arc(120+180:-90+180:0.85mm)arc(-90+180:-270+180:0.7mm);
   
      \draw[blue] ($(a)+(0.07,0.03)$)..controls($(b)+(0,-0.28)$)..($(c)+(-0.04,0.08)$)arc(120:-90:0.85mm)arc(-90:-270:0.7mm);
   
      \draw[blue] ($(d)+(0,-0.1)$)..controls($(c)+(0,0.22)$)..($(b)!0.5!(c)$);
      \draw[blue] ($(a)+(0,0.1)$)..controls($(b)+(0,-0.22)$)..($(b)!0.5!(c)$);
      \draw[blue] ($(a)+(0,0.1)$)arc(-90+180:90+180:0.85mm)arc(90+180:270+180:0.7mm);

      \draw ($(x)+(0.25,0.2)$)--($(x)+(-0.3,0.55)$);
      \draw ($(x)+(1.5,0.2)$)--($(x)+(1.9,0.55)$);
      %%%%%%%%%%%%%%%%%%%%%%%%%%%%%%%%%%%%%%%%%%%
      %%%%%%%%%%%%%%%%%%%%%%%%%%%%%%%%%%%%%%%%%%
      \coordinate (x) at(-1.1,2);
      \coordinate(a) at ($(x)+(0,0)$);
      \coordinate(b) at ($(x)+(0.5,0)$);
      \coordinate(c) at ($(x)+(1,0)$);
      \coordinate(d) at ($(x)+(1.5,0)$);
      \draw (a)--(b)--(c)--(d);
      \draw[fill=white](a)circle(0.4mm);
      \draw[fill=white](b)circle(0.4mm);
      \draw[fill=white](c)circle(0.4mm);
      \draw[fill=white](d)circle(0.4mm);
   
      \draw ($(x)+(0.75,0.2)$)--($(x)+(1.3,0.55)$);

      \draw[blue]($(c)+(0,0.1)$)--($(d)+(0,-0.1)$);
      \draw[blue]($(d)+(0,-0.1)$)arc(90+180:270+180:0.85mm)arc(270+180:450+180:0.7mm);
      \draw[blue]($(c)+(0,0.1)$)arc(-90+180:90+180:0.85mm)arc(90+180:270+180:0.7mm);
   
      \draw[blue] ($(d)+(-0.07,-0.05)$)..controls($(c)+(0,0.22)$)..($(b)!0.5!(c)$);
      \draw[blue] ($(a)+(0,0.1)$)..controls($(b)+(0,-0.22)$)..($(b)!0.5!(c)$);
      \draw[blue] ($(a)+(0,0.1)$)arc(-90+180:90+180:0.85mm)arc(90+180:270+180:0.7mm);
   
      \draw[blue] ($(d)+(-0.07,-0.03)$)..controls($(c)+(0,0.28)$)..($(b)+(0.04,-0.08)$)arc(120+180:-90+180:0.85mm)arc(-90+180:-270+180:0.7mm);

      %%%%%%%%%%%%%%%%%%%%%%%%%%%%%%%%%%%%%%%%%%%
      %%%%%%%%%%%%%%%%%%%%%%%%%%%%%%%%%%%%%%%%%%
      \coordinate (x) at(1.1,2);
      \coordinate(a) at ($(x)+(0,0)$);
      \coordinate(b) at ($(x)+(0.5,0)$);
      \coordinate(c) at ($(x)+(1,0)$);
      \coordinate(d) at ($(x)+(1.5,0)$);
      \draw (a)--(b)--(c)--(d);
      \draw[fill=white](a)circle(0.4mm);
      \draw[fill=white](b)circle(0.4mm);
      \draw[fill=white](c)circle(0.4mm);
      \draw[fill=white](d)circle(0.4mm);

      \draw ($(x)+(0.75,0.2)$)--($(x)+(0.05,0.55)$);
   
      \draw[blue]($(b)+(0,-0.1)$)--($(a)+(0,0.1)$);
      \draw[blue]($(a)+(0,0.1)$)arc(90:270:0.85mm)arc(270:450:0.7mm);
      \draw[blue]($(b)+(0,-0.1)$)arc(-90:90:0.85mm)arc(90:270:0.7mm);
   
      \draw[blue] ($(a)+(0.07,0.05)$)..controls($(b)+(0,-0.22)$)..($(b)!0.5!(c)$);
      \draw[blue] ($(d)+(0,-0.1)$)..controls($(c)+(0,0.22)$)..($(b)!0.5!(c)$);
      \draw[blue] ($(d)+(0,-0.1)$)arc(-90:90:0.85mm)arc(90:270:0.7mm);
   
      \draw[blue] ($(a)+(0.07,0.03)$)..controls($(b)+(0,-0.28)$)..($(c)+(-0.04,0.08)$)arc(120:-90:0.85mm)arc(-90:-270:0.7mm);

      %%%%%%%%%%%%%%%%%%%%%%%%%%%%%%%%%%%%%%%%%%%
      \coordinate (x) at(0,2.7);
      \coordinate(a) at ($(x)+(0,0)$);
      \coordinate(b) at ($(x)+(0.5,0)$);
      \coordinate(c) at ($(x)+(1,0)$);
      \coordinate(d) at ($(x)+(1.5,0)$);
      \draw (a)--(b)--(c)--(d);
      \draw[fill=white](a)circle(0.4mm);
      \draw[fill=white](b)circle(0.4mm);
      \draw[fill=white](c)circle(0.4mm);
      \draw[fill=white](d)circle(0.4mm);
   
      \draw ($(x)+(0.75,0.25)$)--($(x)+(0.75,0.7)$);
   
      \draw[blue]($(b)+(0,-0.1)$)--($(a)+(0,0.1)$);
      \draw[blue]($(a)+(0,0.1)$)arc(90:270:0.85mm)arc(270:450:0.7mm);
      \draw[blue]($(b)+(0,-0.1)$)arc(-90:90:0.85mm)arc(90:270:0.7mm);
   
      \draw[blue] ($(a)+(0.07,0.05)$)..controls($(b)+(0,-0.22)$)..($(b)!0.5!(c)$);
      \draw[blue] ($(d)+(-0.07,-0.05)$)..controls($(c)+(0,0.22)$)..($(b)!0.5!(c)$);
   
      \draw[blue]($(d)+(0,-0.1)$)--($(c)+(0,0.1)$);
      \draw[blue]($(c)+(0,0.1)$)arc(90:270:0.85mm)arc(270:450:0.7mm);
      \draw[blue]($(d)+(0,-0.1)$)arc(-90:90:0.85mm)arc(90:270:0.7mm);
      %%%%%%%%%%%%%%%%%%%%%%%%%%%%%%%%%%%%%%%%%%%
      \coordinate (x) at(0,3.7);
      \coordinate(a) at ($(x)+(0,0)$);
      \coordinate(b) at ($(x)+(0.5,0)$);
      \coordinate(c) at ($(x)+(1,0)$);
      \coordinate(d) at ($(x)+(1.5,0)$);
      \draw (a)--(b)--(c)--(d);
      \draw[fill=white](a)circle(0.4mm);
      \draw[fill=white](b)circle(0.4mm);
      \draw[fill=white](c)circle(0.4mm);
      \draw[fill=white](d)circle(0.4mm);

      \draw ($(x)+(0.75,0)+(-0.4,-0.3)$)--($(x)+(0.75,0)+(-2,-0.8)$);
      \draw ($(x)+(0.75,0)+(0.4,-0.3)$)--($(x)+(0.75,0)+(2,-0.8)$);
      
      \draw[blue]($(b)+(0,-0.1)$)--($(a)+(0,0.1)$);
      \draw[blue]($(a)+(0,0.1)$)arc(90:270:0.85mm)arc(270:450:0.7mm);
      \draw[blue]($(b)+(0,-0.1)$)arc(-90:90:0.85mm);
   
      \draw[blue]($(c)+(0,-0.1)$)--($(b)+(0,0.1)$);
      \draw[blue]($(b)+(0,0.1)$)arc(90:270:0.85mm);
      \draw[blue]($(c)+(0,-0.1)$)arc(-90:90:0.85mm);
   
      \draw[blue]($(d)+(0,-0.1)$)--($(c)+(0,0.1)$);
      \draw[blue]($(c)+(0,0.1)$)arc(90:270:0.85mm);
      \draw[blue]($(d)+(0,-0.1)$)arc(-90:90:0.85mm)arc(90:270:0.7mm);
      %%%%%%%%%%%%%%%%%%%%%%%%%%%%%%%%%%%%%%%%%%%
   
      \coordinate (x) at(2.8,0.75);
      \coordinate(a) at ($(x)+(0,0)$);
      \coordinate(b) at ($(x)+(0.5,0)$);
      \coordinate(c) at ($(x)+(1,0)$);
      \coordinate(d) at ($(x)+(1.5,0)$);
      \draw (a)--(b)--(c)--(d);
      \draw[fill=white](a)circle(0.4mm);
      \draw[fill=white](b)circle(0.4mm);
      \draw[fill=white](c)circle(0.4mm);
      \draw[fill=white](d)circle(0.4mm);
   
      \draw ($(x)+(0.75,0)+(0,-0.3)$)--($(x)+(0.75,0)+(0,-1.2)$);
   
      \draw ($(x)+(0.75,0)+(0,0.3)$)--($(x)+(0.75,0)+(-0.2,1.7)$);
   
      \draw ($(x)+(0.75,0)+(-0.3,-0.3)$)--($(x)+(0.75,0)+(-1.6,-2.5)$);

      \draw[blue]($(b)+(0,-0.1)$)--($(a)+(0,0.1)$);
      \draw[blue]($(a)+(0,0.1)$)arc(90:270:0.85mm)arc(270:450:0.7mm);
      \draw[blue]($(b)+(0,-0.1)$)arc(-90:90:0.85mm);
   
      %\draw[blue]($(c)+(0,-0.1)$)--($(b)+(0,0.1)$);
      \draw[blue]($(b)+(0.06,0.08)$)arc(60:270:0.85mm);
      %\draw[blue]($(c)+(0,-0.1)$)arc(-90:90:0.85mm)arc(90:270:0.7mm);
   
      \draw[blue] ($(b)+(0.06,0.08)$)..controls($(c)+(0,-0.24)$)..($(d)+(-0.04,0.08)$);
      \draw[blue]($(d)+(-0.04,0.08)$)arc(120:-90:0.7mm);
      \draw[blue] ($(d)+(-0,-0.1)$)arc(-90:-270:0.8mm);
      \draw[blue] ($(d)+(-0,-0.1)$)arc(-90:90:1.1mm);
      \draw[blue] ($(c)+(-0,-0.07)$)arc(-90:90:0.7mm)arc(90:270:0.85mm)--($(d)+(-0,0.12)$);

      %%%%%%%%%%%%%%%%%%%%%%%%%%%%%%%%%%%%%%%%%%%
      
      %%%%%%%%%%%%%%%%%%%%%%%%%%%%%%%%%%%%%%%%%%%
      \coordinate (x) at(2.7,2.7);
      \coordinate(a) at ($(x)+(0,0)$);
      \coordinate(b) at ($(x)+(0.5,0)$);
      \coordinate(c) at ($(x)+(1,0)$);
      \coordinate(d) at ($(x)+(1.5,0)$);
      \draw (a)--(b)--(c)--(d);
      \draw[fill=white](a)circle(0.4mm);
      \draw[fill=white](b)circle(0.4mm);
      \draw[fill=white](c)circle(0.4mm);
      \draw[fill=white](d)circle(0.4mm);

      \draw ($(x)+(0.75,0)+(-0.3,-0.3)$)--($(x)+(0.75,0)+(-2.1,-2.9)$);

      \draw[blue]($(b)+(0,-0.1)$)--($(a)+(0,0.1)$);
      \draw[blue]($(a)+(0,0.1)$)arc(90:270:0.85mm)arc(270:450:0.7mm);
      \draw[blue]($(b)+(0,-0.1)$)arc(-90:90:0.85mm);
   
      \draw[blue]($(c)+(0,-0.1)$)--($(b)+(0,0.1)$);
      \draw[blue]($(b)+(0,0.1)$)arc(90:270:0.85mm);
      \draw[blue]($(c)+(0,-0.1)$)arc(-90:90:0.85mm)arc(90:270:0.7mm);
   
      \draw[blue] ($(b)+(0.06,0.06)$)..controls($(c)+(0,-0.24)$)..($(d)+(-0.04,0.08)$);
      \draw[blue]($(d)+(-0.04,0.08)$)arc(120:-90:0.85mm)arc(-90:-270:0.7mm);
      %%%%%%%%%%%%%%%%%%%%%%%%%%%%%%%%%%%%%%%%%%%
      \coordinate (x) at(-2.8,0.75);
      \coordinate(a) at ($(x)+(0,0)$);
      \coordinate(b) at ($(x)+(0.5,0)$);
      \coordinate(c) at ($(x)+(1,0)$);
      \coordinate(d) at ($(x)+(1.5,0)$);
      \draw (a)--(b)--(c)--(d);
      \draw[fill=white](a)circle(0.4mm);
      \draw[fill=white](b)circle(0.4mm);
      \draw[fill=white](c)circle(0.4mm);
      \draw[fill=white](d)circle(0.4mm);

      \draw ($(x)+(0.75,0)+(0,-0.3)$)--($(x)+(0.75,0)+(0,-1.2)$);
   
      \draw ($(x)+(0.75,0)+(0,0.3)$)--($(x)+(0.75,0)+(0.2,1.7)$);
   
      \draw ($(x)+(0.75,0)+(0.3,-0.3)$)--($(x)+(0.75,0)+(1.6,-2.5)$);

      %\draw[blue]($(b)+(0,-0.1)$)--($(a)+(0,0.1)$);
      \draw[blue]($(a)+(0.03,-0.08)$)arc(-70:-270:0.85mm)arc(-270:-450:0.7mm);
      %\draw[blue]($(b)+(0,-0.1)$)arc(-90:90:0.85mm);
      \draw[blue] ($(a)+(0.07,-0.05)$)..controls($(b)+(0,0.24)$)..($(c)+(-0.06,-0.08)$);
      \draw[blue] ($(a)+(0.03,-0.08)$)--($(b)+(-0.04,0.085)$);
      \draw[blue] ($(b)+(-0.04,0.085)$)arc(120:-90:0.85mm)arc(-90:-270:0.7mm);

      %\draw[blue]($(c)+(0,-0.1)$)--($(b)+(0,0.1)$);
      %\draw[blue]($(b)+(0,0.1)$)arc(90:270:0.85mm)arc(270:450:0.7mm);
      \draw[blue]($(c)+(-0.06,-0.08)$)arc(-120:90:0.85mm);
   
      \draw[blue]($(d)+(0,-0.1)$)--($(c)+(0,0.1)$);
      \draw[blue]($(c)+(0,0.1)$)arc(90:270:0.85mm);
      \draw[blue]($(d)+(0,-0.1)$)arc(-90:90:0.85mm)arc(90:270:0.7mm);
      %%%%%%%%%%%%%%%%%%%%%%%%%%%%%%%%%%%%%%%%%%%
   
      %%%%%%%%%%%%%%%%%%%%%%%%%%%%%%%%%%%%%%%%%%%
      \coordinate (x) at(-2.7,2.7);
      \coordinate(a) at ($(x)+(0,0)$);
      \coordinate(b) at ($(x)+(0.5,0)$);
      \coordinate(c) at ($(x)+(1,0)$);
      \coordinate(d) at ($(x)+(1.5,0)$);
      \draw (a)--(b)--(c)--(d);
      \draw[fill=white](a)circle(0.4mm);
      \draw[fill=white](b)circle(0.4mm);
      \draw[fill=white](c)circle(0.4mm);
      \draw[fill=white](d)circle(0.4mm);

      \draw ($(x)+(0.75,0)+(0.3,-0.3)$)--($(x)+(0.75,0)+(2.1,-2.9)$);

      %\draw[blue]($(b)+(0,-0.1)$)--($(a)+(0,0.1)$);
      \draw[blue]($(a)+(0.03,-0.07)$)arc(-70:-270:0.85mm)arc(-270:-450:0.7mm);
      %\draw[blue]($(b)+(0,-0.1)$)arc(-90:90:0.85mm);
      \draw[blue] ($(a)+(0.03,-0.07)$)..controls($(b)+(0,0.24)$)..($(c)+(-0.06,-0.06)$);
   
      \draw[blue]($(c)+(0,-0.1)$)--($(b)+(0,0.1)$);
      \draw[blue]($(b)+(0,0.1)$)arc(90:270:0.85mm)arc(270:450:0.7mm);
      \draw[blue]($(c)+(0,-0.1)$)arc(-90:90:0.85mm);
   
      \draw[blue]($(d)+(0,-0.1)$)--($(c)+(0,0.1)$);
      \draw[blue]($(c)+(0,0.1)$)arc(90:270:0.85mm);
      \draw[blue]($(d)+(0,-0.1)$)arc(-90:90:0.85mm)arc(90:270:0.7mm);
      %%%%%%%%%%%%%%%%%%%%%%%%%%%%%%%%%%%%%%%%%%%
      %%%%%%%%%%%%%%%%%%%%%%%%%%%%%%%%%%%%%%%%%%%
      %%%%%%%%%%%%%%%%%%%%%%%%%%%%%%%%%%%%%%%%%%
         \coordinate (x) at(0,-0.4);
         \coordinate(a) at ($(x)+(0,0)$);
         \coordinate(b) at ($(x)+(0.5,0)$);
         \coordinate(c) at ($(x)+(1,0)$);
         \coordinate(d) at ($(x)+(1.5,0)$);
         \draw (a)--(b)--(c)--(d);
         \draw[fill=white](a)circle(0.4mm);
         \draw[fill=white](b)circle(0.4mm);
         \draw[fill=white](c)circle(0.4mm);
         \draw[fill=white](d)circle(0.4mm);

      \draw[blue] ($(d)+(-0.06,0.07)$)..controls($(c)+(0,-0.28)$)..($(b)+(0.06,0.06)$);
      \draw[blue] ($(d)+(-0.06,0.07)$)arc(120:-90:0.85mm)arc(-90:-270:0.7mm);

      \draw[blue] ($(a)+(0.03,-0.08)$)..controls($(b)+(0,0.28)$)..($(c)+(-0.06,-0.06)$);
      \draw[blue] ($(a)+(0.03,-0.08)$)arc(-80:-270:0.85mm)arc(-270:-450:0.7mm);

      \draw[blue]($(b)+(0,0.1)$)--($(c)+(0,-0.1)$);
      \draw[blue]($(c)+(0,-0.1)$)arc(-90:90:0.85mm)arc(90:270:0.7mm);
      \draw[blue]($(b)+(0,0.1)$)arc(90:270:0.85mm)arc(270:450:0.7mm);

      \draw ($(x)+(0.75,0)+(0,-0.3)$)--($(x)+(0.75,0)+(0,-0.7)$);

         %%%%%%%%%%%%%%%%%%%%%%%%%%%%%%%%%%%%%%%%%%%
         %%%%%%%%%%%%%%%%%%%%%%%%%%%%%%%%%%%%%%%%%%
         \coordinate (x) at(0,-1.3);
         \coordinate(a) at ($(x)+(0,0)$);
         \coordinate(b) at ($(x)+(0.5,0)$);
         \coordinate(c) at ($(x)+(1,0)$);
         \coordinate(d) at ($(x)+(1.5,0)$);
         \draw (a)--(b)--(c)--(d);
         \draw[fill=white](a)circle(0.4mm);
         \draw[fill=white](b)circle(0.4mm);
         \draw[fill=white](c)circle(0.4mm);
         \draw[fill=white](d)circle(0.4mm);
   
         \draw ($(x)+(0.25,-0.2)$)--($(x)+(-0.3,-0.45)$);
         \draw ($(x)+(1.4,-0.2)$)--($(x)+(1.8,-0.45)$);

         \draw[blue] ($(d)+(0,0.1)$)arc(90:-90:0.85mm)arc(-90:-270:0.7mm);

         \draw[blue] ($(d)+(-0.07,0.03)$)..controls($(c)+(0,-0.28)$)..($(b)+(0.04,0.08)$)arc(-120-180:90-180:0.85mm)arc(90-180:270-180:0.7mm);
      
         \draw[blue] ($(a)+(0.07,-0.03)$)..controls($(b)+(0,0.28)$)..($(c)+(-0.04,-0.08)$)arc(-120:90:0.85mm)arc(90:270:0.7mm);
      
         \draw[blue] ($(d)+(0,0.1)$)..controls($(c)+(0,-0.22)$)..($(b)!0.5!(c)$);
         \draw[blue] ($(a)+(0,-0.1)$)..controls($(b)+(0,0.22)$)..($(b)!0.5!(c)$);
         \draw[blue] ($(a)+(0,-0.1)$)arc(90-180:-90-180:0.85mm)arc(-90-180:-270-180:0.7mm);

         %%%%%%%%%%%%%%%%%%%%%%%%%%%%%%%%%%%%%%%%%%%
         %%%%%%%%%%%%%%%%%%%%%%%%%%%%%%%%%%%%%%%%%%
         \coordinate (x) at(1.1,-2);
         \coordinate(a) at ($(x)+(0,0)$);
         \coordinate(b) at ($(x)+(0.5,0)$);
         \coordinate(c) at ($(x)+(1,0)$);
         \coordinate(d) at ($(x)+(1.5,0)$);
         \draw (a)--(b)--(c)--(d);
         \draw[fill=white](a)circle(0.4mm);
         \draw[fill=white](b)circle(0.4mm);
         \draw[fill=white](c)circle(0.4mm);
         \draw[fill=white](d)circle(0.4mm);

         \draw[blue]($(c)+(0,-0.1)$)--($(d)+(0,0.1)$);
         \draw[blue]($(d)+(0,0.1)$)arc(-90-180:-270-180:0.85mm)arc(-270-180:-450-180:0.7mm);
         \draw[blue]($(c)+(0,-0.1)$)arc(90-180:-90-180:0.85mm)arc(-90-180:-270-180:0.7mm);
      
         \draw[blue] ($(d)+(-0.07,0.05)$)..controls($(c)+(0,-0.22)$)..($(b)!0.5!(c)$);
         \draw[blue] ($(a)+(0,-0.1)$)..controls($(b)+(0,0.22)$)..($(b)!0.5!(c)$);
         \draw[blue] ($(a)+(0,-0.1)$)arc(90-180:-90-180:0.85mm)arc(-90-180:-270-180:0.7mm);
      
         \draw[blue] ($(d)+(-0.07,0.03)$)..controls($(c)+(0,-0.28)$)..($(b)+(0.04,0.08)$)arc(-120-180:90-180:0.85mm)arc(90-180:270-180:0.7mm);

         %%%%%%%%%%%%%%%%%%%%%%%%%%%%%%%%%%%%%%%%%%%
         %%%%%%%%%%%%%%%%%%%%%%%%%%%%%%%%%%%%%%%%%%
         \coordinate (x) at(-1.1,-2);
         \coordinate(a) at ($(x)+(0,0)$);
         \coordinate(b) at ($(x)+(0.5,0)$);
         \coordinate(c) at ($(x)+(1,0)$);
         \coordinate(d) at ($(x)+(1.5,0)$);
         \draw (a)--(b)--(c)--(d);
         \draw[fill=white](a)circle(0.4mm);
         \draw[fill=white](b)circle(0.4mm);
         \draw[fill=white](c)circle(0.4mm);
         \draw[fill=white](d)circle(0.4mm);

         \draw[blue]($(b)+(0,0.1)$)--($(a)+(0,-0.1)$);
         \draw[blue]($(a)+(0,-0.1)$)arc(-90:-270:0.85mm)arc(-270:-450:0.7mm);
         \draw[blue]($(b)+(0,0.1)$)arc(90:-90:0.85mm)arc(-90:-270:0.7mm);
      
         \draw[blue] ($(a)+(-0.07,-0.05)$)..controls($(b)+(0,0.22)$)..($(b)!0.5!(c)$);
         \draw[blue] ($(d)+(0,0.1)$)..controls($(c)+(0,-0.22)$)..($(b)!0.5!(c)$);
         \draw[blue] ($(d)+(0,0.1)$)arc(90:-90:0.85mm)arc(-90:-270:0.7mm);
      
         \draw[blue] ($(a)+(0.07,-0.03)$)..controls($(b)+(0,0.28)$)..($(c)+(-0.04,-0.08)$)arc(-120:90:0.85mm)arc(90:270:0.7mm);

         %%%%%%%%%%%%%%%%%%%%%%%%%%%%%%%%%%%%%%%%%%%
         \coordinate (x) at(0,-2.7);
         \coordinate(a) at ($(x)+(0,0)$);
         \coordinate(b) at ($(x)+(0.5,0)$);
         \coordinate(c) at ($(x)+(1,0)$);
         \coordinate(d) at ($(x)+(1.5,0)$);
         \draw (a)--(b)--(c)--(d);
         \draw[fill=white](a)circle(0.4mm);
         \draw[fill=white](b)circle(0.4mm);
         \draw[fill=white](c)circle(0.4mm);
         \draw[fill=white](d)circle(0.4mm);
      
         \draw[blue]($(b)+(0,0.1)$)--($(a)+(0,-0.1)$);
         \draw[blue]($(a)+(0,-0.1)$)arc(-90:-270:0.85mm)arc(-270:-450:0.7mm);
         \draw[blue]($(b)+(0,0.1)$)arc(90:-90:0.85mm)arc(-90:-270:0.7mm);
      
         \draw[blue] ($(a)+(0.07,-0.05)$)..controls($(b)+(0,0.22)$)..($(b)!0.5!(c)$);
         \draw[blue] ($(d)+(-0.07,0.05)$)..controls($(c)+(0,-0.22)$)..($(b)!0.5!(c)$);
      
         \draw[blue]($(d)+(0,0.1)$)--($(c)+(0,-0.1)$);
         \draw[blue]($(c)+(0,-0.1)$)arc(-90:-270:0.85mm)arc(-270:-450:0.7mm);
         \draw[blue]($(d)+(0,0.1)$)arc(90:-90:0.85mm)arc(-90:-270:0.7mm);
   
      \draw ($(x)+(0.25,0.2)$)--($(x)+(-0.3,0.55)$);
      \draw ($(x)+(1.2,0.2)$)--($(x)+(1.7,0.55)$);
   
      \draw ($(x)+(0.75,0)+(0,-0.25)$)--($(x)+(0.75,0)+(0,-0.7)$);
      
         %%%%%%%%%%%%%%%%%%%%%%%%%%%%%%%%%%%%%%%%%%%
         \coordinate (x) at(0,-3.7);
         \coordinate(a) at ($(x)+(0,0)$);
         \coordinate(b) at ($(x)+(0.5,0)$);
         \coordinate(c) at ($(x)+(1,0)$);
         \coordinate(d) at ($(x)+(1.5,0)$);
         \draw (a)--(b)--(c)--(d);
         \draw[fill=white](a)circle(0.4mm);
         \draw[fill=white](b)circle(0.4mm);
         \draw[fill=white](c)circle(0.4mm);
         \draw[fill=white](d)circle(0.4mm);

      \draw ($(x)+(0.75,0)+(-0.4,0.3)$)--($(x)+(0.75,0)+(-2,0.8)$);
      \draw ($(x)+(0.75,0)+(0.4,0.3)$)--($(x)+(0.75,0)+(2,0.8)$);

         \draw[blue]($(b)+(0,0.1)$)--($(a)+(0,-0.1)$);
         \draw[blue]($(a)+(0,-0.1)$)arc(-90:-270:0.85mm)arc(-270:-450:0.7mm);
         \draw[blue]($(b)+(0,0.1)$)arc(90:-90:0.85mm);
      
         \draw[blue]($(c)+(0,0.1)$)--($(b)+(0,-0.1)$);
         \draw[blue]($(b)+(0,-0.1)$)arc(-90:-270:0.85mm);
         \draw[blue]($(c)+(0,0.1)$)arc(90:-90:0.85mm);
      
         \draw[blue]($(d)+(0,0.1)$)--($(c)+(0,-0.1)$);
         \draw[blue]($(c)+(0,-0.1)$)arc(-90:-270:0.85mm);
         \draw[blue]($(d)+(0,0.1)$)arc(90:-90:0.85mm)arc(-90:-270:0.7mm);
         %%%%%%%%%%%%%%%%%%%%%%%%%%%%%%%%%%%%%%%%%%%
      
         \coordinate (x) at(-2.8,-0.75);
         \coordinate(a) at ($(x)+(0,0)$);
         \coordinate(b) at ($(x)+(0.5,0)$);
         \coordinate(c) at ($(x)+(1,0)$);
         \coordinate(d) at ($(x)+(1.5,0)$);
         \draw (a)--(b)--(c)--(d);
         \draw[fill=white](a)circle(0.4mm);
         \draw[fill=white](b)circle(0.4mm);
         \draw[fill=white](c)circle(0.4mm);
         \draw[fill=white](d)circle(0.4mm);

      \draw ($(x)+(0.75,0)+(0,-0.3)$)--($(x)+(0.75,0)+(0.2,-1.7)$);
   
      \draw ($(x)+(0.75,0)+(0.3,0.3)$)--($(x)+(0.75,0)+(1.6,2.5)$);

         \draw[blue]($(b)+(0,0.1)$)--($(a)+(0,-0.1)$);
         \draw[blue]($(a)+(0,-0.1)$)arc(-90:-270:0.85mm)arc(-270:-450:0.7mm);
         \draw[blue]($(b)+(0,0.1)$)arc(90:-90:0.85mm);
      
         %\draw[blue]($(c)+(0,-0.1)$)--($(b)+(0,0.1)$);
         \draw[blue]($(b)+(0.06,-0.08)$)arc(-60:-270:0.85mm);
         %\draw[blue]($(c)+(0,-0.1)$)arc(-90:90:0.85mm)arc(90:270:0.7mm);
      
         \draw[blue] ($(b)+(0.06,-0.08)$)..controls($(c)+(0,0.24)$)..($(d)+(-0.04,-0.08)$);
         \draw[blue]($(d)+(-0.04,-0.08)$)arc(-120:90:0.7mm);
         \draw[blue] ($(d)+(-0,0.1)$)arc(90:270:0.8mm);
         \draw[blue] ($(d)+(-0,0.1)$)arc(90:-90:1.1mm);
         \draw[blue] ($(c)+(-0,0.07)$)arc(90:-90:0.7mm)arc(-90:-270:0.85mm)--($(d)+(-0,-0.12)$);

         %%%%%%%%%%%%%%%%%%%%%%%%%%%%%%%%%%%%%%%%%%%
         
         %%%%%%%%%%%%%%%%%%%%%%%%%%%%%%%%%%%%%%%%%%%
         \coordinate (x) at(-2.7,-2.7);
         \coordinate(a) at ($(x)+(0,0)$);
         \coordinate(b) at ($(x)+(0.5,0)$);
         \coordinate(c) at ($(x)+(1,0)$);
         \coordinate(d) at ($(x)+(1.5,0)$);
         \draw (a)--(b)--(c)--(d);
         \draw[fill=white](a)circle(0.4mm);
         \draw[fill=white](b)circle(0.4mm);
         \draw[fill=white](c)circle(0.4mm);
         \draw[fill=white](d)circle(0.4mm);

      \draw ($(x)+(0.75,0)+(0.3,0.3)$)--($(x)+(0.75,0)+(2.1,2.9)$);

         \draw[blue]($(b)+(0,0.1)$)--($(a)+(0,-0.1)$);
         \draw[blue]($(a)+(0,-0.1)$)arc(-90:-270:0.85mm)arc(-270:-450:0.7mm);
         \draw[blue]($(b)+(0,0.1)$)arc(90:-90:0.85mm);
      
         \draw[blue]($(c)+(0,0.1)$)--($(b)+(0,-0.1)$);
         \draw[blue]($(b)+(0,-0.1)$)arc(-90:-270:0.85mm);
         \draw[blue]($(c)+(0,0.1)$)arc(90:-90:0.85mm)arc(-90:-270:0.7mm);
      
         \draw[blue] ($(b)+(0.06,-0.06)$)..controls($(c)+(0,0.24)$)..($(d)+(-0.04,-0.08)$);
         \draw[blue]($(d)+(-0.04,-0.08)$)arc(-120:90:0.85mm)arc(90:270:0.7mm);
         %%%%%%%%%%%%%%%%%%%%%%%%%%%%%%%%%%%%%%%%%%%
         \coordinate (x) at(2.8,-0.75);
         \coordinate(a) at ($(x)+(0,0)$);
         \coordinate(b) at ($(x)+(0.5,0)$);
         \coordinate(c) at ($(x)+(1,0)$);
         \coordinate(d) at ($(x)+(1.5,0)$);
         \draw (a)--(b)--(c)--(d);
         \draw[fill=white](a)circle(0.4mm);
         \draw[fill=white](b)circle(0.4mm);
         \draw[fill=white](c)circle(0.4mm);
         \draw[fill=white](d)circle(0.4mm);

      \draw ($(x)+(0.75,0)+(0,-0.3)$)--($(x)+(0.75,0)+(-0.2,-1.7)$);
   
      \draw ($(x)+(0.75,0)+(-0.3,0.3)$)--($(x)+(0.75,0)+(-1.6,2.5)$);

         \draw[blue]($(c)+(0,0.1)$)--($(d)+(0,-0.1)$);
         \draw[blue]($(d)+(0,-0.1)$)arc(-90:90:0.85mm)arc(90:270:0.7mm);
         \draw[blue]($(c)+(0,0.1)$)arc(90:270:0.85mm);
      
         %\draw[blue]($(c)+(0,-0.1)$)--($(b)+(0,0.1)$);
         \draw[blue]($(c)+(-0.06,-0.08)$)arc(-120:90:0.85mm);
         %\draw[blue]($(c)+(0,-0.1)$)arc(-90:90:0.85mm)arc(90:270:0.7mm);
      
         \draw[blue] ($(c)+(-0.06,-0.08)$)..controls($(b)+(0,0.24)$)..($(a)+(0.04,-0.08)$);
         \draw[blue]($(a)+(0.04,-0.08)$)arc(-70:-270:0.7mm);
         \draw[blue] ($(a)+(-0,0.1)$)arc(90:-90:0.8mm);
         \draw[blue] ($(a)+(-0,0.1)$)arc(90:270:1.1mm);
         \draw[blue] ($(b)+(-0,0.07)$)arc(90:270:0.7mm)arc(270:450:0.85mm)--($(a)+(-0,-0.12)$);
         %%%%%%%%%%%%%%%%%%%%%%%%%%%%%%%%%%%%%%%%%%%
      
         %%%%%%%%%%%%%%%%%%%%%%%%%%%%%%%%%%%%%%%%%%%
         \coordinate (x) at(2.7,-2.7);
         \coordinate(a) at ($(x)+(0,0)$);
         \coordinate(b) at ($(x)+(0.5,0)$);
         \coordinate(c) at ($(x)+(1,0)$);
         \coordinate(d) at ($(x)+(1.5,0)$);
         \draw (a)--(b)--(c)--(d);
         \draw[fill=white](a)circle(0.4mm);
         \draw[fill=white](b)circle(0.4mm);
         \draw[fill=white](c)circle(0.4mm);
         \draw[fill=white](d)circle(0.4mm);

         \draw ($(x)+(0.75,0)+(-0.3,0.3)$)--($(x)+(0.75,0)+(-2.1,2.9)$); 
         
         %\draw[blue]($(b)+(0,-0.1)$)--($(a)+(0,0.1)$);
         \draw[blue]($(a)+(0.03,0.07)$)arc(70:270:0.85mm)arc(270:450:0.7mm);
         %\draw[blue]($(b)+(0,-0.1)$)arc(-90:90:0.85mm);
         \draw[blue] ($(a)+(0.03,0.07)$)..controls($(b)+(0,-0.24)$)..($(c)+(-0.06,0.06)$);
      
         \draw[blue]($(c)+(0,0.1)$)--($(b)+(0,-0.1)$);
         \draw[blue]($(b)+(0,-0.1)$)arc(-90:-270:0.85mm)arc(-270:-450:0.7mm);
         \draw[blue]($(c)+(0,0.1)$)arc(90:-90:0.85mm);
      
         \draw[blue]($(d)+(0,0.1)$)--($(c)+(0,-0.1)$);
         \draw[blue]($(c)+(0,-0.1)$)arc(-90:-270:0.85mm);
         \draw[blue]($(d)+(0,0.1)$)arc(90:-90:0.85mm)arc(-90:-270:0.7mm);
         %%%%%%%%%%%%%%%%%%%%%%%%%%%%%%%%%%%%%%%%%%%
   \end{tikzpicture}
   \end{tabular}
   }
\caption{All $2$-tilting complexes (in the left) and their geometric description (in the right) over the Brauer tree algebra of the line having $3$ edges. There are precisely $\binom{6}{3}=20$ elements.}
\label{fig:intro}
\end{figure}

To prove Theorems \ref{theorem main} and \ref{theorem sub}, we use a combinatorial description of $2$-tilting complexes over a Brauer graph algebra due to \cite{AAC}, and its geometric interpretation in terms of the oriented marked surface (Theorem \ref{thm:AAC}). 
Our main results are Theorems \ref{thm:main} and \ref{thm:gluing}; which correspond to Theorems \ref{theorem main} and \ref{theorem sub} in our geometric model. 

Finally, we give an application to a class of special biserial algebras. 
To a given plane tree $\mathbf{G}$, we associate a complete gentle algebra $A_{\mathbf{G}}$ (Definition \ref{def:BGA}) that is infinite dimensional. More precisely, it is a module-finite algebra over the formal power series ring $k[[t]]$ in one valuable over a field $k$. 
In particular, we can apply $2$-silting theory developed in \cite{Kimura}.  
For the ideal $I_{\rm sc}$ generated by all special cycles in $A_{\mathbf{G}}$ (see Definition \ref{def:QG}), we prove that canonical surjections 
\[
   A_{\mathbf{G}} \twoheadrightarrow B_{\mathbf{G}} \twoheadrightarrow A_{\mathbf{G}}/I_{\rm sc}
\]
induce isomorphisms 
\[
   \twosilt A_{\mathbf{G}} \cong \twotilt B_{\mathbf{G}} \cong \twosilt (A_{\mathbf{G}}/I_{\rm sc}),
\]
where we denote by $\twosilt A_{\mathbf{G}}$ the set of isomorphism classes of basic $2$-silting complexes of $A_{\mathbf{G}}$. More generally, we prove the following. 

\begin{corollary} (Theorem \ref{theorem main} and Corollary \ref{cor:gentle})
   Let $J$ be an ideal contained in $I_{\rm sc}$. 
   Then there is a bijection between $\twosilt A_{\mathbf{G}}$ and $\twosilt (A_{\mathbf{G}}/J)$, and therefore their cardinality are 
   \[
   \#\twosilt A_{\mathbf{G}} = \#\twosilt (A_{\mathbf{G}}/ J) =\binom{2|\mathbf{G}|}{|\mathbf{G}|}. 
   \]
   In particular, this number depends only on the number of edges of $\mathbf{G}$.
\end{corollary}

Note that the equality in Theorem \ref{theorem main}(2) is known for particular cases: for Brauer star algebras \cite{Adachi16b} and for Brauer line algebras \cite{Aoki18}. 
After the author obtained the main result (Theorem \ref{theorem main}) of this paper, Asashiba-Mizuno-Nakashima \cite{AMN} prove the same equality $\#\twotilt B_{\mathbf{G}}=\binom{2|\mathbf{G}|}{|\mathbf{G}|}$ in (2) by a completely different way. 

This paper is organized as follows. 
In Section \ref{sec:preliminaries}, we define Brauer graph algebras and complete gentle algebras from ribbon graphs. We also discuss their two-term silting/tilting theory. 
In Section \ref{sec:AAC}, we give a description of all two-term tilting complexes over Brauer tree algebras in terms of arcs on the associated plane tree canonically embedded into the sphere. This is a geometric interpretation of \cite[Theorem 4.6]{AAC}, see Section \ref{sec:sw}. 
In Section \ref{sec:glue}, we establish a method of gluing collections of arcs, which is central in this paper. It enables us to enumerate all complete collections of arcs on the initial plane tree from those of its subtrees (Theorem \ref{theorem sub}). Our construction can be parametrized by combinatorial objects called lattice paths. 
In the end of this paper (Section \ref{sec:proof}), we give a proof of Theorem \ref{thm:main} (hence Theorem \ref{theorem main}) by using results established in the previous section.

\section{Preliminaries} \label{sec:preliminaries} 
To each ribbon graph, we associate two kinds of algebras, Brauer graph algebras and complete gentle algebras. We study their two-term silting/tilting complexes.

\subsection{Ribbon graphs}
We begin with basic terminology of graph theory.

\begin{definition}
   A \textit{finite graph} is a triple $G=(G_0,G_1,s)$, where $G_0, G_1$ are finite non-empty sets and $s$ is a map $s\colon G_1 \to \{\{u,v\}\mid u,v\in G_0\}$. 
   \begin{itemize}
      \item An element of $G_0$ is called a \textit{vertex} of $G$. 
      \item An element of $G_1$ is called an \textit{edge} of $G$. 
      \item For an edge $e\in G_1$, each element $v$ of $s(e)$ is called an \textit{endpoint} of $e$. 
   \end{itemize}
\end{definition}

\begin{definition} \label{d_graph-theory}
   Let $G=(G_0,G_1,s)$ be a finite graph. 
   \begin{enumerate}
      \item A \textit{walk} on $G$ is a sequence $w=(e_1,\ldots,e_k)$ of pairwise distinct edges for which there is a sequence of vertices $(v_1,\ldots, v_{k+1})$ such that $s(e_i)=\{v_{i},v_{i+1}\}$ for $i\in \{1,\ldots,k\}$. In this case, $v_1$ and $v_{k+1}$ are called \textit{endpoints} of $w$. 
      \item A walk $w$ is called a \textit{cycle} if its endpoints are the same. 
      \item A \textit{signed walk} is a walk $w=(e_1,\ldots,e_k)$ equipped with assignment of signs $\epsilon(e)\in \{\pm1\}$ with $\epsilon(e_i)\neq \epsilon(e_{i+1})$ for all $i\in \{1,\ldots,k-1\}$.
      \item A graph $G$ is called a \textit{connected graph} if every pair of vertices of $G$ is connected, that is, there exists a walk having them as its endpoints. 
      \item A graph $G$ is called \textit{tree} if it has no cycle and is connected.  
   \end{enumerate}
\end{definition}

\begin{definition}
   A \textit{ribbon graph} is a finite connected graph $G=(G_0,G_1,s)$ equipped with the following data: 
   \begin{itemize}
      \item A collection $\sigma=(\sigma_v)_{v\in G_0}$, where $\sigma_v$ is a cyclic permutation of the edges incident to a vertex $v$.
   \end{itemize} 
   We write $\mathbf{G}=(G_0,G_1,s,\sigma)$ for this ribbon graph. 
\end{definition}

A ribbon graph equipped with an assignment of positive integers $\mathfrak{m}_v$ for each vertex $v$ is called \textit{Brauer graph}. We regard each ribbon graph as a Brauer graph with $\mathfrak{m}_v=1$ for all vertices $v$. 

\subsection{Algebras associated to ribbon graphs}
Brauer graph algebras and complete gentle algebras are defined by quivers with relations. Let $k$ be an algebraically closed field.

\begin{definition} \label{def:QG}
For a ribbon graph $\mathbf{G}=(G_0,G_1,s,\sigma)$, we define a finite quiver $Q_{\mathbf{G}}:=(Q_0,Q_1)$ as follows:
\begin{itemize}
   \item The set $Q_0$ of vertices bijectively corresponds to $G_1$. 
   \item The set $Q_1$ of arrows is given by $\{(e|\sigma_v(e))\mid e\in G_1,v\in s(e)\}$, that is, we draw an arrow from $e$ to $f$ for $e,f\in G_1$ whenever there is a vertex $v\in G_0$ with $f = \sigma_v(e)$.
\end{itemize}
For a vertex $v\in G_0$ and an edge $e\in G_1$ with $v\in s(e)$, 
the \textit{special cycle} $C_{v,e}$ is defined to be a cycle 
   \[
      (e|\sigma_v(e))(\sigma_v(e)|\sigma^2_v(e)) \cdots (\sigma^{m-1}_v(e)|e)
   \] 
in $Q_{\mathbf{G}}$ with no repetitions.
\end{definition}

\begin{definition} \label{def:BGA}
For a ribbon graph ${\mathbf{G}}=(G_0,G_1,s,\sigma)$, let $kQ_{\mathbf{G}}$ be the path algebra of $Q_{\mathbf{G}}$ and $\widehat{kQ_{\mathbf{G}}}$ the completed path algebra of $kQ_{\mathbf{G}}$.
Let $I_{\mathbf{G}}$ (resp, $J_{\mathbf{G}}$) be an ideal in $\widehat{kQ_{\mathbf{G}}}$ generated by all paths of the form 
\[ 
   (\sigma^{-1}_u(e)|e)(e|\sigma_v(e)) \quad \text{(resp., }  C_{u,e} - C_{v,e})
\] 
   for $e\in G_1$ and $u,v\in s(e)$.
We call $A_{\mathbf{G}}: = \widehat{kQ_{\mathbf{G}}}/I_{\mathbf{G}}$ the \textit{complete gentle algebra} and $B_{\mathbf{G}} := \widehat{kQ_{\mathbf{G}}}/(I_{\mathbf{G}}+J_{\mathbf{G}})$ the \textit{Brauer graph algebra} associated to $\mathbf{G}$. 
\end{definition}

Let $k[[t]]$ be the formal power series ring in the valuable $t$ over $k$. 

\begin{proposition} \label{prop:alg}
For a ribbon graph $\mathbf{G}$, the following hold. 
\begin{enumerate}
   \item \cite[Proposition 7.13]{AY20} $A_{\mathbf{G}}$ is a module-finite $k[[t]]$-algebra, that is, finitely generated as a $k[[t]]$-module. 
   \item \cite{Schroll18} $B_{\mathbf{G}}$ is a finite dimensional symmetric special biserial $k$-algebra.  
\end{enumerate}
\end{proposition}

\begin{example} \label{ex:tree}
   Let $\mathbf{G}=(G_0,G_1,s,\sigma)$ be a ribbon graph given as follows:
   
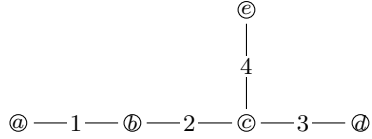
\[
      \begin{tikzpicture}
         \node(a) at (0,0) {\footnotesize$a$}; 
         \node(b) at (1.5,0) {\footnotesize$b$}; 
         \node(c) at (3,0) {\footnotesize$c$}; 
         \node(d) at (4.5,0) {\footnotesize$d$}; 
         \node(e) at (3,1.5) {\footnotesize$e$}; 
         \draw(a) circle(1.1mm); 
         \draw(b) circle(1.1mm);
         \draw(c) circle(1.1mm);
         \draw(d) circle(1.1mm);
         \draw(e) circle(1.1mm);

         \draw (a)--node[fill=white,inner sep=1]{\small$1$}(b)--node[fill=white,inner sep=1]{\small$2$}(c)--node[fill=white,inner sep=1]{\small$3$}(d) (c)--node[fill=white,inner sep=1]{\small$4$}(e);
      \end{tikzpicture}
   \]
   Namely, $G_0:=\{a,b,c,d,e\}$, $G_1:=\{1,2,3,4\}$ and $s(1)=\{a,b\}$, $s(2)=\{b,c\}$, $s(3)=\{c,d\}$, $s(4)=\{c,e\}$. In addition, the cyclic orderings of edges around vertices are given by
   \[
      \sigma_a = (1), \ \sigma_b = (1,2), \ \sigma_c = (2,3,4), \ \sigma_d = (3) \quad \text{and} \quad \sigma_e = (4).
   \]
   The quiver $Q_{\mathbf{G}}$ and ideals $I_{\mathbf{G}}, J_{\mathbf{G}}$ are given by
   
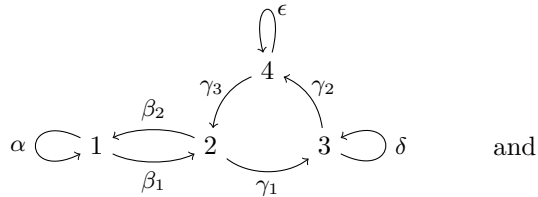
\[
      \begin{tikzpicture}
         \node(1) at (0,0) {$1$}; 
         \node(2) at (1.5,0) {$2$}; 
         \node(3) at (3,0) {$3$}; 
         \node(4) at (2.25,1) {$4$}; 
         \draw[->] (1)..controls(-1,0.5)and(-1,-0.5)..node[left]{\small$\alpha$}(1); 
         \draw[->] (1)..controls(0.5,-0.25)and(1,-0.25)..node[below]{\small$\beta_1$}(2);
         \draw[<-] (1)..controls(0.5,0.25)and(1,0.25)..node[above]{\small$\beta_2$}(2);
         \draw[->] (2)..controls(2,-0.4)and(2.5,-0.4)..node[below]{\small$\gamma_1$}(3);
         \draw[->] (3)..controls(2.9,0.75)and(2.5,0.9)..node[right]{\small$\gamma_2$}(4);
         \draw[->] (4)..controls(2,0.9)and(1.6,0.75)..node[left]{\small$\gamma_3$}(2);
         \draw[->] (3)..controls(4,-0.5)and(4,0.5)..node[right]{\small$\delta$}(3); 
         \draw[->] (4)..controls(2.5,2)and(2,2)..node[right]{\small$\epsilon$}(4); 
         \node at (5.5,0) {\text{and}};
      \end{tikzpicture} 
   \] 
   \begin{eqnarray}\nonumber
      I_{\mathbf{G}} = \langle \alpha\beta_1,\beta_1\gamma_1,\gamma_1\delta,\delta\gamma_2,\gamma_2\epsilon,\epsilon\gamma_3, \gamma_3\beta_2, \beta_2\alpha\rangle, \ 
      J_{\mathbf{G}} = \langle \alpha-\beta_1\beta_2, \beta_2\beta_1- \gamma_1\gamma_2\gamma_3, \gamma_2\gamma_3\gamma_1-\delta, \gamma_3\gamma_1\gamma_2-\epsilon\rangle.
   \end{eqnarray}
\end{example}

\subsection{$2$-siting complexes}
We discuss two-term silting complexes over a module-finite $k[[t]]$-algebra $\Lambda$.
We denote by $\proj \Lambda$ the category of finitely generated projective right $\Lambda$-modules and $\Kb(\proj \Lambda)$ the bounded homotopy category of complexes of $\proj \Lambda$.
Since $k[[t]]$ is a complete local noetherian ring, $\Kb(\proj \Lambda)$ is the Krull-Schmidt triangulated category (\cite[p.132]{CR}).

\begin{definition}
Let $T=(T^i,f^i)$ be a complex in $\Kb(\proj \Lambda)$. 
\begin{enumerate}
   \item $T$ is said to be \textit{presilting} if $\Hom_{\Kb(\proj A)}(T,T[i])=0$ for all integers $i>0$. 
   \item $T$ is said to be \textit{silting} if it is presilting and $\Kb(\proj \Lambda)=\thick T$, where $\thick T$ is the smallest triangulated full subcategory which contains $T$ and is closed under taking direct summands.
   \item $T$ is said to be \textit{tilting} if it is silting and $\Hom_{\Kb(\proj \Lambda)}(T,T[i])=0$ for all integers $i\neq 0$.
   \item $T$ is said to be \textit{two-term} if $T^i=0$ for $i\neq 0,-1$. 
   \item $T$ is said to be \textit{basic} if all indecomposable direct summands of $T$ are pairwise non-isomorphic. 
   \end{enumerate}
We say that $T$ is a \textit{$2$-silting} (resp., \textit{$2$-tilting}) complex if it is two-term and silting (resp., tilting). 
We denote by $\twosilt \Lambda$ (resp., $\twotilt \Lambda$) the set of isomorphism classes of basic $2$-silting (resp., $2$-tilting) complexes for $\Lambda$.
\end{definition}

Let $\Lambda = \bigoplus_{e \in I} P_e$ be a decomposition of $\Lambda$ into indecomposable projective $\Lambda$-modules. 

\begin{definition} \label{def:g-vector}
   For a two-term complex $T=(T^{-1}\to T^0)$ in $\Kb(\proj \Lambda)$, 
   the \textit{$g$-vector} of $T$ is defined to be an integer vector $g(T) = (m_e-m'_e)_{e\in I}$, where $m_e$ (resp., $m'_e$) is the multiplicity of $P_e$ as indecomposable direct summands of $T^0$ (resp., $T^{-1}$). 
\end{definition}

For $e\in I$ and $j\in \mathbb{Z}$, let 
   \[
      \twosilt_e^j \Lambda := \{T \in \twosilt \Lambda \mid g_e(T) =j\}. 
   \]
In addition, let 
\[
   \twosilt_e^{>0} \Lambda := \{T \in \twosilt \Lambda \mid g_e(T) >0\} \quad \text{and} \quad 
   \twosilt_e^{<0} \Lambda := \{T \in \twosilt \Lambda \mid g_e(T) <0\}. 
\]

\begin{proposition} \label{prop:zero}
For any $e\in I$, the set $\twosilt_e^0 \Lambda$ is empty.
\end{proposition}

\begin{proof}
   It follows from \cite[Theorem 2.27]{AI} that the set of indecomposable direct summands of a basic $2$-silting complex forms a basis of the Grothendieck group of the triangulated category $\Kb(\proj \Lambda)$.
\end{proof}

The following are basic properties of $2$-presilting complexes. 

\begin{proposition} \cite[Theorem 5.5]{AIR} \label{prop:g-inv}
If two $2$-presilting complexes $T, U$ satisfy $g(T)=g(U)$, then $T\cong U$.
\end{proposition}

\begin{proposition}\cite[Proposition 2.16]{Aihara13} \label{prop:partialsilt}
For a basic $2$-presilting complex $T$, the following hold.
\begin{enumerate}
   \item $T$ is a direct summand of some basic $2$-silting complex. 
   \item The following conditions are equivalent: 
\begin{enumerate}
   \item[(a)] $T$ is silting. 
   \item[(b)] The number of indecomposable direct summands of $T$ is $\#I$. 
\end{enumerate}
\end{enumerate}
\end{proposition}

Now, let ${\mathbf{G}}$ be a ribbon graph and $A_{\mathbf{G}}=\widehat{kQ_{\mathbf{G}}}/I_{\mathbf{G}}$ the complete gentle algebra of ${\mathbf{G}}$, which is module-finite over $k[[t]]$. 
We have a decomposition $A_{\mathbf{G}} = \bigoplus_{e\in G_1} P_e$, where $P_e$ is an indecomposable projective right $A_{\mathbf{G}}$-module corresponding to $e\in G_1$.

\begin{theorem} \label{thm:reduction}
Let $I_{\rm sc}$ be an ideal generated by all special cycles in $Q_{\mathbf{G}}$.
For any ideal $J$ contained in $I_{\rm sc}$, we have a bijection
\begin{equation}
   \twosilt (A_{\mathbf{G}}) \overset{\sim}{\longrightarrow} \twosilt (A_{\mathbf{G}}/J)
\end{equation} \label{eq:reduced}
that preserves $g$-vectors of complexes.
In particular, it gives a bijection
\begin{equation}
   \twosilt_e^j (A_{\mathbf{G}}) \overset{\sim}{\longrightarrow}\twosilt_e^j (A_{\mathbf{G}}/J) 
\end{equation}
for $e \in G_1$ and $j\in \mathbb{Z}$. 
\end{theorem}

\begin{proof}
   By \cite[Proposition 7.13]{AY20}, $A_{\mathbf{G}}$ is module-finite over $k[[t]]$ with $t\in I_{{\rm sc}}$. 
   Since every special cycle annihilates all bricks over $A_{\mathbf{G}}$, we get the desired bijection (\ref{eq:reduced}) by \cite[Theorem 1.4]{Kimura}, (see also \cite[Corollary 5.20]{DIRRT}). 
   The latter assertion is clear from the former one.
\end{proof}

\begin{corollary} 
   Let $B_{\mathbf{G}}$ be the Brauer graph algebra of $\mathbf{G}$. 
   We have bijections
   \[
      \twosilt (A_{\mathbf{G}}) \overset{\sim}{\longrightarrow} \twotilt (B_{\mathbf{G}}) \quad \text{and} \quad \twosilt_e^j (A_{\mathbf{G}}) \overset{\sim}{\longrightarrow} \twotilt_e^j (B_{\mathbf{G}})
   \]
   for $e\in G_1$ and $j\in \mathbb{Z}$.
\end{corollary}

\begin{proof}
   By \cite[Example 2.8]{AI}, silting complexes coincide with tilting complex for a finite dimensional symmetric algebra. As $J_{\mathbf{G}}\subseteq I_{\rm sc}$, we get the assertion by Theorem \ref{thm:reduction}.
\end{proof}

\section{A geometric model of $2$-silting complexes} \label{sec:AAC}
In \cite{AAC}, they classify all two-term tilting complexes over Brauer graph algebras by using the notion of signed walks on the associated ribbon graph. 
Thanks to Corollary \ref{cor:gentle}, one can apply their result to factor algebras of complete gentle algebras modulo ideals contained in $I_{\rm sc}$ too. 
An aim of this section is to give its geometric description on a plane tree (see \cite{AY20,PPP} for an arbitrary ribbon graphs).

\begin{definition}
A ribbon graph whose underlying graph is a tree is called \textit{plane tree}. 
The Brauer graph algebra corresponding to a plane tree is called \textit{Brauer tree algebra}. 
\end{definition}

Throughout this section, let $\mathbf{G}:=(G_0,G_1,s,\sigma)$ be a plane tree. 

\subsection{$2$-silting complexes via arcs}
For our purpose, we slightly modify notations and definitions in \cite{AAC}. 
We refer to \cite{Labourie13} for a canonical embedding of a ribbon graph into a closed oriented marked surface.

Our plane tree $\mathbf{G}$ is canonically embedded into the sphere $\mathbb{S}:=S^2$.  
More precisely, each of vertices is identified with a distinguished point called \textit{punctures} of $\mathbb{S}$, and each of edges is identified with a non-self-intersecting and pairwise non-intersecting line connecting its endpoints, and the cyclic ordering of the edges incident to each vertex is described as the counterclockwise direction around a puncture. Here, all curves of $\mathbb{S}$ are considered up to isotopy relative to punctures. 
In addition, it determines a ribbon graph $\mathbf{L}:=\mathbf{L}_{\mathbf{G}}$ embedded into a common surface $\mathbb{S}$ such that 
\begin{itemize}
   \item $\mathbf{L}$ has precisely one vertex, which lies on $\mathbb{S}\setminus \mathbf{G}$, and
   \item every edge of $\mathbf{L}$ is a loop, that intersects exactly one edge of $\mathbf{G}$, vise versa. We denote by $\alpha_e$ the loop of $\mathbf{L}$ intersecting with $e\in G_1$. 
\end{itemize}
Notice that the cyclic orientation of $\mathbf{L}$ is determined by that of $\mathbf{G}$. 
We call it the \textit{dual} of $\mathbf{G}$.

We find in Example \ref{ex:collections} an example of a plane tree with its dual. 
By our construction, the closure of each connected component of $\mathbb{S}\setminus \mathbf{L}$ is a polygon having precisely one puncture $v$ of $\mathbf{G}$ in its center. 
Up to cyclic permutation, its sides are labeled as 
\[
   \alpha_e, \alpha_{\sigma_v(e)}, \alpha_{\sigma^2_v(e)},\ldots, \alpha_{\sigma^{m-1}_v(e)}
\]
with no repetition in counterclockwise direction around $v$, where $e,\sigma_v(e),\ldots,\sigma^{m-1}_v(e)$ are edges appearing in the special cycle $C_{v,e}$ in this order. 
We denote this $m$-gon by $\triangle_v$. 

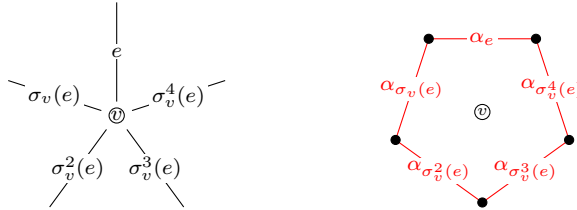
\[
   \begin{tikzpicture}
   \node(0) at (0,0) {\footnotesize$v$}; 
   \draw (0)--node[fill=white,inner sep=1]{\footnotesize$e$}(90:1.5) 
   (0)--node[fill=white,inner sep=1]{\footnotesize$\sigma_v(e)$}(162:1.5) 
   (0)--node[fill=white,inner sep=1]{\footnotesize$\sigma^2_v(e)$}(234:1.5) 
   (0)--node[fill=white,inner sep=1]{\footnotesize$\sigma^3_v(e)$}(306:1.5)  
   (0)--node[fill=white,inner sep=1]{\footnotesize$\sigma^4_v(e)$}(18:1.5);
   \draw(0) circle(1mm);
   \end{tikzpicture}
   \hspace{20mm}
   \begin{tikzpicture}
   \coordinate (q) at (-90:1.2);
   \coordinate (w) at (-162:1.2);
   \coordinate (e) at (-234:1.2);
   \coordinate (r) at (-306:1.2);
   \coordinate (t) at (-18:1.2);

   \draw[red] (q)--node[fill=white,inner sep=1]{\footnotesize$\alpha_{\sigma^2_v(e)}$}(w)--node[fill=white,inner sep=1]{\footnotesize$\alpha_{\sigma_v(e)}$}(e)--node[fill=white,inner sep=1]{\footnotesize$\alpha_{e}$}(r)--node[fill=white,inner sep=1]{\footnotesize$\alpha_{\sigma_v^4(e)}$}(t)--node[fill=white,inner sep=1]{\footnotesize$\alpha_{\sigma^3_v(e)}$}cycle;

   \draw[fill=black] (q)circle(0.6mm); 
   \draw[fill=black] (w)circle(0.6mm); 
   \draw[fill=black] (e)circle(0.6mm); 
   \draw[fill=black] (r)circle(0.6mm); 
   \draw[fill=black] (t)circle(0.6mm); 
   \node(0) at(0,0) {\tiny$v$};
   \draw(0) circle(1mm);

   \end{tikzpicture}
\]

\begin{definition} \label{def:arc}
   An \textit{arc} (more precisely \textit{{\bf G}-arc}) is a non-self-intersecting curve $\gamma$ of $\mathbb{S}$, considered up to isotopy relative to punctures, satisfying all the following conditions:
   \begin{itemize}
   \item $\gamma$ intersects at least one $\alpha_e$ of $\mathbf{L}$. 
   \item Each endpoint of $\gamma$ is a spiral around puncture of $\mathbf{G}$, either clockwise or counterclockwise. 
   \item Whenever $\gamma$ intersects with $\alpha_e$ of $\mathbf{L}$ for $e\in G_1$, the endpoints $u,v$ of $e$ lie on the opposite side of $\gamma$. 
   \end{itemize}
   Here, we consider that the point $v$ lies on the right (resp., left) to $\gamma$ if $\gamma$ circles clockwise (resp., counterclockwise) around $v$ in $\triangle_v$.
\end{definition}

\begin{definition}
   We say that two arcs are \textit{admissible} if they do not intersect. 
   A collection of arcs are said to be  
   \begin{itemize}
      \item \textit{admissible} if it consists of pairwise admissible arcs. 
      \item \textit{reduced} if it is admissible and consists of pairwise distinct arcs. 
      \item \textit{complete} if it is reduced and there are no reduced collection properly containing it.
   \end{itemize}
    We denote by $\mathcal{A}(\mathbf{G})$ the set of all complete collections of $\mathbf{G}$-arcs.
\end{definition}

Let $\gamma$ be a $\mathbf{G}$-arc. Using the notation in Definition \ref{def:arc}, let $p$ be an intersection point of $\gamma$ and $\alpha_e$ such that $\gamma$ leaves $\triangle_u$ to enter $\triangle_v$ via $p$. Then $p$ is said to be \textit{positive} (resp., \textit{negative}) if $u$ is to its right (resp., left), or equivalently, $v$ is to its left (resp., right). 
The \textit{$g$-vector} of $\gamma$ is defined by $g(\gamma)=(g_e(\gamma))_{e\in G_1}$, where 
\begin{eqnarray*}
   g_e(\gamma) := &\#\{\text{positive intersection points of $\gamma$ and $\alpha_e$}\} \\ &- \#\{\text{negative intersection points of $\gamma$ and $\alpha_e$}\}.
\end{eqnarray*}
By definition, if $p,q$ are two intersection points of $\gamma$ and $\alpha_e$, then both are positive or negative. This means that the absolute value $|g_e(\gamma)|$ is precisely the number of intersection points of $\gamma$ and $\alpha_e$.

\begin{figure}[htp]   \centering
   \renewcommand{\arraystretch}{2}
   \setlength{\tabcolsep}{3mm}
   \begin{tabular}{ccccccccc}
      $p$: positive & $p$: negative \\
   \begin{tikzpicture}[scale=0.9]
         \coordinate(u)at(0,0.8); \coordinate(d)at(0,-1);
         \coordinate(llu)at(-1.5,1); \coordinate(lu)at(-0.6,1); \coordinate(rru)at(1.5,1); \coordinate(ru)at(0.6,1);
         \coordinate(lld)at(-1.5,-1.2); \coordinate(ld)at(-0.6,-1.2); \coordinate(rrd)at(1.5,-1.2); \coordinate(rd)at(0.6,-1.2);
         %\draw(llu)--(lu)--(u)--(ru)--(rru) (lld)--(ld)--(d)--(rd)--(rrd) (u)--(d);
         \draw[blue,shift={(0,0.1)}](-1.6,0.8)..controls(-1.1,0.8)and(1.1,-0.8)..(1.6,-0.8);
         %\draw[blue,shift={(0,-0.4)}](-1.6,1.1)..controls(-1.1,1.1)and(1.1,-0.4)..(1.6,-0.4);
         %\draw[blue,shift={(0,-0.75)}](-1.6,1.1)..controls(-1.1,1.1)and(1.1,-0.4)..(1.6,-0.4);

         %\draw[blue,shift={(0,0)}](-1.5,0.5)..controls(-0.5,0.5)and(0.5,-0.5)..(1.5,-0.5);
         %\draw[blue,shift={(0,-0.25)}](-1.5,0.5)..controls(-0.5,0.5)and(0.5,-0.5)..(1.5,-0.5);
         %\fill(lu)circle(1mm); \fill(ru)circle(1mm); \fill(ld)circle(1mm); \fill(rd)circle(1mm);
         %\fill(u)circle(1mm); \fill(d)circle(1mm);
         \draw[fill=white](-1.4,0)circle(1.3mm); \draw[fill=white](1.4,0)circle(1.3mm);
         \node[fill=white, inner sep=0.05] at(-1.4,0){\tiny$u$};
         \node[fill=white, inner sep=0.05] at(1.4,0){\tiny$v$};
         \draw (-1.25,0)--(1.25,0);
         \draw[red] (0,1)--(0,-0.9); 
         \draw[fill=black] (0,0.1)circle(0.6mm);
         \node at (0.2,0.2) {\footnotesize$p$};
         \node[red] at (0,-1.12) {$\alpha_e$};
         %\node at(0.2,0.15){$p$};
         %\node[blue] at(-1.9,1.1) {\footnotesize$\gamma_3^a$};
         %\node[blue] at(-1.9,0.75) {\footnotesize$\gamma_2^a$};
         %\node[blue] at(0,1.5) {$\gamma^a\underset{e}{\times} \gamma^b$};
         %\node[blue] at(-1.9,0.9) {\footnotesize$\gamma^a$};
         %\node[blue] at(1.9,-0.7) {\footnotesize$\gamma^b$};
         %\node[blue] at(1.9,-0.7) {\footnotesize$\gamma_2^b$};
         %\node[blue] at(1.9,-1.1) {\footnotesize$\gamma_1^b$};
   \end{tikzpicture}
   &
   \begin{tikzpicture}[scale=0.9]
      \coordinate(u)at(0,0.8); \coordinate(d)at(0,-1);
      \coordinate(llu)at(-1.5,1); \coordinate(lu)at(-0.6,1); \coordinate(rru)at(1.5,1); \coordinate(ru)at(0.6,1);
      \coordinate(lld)at(-1.5,-1.2); \coordinate(ld)at(-0.6,-1.2); \coordinate(rrd)at(1.5,-1.2); \coordinate(rd)at(0.6,-1.2);
      %\draw(llu)--(lu)--(u)--(ru)--(rru) (lld)--(ld)--(d)--(rd)--(rrd) (u)--(d);
      \draw[blue,shift={(0,-0.1)}](-1.6,-0.8)..controls(-1.1,-0.8)and(1.1,0.8)..(1.6,0.8);
      %\draw[blue,shift={(0,-0.4)}](-1.6,1.1)..controls(-1.1,1.1)and(1.1,-0.4)..(1.6,-0.4);
      %\draw[blue,shift={(0,-0.75)}](-1.6,1.1)..controls(-1.1,1.1)and(1.1,-0.4)..(1.6,-0.4);

      %\draw[blue,shift={(0,0)}](-1.5,0.5)..controls(-0.5,0.5)and(0.5,-0.5)..(1.5,-0.5);
      %\draw[blue,shift={(0,-0.25)}](-1.5,0.5)..controls(-0.5,0.5)and(0.5,-0.5)..(1.5,-0.5);
      %\fill(lu)circle(1mm); \fill(ru)circle(1mm); \fill(ld)circle(1mm); \fill(rd)circle(1mm);
      %\fill(u)circle(1mm); \fill(d)circle(1mm);
      \draw[fill=white](-1.4,0)circle(1.3mm); \draw[fill=white](1.4,0)circle(1.3mm);
      \node[fill=white, inner sep=0.05] at(-1.4,0){\tiny$u$};
      \node[fill=white, inner sep=0.05] at(1.4,0){\tiny$v$};
      \draw (-1.25,0)--(1.25,0);
      \draw[red] (0,1)--(0,-0.9); 
      \draw[fill=black] (0,-0.1)circle(0.6mm);
      \node at (0.2,-0.2) {\footnotesize$p$};
      \node[red] at (0,-1.12) {$\alpha_e$};
      %\node at(0.2,0.15){$p$};
      %\node[blue] at(-1.9,1.1) {\footnotesize$\gamma_3^a$};
      %\node[blue] at(-1.9,0.75) {\footnotesize$\gamma_2^a$};
      %\node[blue] at(0,1.5) {$\gamma^a\underset{e}{\times} \gamma^b$};
      %\node[blue] at(-1.9,0.9) {\footnotesize$\gamma^a$};
      %\node[blue] at(1.9,-0.7) {\footnotesize$\gamma^b$};
      %\node[blue] at(1.9,-0.7) {\footnotesize$\gamma_2^b$};
      %\node[blue] at(1.9,-1.1) {\footnotesize$\gamma_1^b$};
   \end{tikzpicture}
   \end{tabular}
   \caption{a positive intersection (in the left) and a negative intersection point (in the right).}
   \end{figure}
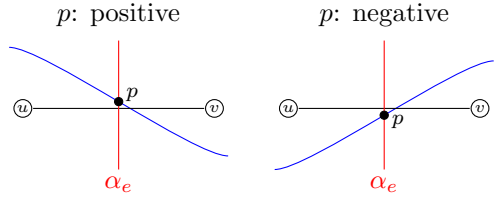

For $e\in G_1$ and $j\in \mathbb{Z}$, let 
\[
   \mathcal{A}(\mathbf{G})_e^j :=\{\mathcal{X}\in \mathcal{A}(\mathbf{G}) \mid g_e(\mathcal{X})=j\}, 
\]
where  $g(\mathcal{X}):=\sum_{\gamma\in \mathcal{X}} g(\gamma)$. In addition, let
\[
   \mathcal{A}(\mathbf{G})_e^{>0} :=\{\mathcal{X}\in \mathcal{A}(\mathbf{G}) \mid g_e(\mathcal{X})>0\} \quad \text{and} \quad 
   \mathcal{A}(\mathbf{G})_e^{<0} :=\{\mathcal{X}\in \mathcal{A}(\mathbf{G}) \mid g_e(\mathcal{X})<0\}.
\]

The following is a geometric interpretation of \cite[Theorem 4.6]{AAC} (see also \cite[Proposition 7.29]{AY20}). 
Suppose that $\gamma$ intersects with $\alpha_{e_1},\ldots, \alpha_{e_k}$ of $\mathbf{L}$ so that $p_{e_1},\ldots,p_{e_k}$ are intersection points of $\gamma$ and $\alpha_{e_1},\ldots, \alpha_{e_k}$ respectively, and $p_{e_{j+1}}$ is the next of $p_{e_j}$ for all $j\in \{1,\ldots,k-1\}$. Remember that $p_{e_{j+1}}$ is negative (resp., positive) if $p_{e_j}$ is positive (resp., negative). It determines a sequence $v_1,\ldots,v_{k+1}$ of  punctures of $\mathbf{G}$ with $s(e_i)=\{v_i,v_{i+1}\}$. 
We set a two-term complex $T_{\gamma} := (T^{-1}\overset{f}{\to} T^{0})$ by 
\[
   T^{0} := 
   \bigoplus_{\text{$p_{e_i}$:positive}} P_{e_{i}}, \quad 
   T^{-1}: = 
   \bigoplus_{\text{$p_{e_j}$:negative}} P_{e_{j}}
\]
and
\[
   f = (f_{ij})_{i,j\in\{1,\ldots,k\}}, \quad f_{ij}:=  
   \begin{cases}
   [e_{j-1}|e_{j}]_{v_j} &\text{if $i=j-1$} \\ 
   [e_{j+1}|e_{j}]_{v_{j+1}} &\text{if $i=j+1$} \\ 
      0  &\text{otherwise.}
   \end{cases}
\]
Here, $[e_{j-1}|e_{j}]_{v_j}$ (resp., $[e_{j+1}|e_{j}]_{v_{j+1}}$) is a homomorphism $P_{e_j} \to P_{e_{j-1}}$ (resp., $P_{e_j} \to P_{e_{j+1}}$) given by multiplying a sub-path of the special cycle $C_{v_j,e_{j-1}}$ (resp., $C_{v_{j+1},e_{j+1}}$) starting at $e_{j-1}$ (resp., at $e_{j+1}$) and ending at $e_{j}$.

\begin{theorem} \cite[Theorem 4.6]{AAC} \label{thm:AAC}
   The correspondence $\gamma \mapsto T_{\gamma}$ gives a bijection from the set of ${\mathbf{G}}$-arcs to the set of isomorphism classes of indecomposable $2$-pretilting complexes for $B_{\mathbf{G}}$.
   Moreover, it induces bijections 
   \[
      \mathcal{A}(\mathbf{G}) \overset{\sim}{\longrightarrow} \twotilt (B_{\mathbf{G}}) \quad \text{and} \quad \mathcal{A}(\mathbf{G})_e^j \overset{\sim}{\longrightarrow} \twotilt_e^j (B_{\mathbf{G}}) 
   \]
   for $e\in G_1$ and $j\in \mathbb{Z}$.
\end{theorem}

\begin{proof}
   In Section \ref{sec:sw}, we find in Proposition \ref{prop:sw} that there is a bijection between the set of $\mathbf{G}$-arcs and the set of signed walks on $\mathbf{G}$. Under this correspondence, one can check that the notion of admissibility on arcs coincide with that on signed walks (see \cite[Section 2]{AAC} for the definition of admissible collections of signed walks). Then we get the desired bijection by Theorem \cite[Theorem 4.6]{AAC}. 
\end{proof}

By Theorem \ref{thm:reduction}, we slightly generalize the above result. 

\begin{corollary} \label{cor:gentle}
   Let $\mathbf{G}$ be a plane tree and $A_{\mathbf{G}}$ the complete gentle algebra of $\mathbf{G}$. 
   For any ideal $J$ contained in $I_{\rm sc}$, we have bijections 
   \[
      \mathcal{A}(\mathbf{G}) \overset{\sim}{\longrightarrow} \twosilt (A_{\mathbf{G}}/J) \quad \text{and} \quad \mathcal{A}(\mathbf{G})_e^j \overset{\sim}{\longrightarrow} \twosilt_e^j (A_{\mathbf{G}}/J)
   \]
   for $e\in G_1$ and $j\in \mathbb{Z}$.
\end{corollary}

\begin{proof}
   The assertion is immediate from Theorems \ref{thm:reduction} and \ref{thm:AAC}.
\end{proof}

\begin{example}\label{ex:collections} 
\begin{enumerate}
\item The following figure describes a plane tree having $2$ edges and its dual: 
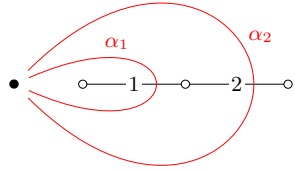
\[
   \scalebox{0.9}{
   \begin{tikzpicture}
   \coordinate(a) at (0,0);
   \coordinate(b) at (1.5,0);
   \coordinate(c) at (3,0);

   \draw (a)--node[fill=white,inner sep=1]{\small$1$}(b)--node[fill=white,inner sep=1]{\small$2$}(c);
      
   \draw[fill=white](a)circle(0.6mm);
   \draw[fill=white](b)circle(0.6mm);
   \draw[fill=white](c)circle(0.6mm);
   \node (a) at (-1,0) {$\bullet$}; 
   \draw[red] (a)..controls(1.7,1.2)and(1.7,-1.2)..(a);
   \draw[red] (a)..controls(1,2)and(2.5,1)..(2.5,0);
   \draw[red] (a)..controls(1,-2)and(2.5,-1)..(2.5,0);
   \node[red](l1) at (0.5,0.6) {\small$\alpha_1$};
   \node[red](l2) at (2.6,0.7) {\small$\alpha_2$};
   \end{tikzpicture} 
   }
\]
There are precisely $6$ complete collections of arcs on this plane tree, all of which are described as follows: 
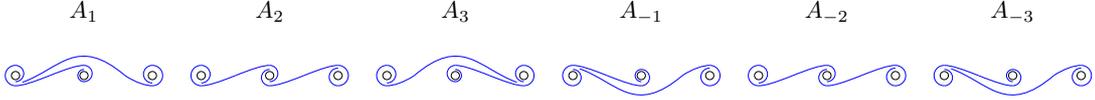
\[
\scalebox{0.9}{
\renewcommand{\arraystretch}{1}
\setlength{\tabcolsep}{2mm}
\begin{tabular}{ccccccccc}
$A_1$ & $A_2$ & $A_3$ & $A_{-1}$& $A_{-2}$& $A_{-3}$ \\
\begin{tikzpicture}
   \draw[fill=white](-1,0)circle(0.6mm);
   \draw[fill=white](0,0)circle(0.6mm);
   \draw[fill=white](1,0)circle(0.6mm);
   
   \draw[blue](-0.5,0)..controls(-0.75,-0.1)and(-0.9,-0.15)..(-1,-0.15);
   \draw[blue](-1,-0.15)arc(-90:-270:1.5mm);
   \draw[blue](-1,0.15)arc(90:-90:1.2mm);
   
   \draw[blue](-0.5,0)..controls(-0.25,0.1)and(-0.1,0.15)..(0,0.15);
   \draw[blue](0,0.15)arc(90:-90:1.2mm)arc(-90:-270:0.9mm);

   \draw[blue] (-0.9,-0.1)..controls(-0.35,0.1)and(-0.1,0.6)..(0.6,0);
   \draw[blue](0.6,0)..controls(0.75,-0.1)and(0.9,-0.15)..(1,-0.15);
   \draw[blue](1,-0.15)arc(-90:90:1.5mm);
   \draw[blue](1,0.15)arc(90:270:1.3mm);
   \node (xx) at(0,-0.6) {};
   \node (xx) at(0,0.6) {};
\end{tikzpicture} 
&
\begin{tikzpicture}
   \draw[fill=white](-1,0)circle(0.6mm);
   \draw[fill=white](0,0)circle(0.6mm);
   \draw[fill=white](1,0)circle(0.6mm);

   \draw[blue](0.5,0)..controls(0.75,0.1)and(0.9,0.15)..(1,0.15);
   \draw[blue](1,0.15)arc(90:-90:1.5mm);
   \draw[blue](1,-0.15)arc(-90:-270:1.3mm);
   
   \draw[blue](0.5,0)..controls(0.25,-0.1)and(0.1,-0.15)..(0,-0.15);
   \draw[blue](0,-0.15)arc(-90:-270:1.2mm);
   
   \draw[blue](-0.5,0)..controls(-0.75,-0.1)and(-0.9,-0.15)..(-1,-0.15);
   \draw[blue](-1,-0.15)arc(-90:-270:1.5mm);
   \draw[blue](-1,0.15)arc(90:-90:1.3mm);
   
   \draw[blue](-0.5,0)..controls(-0.25,0.1)and(-0.1,0.15)..(0,0.15);
   \draw[blue](0,0.15)arc(90:-90:1.2mm);
   \node (xx) at(0,-0.6) {};
   \node (xx) at(0,0.6) {};

\end{tikzpicture} 
&
\begin{tikzpicture}
   \draw[fill=white](-1,0)circle(0.6mm);
   \draw[fill=white](0,0)circle(0.6mm);
   \draw[fill=white](1,0)circle(0.6mm);
   
   \draw[blue](0.5,0)..controls(0.75,-0.1)and(0.9,-0.15)..(1,-0.15);
   \draw[blue](1,-0.15)arc(-90:90:1.5mm);
   \draw[blue](1,0.15)arc(90:270:1.2mm);
   
   \draw[blue](0.5,0)..controls(0.25,0.1)and(0.1,0.15)..(0,0.15);
   \draw[blue](0,0.15)arc(90:270:1.2mm)arc(270:450:0.9mm);

   \draw[blue] (0.9,-0.1)..controls(0.35,0.1)and(0.1,0.6)..(-0.6,0);
   \draw[blue](-0.6,0)..controls(-0.75,-0.1)and(-0.9,-0.15)..(-1,-0.15);
   \draw[blue](-1,-0.15)arc(-90:-270:1.5mm);
   \draw[blue](-1,0.15)arc(-270:-450:1.3mm);
   \node (xx) at(0,-0.6) {};
   \node (xx) at(0,0.6) {};
\end{tikzpicture} 
&
\begin{tikzpicture}
   \draw[fill=white](-1,0)circle(0.6mm);
   \draw[fill=white](0,0)circle(0.6mm);
   \draw[fill=white](1,0)circle(0.6mm);
   
   \draw[blue](-0.5,0)..controls(-0.75,0.1)and(-0.9,0.15)..(-1,0.15);
   \draw[blue](-1,0.15)arc(90:270:1.5mm);
   \draw[blue](-1,-0.15)arc(-90:90:1.2mm);
   
   \draw[blue](-0.5,0)..controls(-0.25,-0.1)and(-0.1,-0.15)..(0,-0.15);
   \draw[blue](0,-0.15)arc(-90:90:1.2mm)arc(90:270:0.9mm);
   
   \draw[blue] (-0.9,0.1)..controls(-0.35,-0.1)and(-0.1,-0.6)..(0.6,0);
   \draw[blue](0.6,0)..controls(0.75,0.1)and(0.9,0.15)..(1,0.15);
   \draw[blue](1,0.15)arc(90:-90:1.5mm);
   \draw[blue](1,-0.15)arc(-90:-270:1.3mm);
   \node (xx) at(0,0.-0.6) {};
   \node (xx) at(0,0.6) {};
\end{tikzpicture} 
&
\begin{tikzpicture}
   \draw[fill=white](-1,0)circle(0.6mm);
   \draw[fill=white](0,0)circle(0.6mm);
   \draw[fill=white](1,0)circle(0.6mm);

   \draw[blue](-0.5,0)..controls(-0.75,-0.1)and(-0.9,-0.15)..(-1,-0.15);
   \draw[blue](-1,-0.15)arc(-90:-270:1.5mm);
   \draw[blue](-1,0.15)arc(90:-90:1.3mm);
   
   \draw[blue](-0.5,0)..controls(-0.25,0.1)and(-0.1,0.15)..(0,0.15);
   \draw[blue](0,0.15)arc(90:-90:1.2mm);
   
   \draw[blue](0.5,0)..controls(0.75,0.1)and(0.9,0.15)..(1,0.15);
   \draw[blue](1,0.15)arc(90:-90:1.5mm);
   \draw[blue](1,-0.15)arc(-90:-270:1.3mm);
   
   \draw[blue](0.5,0)..controls(0.25,-0.1)and(0.1,-0.15)..(0,-0.15);
   \draw[blue](0,-0.15)arc(-90:-270:1.2mm);
   \node (xx) at(0,-0.6) {};
   \node (xx) at(0,0.6) {};

\end{tikzpicture} 
& 
\begin{tikzpicture}
   \draw[fill=white](-1,0)circle(0.6mm);
   \draw[fill=white](0,0)circle(0.6mm);
   \draw[fill=white](1,0)circle(0.6mm);
   
   \draw[blue](-0.5,0)..controls(-0.75,0.1)and(-0.9,0.15)..(-1,0.15);
   \draw[blue](-1,0.15)arc(90:270:1.5mm);
   \draw[blue](-1,-0.15)arc(270:450:1.2mm);
   
   \draw[blue](-0.5,0)..controls(-0.25,-0.1)and(-0.1,-0.15)..(0,-0.15);
   \draw[blue](0,-0.15)arc(-90:90:1.2mm)arc(90:270:0.9mm);

   \draw[blue] (-0.9,0.1)..controls(-0.35,-0.1)and(-0.1,-0.6)..(0.6,0);
   \draw[blue](0.6,0)..controls(0.75,0.1)and(0.9,0.15)..(1,0.15);
   \draw[blue](1,0.15)arc(90:-90:1.5mm);
   \draw[blue](1,-0.15)arc(-90:-270:1.3mm);
   \node (xx) at(0,-0.6) {};
   \node (xx) at(0,0.6) {};
\end{tikzpicture} 
\end{tabular}
}
\]
Here, we have 
\[
   g(A_1) = (-2,1), \quad  g(A_2) = (-1,-1), \quad g(A_3) = (-1,2)
\]
and $g(A_{-i}) = -g(A_{i})$ for $i=1,2,3$.
\item For a plane tree 
\[
   \begin{tikzpicture}
   \coordinate(a) at (0,0);
   \coordinate(b) at (1.5,0);
   \coordinate(c) at (3,0);
   \coordinate(d) at (4.5,0);

   \draw (a)--node[fill=white,inner sep=1]{\small$1$}(b)--node[fill=white,inner sep=1]{\small$2$}(c)--node[fill=white,inner sep=1]{\small$3$}(d);

   \draw[fill=white](a)circle(0.6mm);
   \draw[fill=white](b)circle(0.6mm);
   \draw[fill=white](c)circle(0.6mm);
   \draw[fill=white](d)circle(0.6mm);
   \end{tikzpicture} 
\]
there are precisely $20$ complete collection of arcs. 
We describe in Figure \ref{fig:intro} the bijection in Theorem \ref{thm:AAC} between complete collections of arcs and $2$-tilting complexes for this plane tree.
\end{enumerate}
\end{example}

\subsection{Arcs and signed walks} \label{sec:sw}
In this subsection, we give a bijection between arcs and signed walks. 
Let $\mathbf{G}$ be a plane tree and $\mathbf{L}$ its dual, which are embedded into the sphere $\mathbb{S}$. Let $n=|\mathbf{G}|$ be the number of edges of $\mathbf{G}$. 

Recall that a signed walk is a walk $w=(e_1,\ldots,e_k)$ on $\mathbf{G}$ equipped with assignment of signs $\pm$ with $\epsilon(e_i)\neq \epsilon(e_{i+1})$ for all $i\in \{1,\ldots,k-1\}$.

\begin{proposition}  \label{prop:sw}
   For a plane tree $\mathbf{G}$, there are bijections between the following sets.
   \begin{enumerate}
      \item[(a)] The set of $\mathbf{G}$-arcs. 
      \item[(b)] The set of signed walks on $\mathbf{G}$. 
      \item[(c)] The set $\{\{u,v\} \mid u,v\in G_0,u\neq v\}\times\{\pm 1\}$.
      \end{enumerate}
   In particular, their cardinality are $n(n+1)$. 
\end{proposition}

\begin{proof}
   We give a bijection from (a) to (b). 
   Let $\gamma$ be a $\mathbf{G}$-arc. Suppose that it intersects with $\alpha_{e_1},\ldots, \alpha_{e_k}$ of $\mathbf{L}$ so that $p_{e_1},\ldots,p_{e_k}$ are intersection points of $\gamma$ and $\alpha_{e_1},\ldots, \alpha_{e_k}$ respectively and $p_{e_{j+1}}$ is the next of $p_{e_j}$ for all $j\in \{1,\ldots,k-1\}$. Since $\mathbf{G}$ is a tree, $e_1,\ldots, e_k$ are pairwise distinct edges. To an arc $\gamma$, we associate the signed walk $w_{\gamma}$ consisting of a walk $(e_1,\ldots, e_k)$ and an assignment of signs $\epsilon(e_i):= g_{e_i}(\gamma) \in \{\pm1\}$ for all $i\in [1,k]$. 
   Conversely, one can check that every signed walk on a plane tree $\mathbf{G}$ can be obtained in this way. 
   Thus, the correspondence $\gamma\mapsto w_{\gamma}$ gives the desired bijection. 

   The bijection between (b) and (c) is clear since $\mathbf{G}$ is a tree. 
\end{proof}

\begin{lemma} \label{lem:abs}
   Every arc $\gamma$ satisfies $|g_e(\gamma)|\in \{0,\pm1\}$ for all $e\in G_1$. 
\end{lemma}

\begin{proof}
   It has already shown in the precious discussion.
\end{proof}

\begin{lemma} \label{lem:g-invarc}
   If two admissible collections $\mathcal{X},\mathcal{Y}$ satisfy $g(\mathcal{X})=g(\mathcal{Y})$, then $\mathcal{X}=\mathcal{Y}$ holds.
\end{lemma}

\begin{proof}
   It follows from Proposition \ref{prop:g-inv}.
\end{proof}

\begin{proposition} \label{prop:partialarc}
   Every reduced collection is a subset of some complete collection of $\mathbf{G}$-arcs. 
   In particular, it is complete if and only if it has precisely $n$ elements.
\end{proposition}
\begin{proof}
   It follows from Proposition \ref{prop:partialsilt}. 
\end{proof}

\begin{proposition} \label{prop:greater}
   For an edge $e$ and an integer $|j|\notin [1,n]$, the set $\mathcal{A}(\mathbf{G})_e^j$ is empty.
\end{proposition}

\begin{proof}
   By Proposition \ref{prop:zero}, the set $\mathcal{A}(\mathbf{G})_e^0$ is empty. 
   On the other hand, any complete collection $\mathcal{X}$ of $\mathbf{G}$-arcs consists of precisely $n$ elements by Proposition \ref{prop:partialarc}. So, $|g_e(\mathcal{X})|\leq n$ holds by Lemma \ref{lem:abs}.
\end{proof}

We end this subsection with the following observation on external edges. 
Now, we say that an edge of $\mathbf{G}$ is an \textit{external edge} if one of its endpoints is \textit{external}, that is, the number of edges incident to it is exactly $1$.

\begin{lemma} \label{lem:external}
Let $e$ be an external edge having external vertex $v$. 
For an arc $\gamma$, the following conditions are equivalent: 
\begin{enumerate}
   \item[(a)] $\gamma$ intersects with $\alpha_e$ positively (resp., negatively); 
   \item[(b)] $\gamma$ spirals around $v$ counterclockwise (resp., clockwise). 
   \item[(c)] $g_e(\gamma)=1$ (resp., $g_e(\gamma)=-1$).
\end{enumerate}
\end{lemma}

\begin{proof}
   The conditions (a) and (b) are equivalent since the polygon $\triangle_v$ is a monogon given by a loop $\alpha_e$. 
   On the other hand, (a) and (c) are equivalent by Lemma \ref{lem:abs}.
\end{proof}

\subsection{Flipping a plane tree}
We study a combinatorial operation called flip on ribbon graphs. It was introduced by Kauer \cite{Kauer98} and Aihara \cite{Aihara14} to describe tilting mutation of Brauer graph algebras. 

For our purpose, we define flips for plane trees here. Let $\mathbf{G}=(G_0,G_1,s,\sigma)$ be a plane tree.

\begin{definition} \label{def:flip}
For an edge $e$ of $\mathbf{G}$, let $\mathbf{G}'=(G_0',G_1',s',\sigma')$ be a plane tree defined as follows:  Let $a,b$ be endpoints of $e$. 
\begin{enumerate}
   \item If $e$ is an external edge having external vertex $a$, let $f:=\sigma_b(e)$ and $c$ its another endpoint. Then $\mathbf{G}'$ is given by the following data.
   \begin{itemize}
      \item The set $G_0'$ of vertices bijectively corresponds to $G_0$.  
      \item The set $G'_1$ of edges bijectively corresponds to $G_1$. 
      \item The map $s'$ is given by $s'(e):=\{a,c\}$ and $s'(d):=s(d)$ for all $e\neq d$. 
      \item The cyclic permutation $\sigma'=(\sigma'_v)_{v}$ is given by 
      \[
         \sigma'_b:=(\ldots, \sigma^{-3}_b(f), \sigma^{-2}_b(f), f, \sigma_b(f), \ldots), \quad \sigma'_c:=(\ldots, \sigma^{-1}_c(f),f, e, \sigma_c(f),\sigma^2_c(f),\ldots)
      \] 
      and $\sigma'_v:=\sigma_v$ for all $v\notin \{b,c\}$.
   \end{itemize}
   \item If $e$ is not external, let $f_v:=\sigma_v(e)$ and $c_v$ its another endpoint for $v\in \{a,b\}$. Then $\mathbf{G}'$ is given by the following data.
   \begin{itemize}
      \item The set $G_0'$ of vertices is bijectively correspond to $G_0$.  
      \item The set $G'_1$ of edges is bijectively correspond to $G_1$. 
      \item The map $s'$ is given by $s'(e):=\{c_a,c_b\}$ and $s'(d):=s(d)$ for all $e\neq d$. 
      \item The cyclic permutation $\sigma'=(\sigma'_v)_{v}$ is given by 
      \[
         \sigma'_v:=(\ldots, \sigma^{-3}_v(f_v), \sigma^{-2}_v(f_v), f_v, \sigma_v(f_v), \ldots), \quad \sigma'_{c_v}:=(\ldots, \sigma^{-1}_{c_v}(f_v),f_v, e, \sigma_{c_v}(f_v),\sigma^2_{c_v}(f_v),\ldots)
      \]
      for $v\in \{a,b\}$ and $\sigma'_v:=\sigma_v$ for all $v\notin \{a,b,c_a,c_b\}$.
   \end{itemize}
\end{enumerate} 
We denote this plane tree by $\mu_e(\mathbf{G})$ and call \textit{flip} of $\mathbf{G}$ at $e$. See Figure \ref{fig:flip}. 
\end{definition}

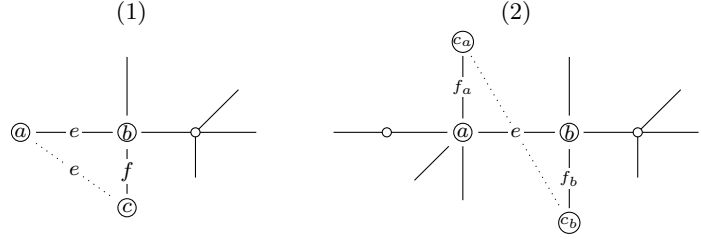
\begin{figure}[t]
\[
\setlength{\tabcolsep}{5mm}
\begin{tabular}{ccccccccc}
   \small (1) & \small (2) \\
   \begin{tikzpicture}[baseline=0mm]
      %\draw[red] (0,0.6)--(0,-0.7);
      %\node(al) at (0,-0.85) {\color{red}\footnotesize$\alpha_e$};
      \node(aa) at (-2,0) {\footnotesize$a$};
      \node(b) at (-0.6,0) {\footnotesize$b$};
      \node(cc) at (-0.6,-1) {\footnotesize$c$};
      \coordinate(a) at (0.3,0);

      \draw (a)--(b); 
      \draw (a)--($(a)+(45:0.8)$) (a)--($(a)+(0:0.8)$) (a)--($(a)+(-90:0.6)$);
      \draw (b)--($(b)+(90:1)$) (b)--node[fill=white,inner sep=1]{\footnotesize$e$}(aa) (b)--node[fill=white,inner sep=1]{\footnotesize$f$}(cc); 

      \draw[fill=white](a)circle(0.6mm);
      \draw (aa)circle(1.2mm);
      \draw (b)circle(1.2mm);
      \draw (cc)circle(1.2mm);
      \draw[dotted] (aa)--node[fill=white,inner sep=1]{\footnotesize$e$}(cc);

   \end{tikzpicture}
   & 
   \begin{tikzpicture}[baseline=0mm]
      %\draw[red] (0,0.6)--(0,-0.7);
      %\node(al) at (0,-0.85) {\color{red}\footnotesize$\alpha_e$};
      \node(aa) at (-2,0) {\footnotesize$a$};
      \node(b) at (-0.6,0) {\footnotesize$b$};
      \node(cc) at (-0.6,-1.2) {\tiny$c_b$};
      \node(cca) at (-2,1.2) {\tiny$c_a$};
      \coordinate(a) at (0.3,0);
      \coordinate(d) at (-3,0);

      \draw (a)--(b); 
      \draw (a)--($(a)+(45:0.8)$) (a)--($(a)+(0:0.8)$) (a)--($(a)+(-90:0.6)$);
      \draw (b)--($(b)+(90:1)$) (b)--(aa) (b)--node[fill=white,inner sep=1]{\tiny$f_b$}(cc); 
      \draw (aa)--node[fill=white,inner sep=1]{\tiny$f_a$}(cca);
      \draw (aa)--(d);
      \draw (d)--($(d)+(180:0.7)$);
      \draw (aa)--($(aa)+(-90:0.9)$);
      \draw (aa)--($(aa)+(-135:0.9)$);

      \draw[fill=white](a)circle(0.6mm);
      \draw[fill=white](d)circle(0.6mm);
      \draw (aa)circle(1.2mm);
      \draw (b)circle(1.2mm);
      \draw (cc)circle(1.4mm);
      \draw (cca)circle(1.4mm);

      \draw[dotted] (cca)--node[fill=white,inner sep=1]{\footnotesize$e$}(cc);

   \end{tikzpicture}
\end{tabular}
\]
\caption{A local figure for a flip at $e$}
\label{fig:flip}
\end{figure}

Now, we fix an edge $e$ of $\mathbf{G}$. Let $T$ be a direct sum of complexes
\begin{eqnarray}\nonumber
   \text{\small $0$-th} \quad \quad \text{\small $1$-st} \ \\ \nonumber
   X := \bigoplus_{d\notin G_1\setminus\{e\}} P_d \quad \text{and} \quad P'_e := \quad (P \ \overset{h}{\longrightarrow} \ P_e),
\end{eqnarray}
where $h$ is a minimal right $(\add X)$-approximation of $P_e$.   
By \cite[Theorem 2.2]{Aihara14}, this is a tilting complex such that
\[
   \End_{\Db(\module B_\mathbf{G})}(T) \cong B_{\mathbf{H}},
\]
where $\mathbf{H}:=\mu_e(\mathbf{G})$.
Thus, it induces a triangle equivalence 
\[
   {\mathbf{R}\mathrm{Hom}}(T,-) \colon \Db(\module B_{\mathbf{G}}) \overset{\sim}{\longrightarrow} \Db(\module B_\mathbf{H})
\]
mapping $T \mapsto B_\mathbf{H}$. 
Furthermore, it gives a bijection 
\begin{equation} \label{right mutation}
   \twotilt_e^{>0} (B_{\mathbf{G}}) \overset{\sim}{\longrightarrow} \twotilt_e^{<0} (B_\mathbf{H}).
\end{equation}

\begin{proposition} \label{prop:flip}
   In the above, we have a commutative diagram

\[
   \xymatrix{
      \mathcal{A}(\mathbf{G})_{e}^{>0} \ar[r]^{\sim} \ar[d]& \mathcal{A}(\mathbf{H})_{e}^{<0} \ar[d] \\ 
      \twotilt_e^{>0} (B_{\mathbf{G}}) \ar[r]^{\sim} & \twotilt_e^{<0} (B_{\mathbf{H}})
   }
\]
where two vertical arrows are bijections in Theorem \ref{thm:AAC} and the bottom horizontal arrow is the bijection in (\ref{right mutation}). 
\end{proposition}

\begin{proof}
   It is clear from the previous discussion.
\end{proof}

Here, we describe the bijection 
\begin{equation} \label{eq:fliparc}
   \mu_e\colon \mathcal{A}(\mathbf{G})_e^{>0} \overset{\sim}{\longrightarrow} \mathcal{A}(\mathbf{H})_e^{<0}
\end{equation}
in Proposition \ref{prop:flip}. 

\begin{enumerate}
\item Assume that $e$ is external with external vertex $a$. By definition, each $\mathbf{G}$-arc $\gamma$ with $g_e(\gamma)=0$ is naturally regarded as an $\mathbf{H}$-arc. So, let $\gamma'$ be $\gamma$ itself.
On the other hand, for a $\mathbf{G}$-arc $\gamma$ with $g_e(\gamma)=1$, it spirals around $a$ in counterclockwise direction by Lemma \ref{lem:external}. 
Let $\gamma'$ be a curve of $\mathbb{S}$ obtained from $\gamma$ by reversing the spiral around $a$, in clockwise direction. 
Then it clearly provides an $\mathbf{H}$-arc whose $g$-vector is given by
\[
   g_d(\gamma') = 
   \begin{cases} 
      - g_e(\gamma)  & \text{if $d = e$,} \\ 
      g_e(\gamma) + g_d(\gamma) & \text{if $d = f$,}\\ 
      g_d(\gamma) & \text{otherwise}. 
   \end{cases}
\]

\item Assume that $e$ is not external. In this case, it is easy to see that every $\mathbf{G}$-arc $\gamma$ with $g_e(\gamma) \in \{0,1\}$ can be regarded as an $\mathbf{H}$-arc. So, let $\gamma'$ be $\gamma$ itself. 
Furthermore, its $g$-vector is given by
\[
   g_d(\gamma') = 
   \begin{cases} 
      - g_e(\gamma)  & \text{if $d = e$}, \\ 
      g_e(\gamma) + g_d(\gamma) & \text{if $d\in \{f_u,f_v\}$}, \\ 
      g_d(\gamma) & \text{otherwise}. 
   \end{cases}
\]
\end{enumerate}

For each case, the correspondence $\gamma\mapsto \gamma'$ induces the desired bijection in (\ref{eq:fliparc}). In fact, the commutativity of the diagram in Theorem \ref{prop:flip} follows from Lemma \ref{lem:g-invarc}.

\begin{example}
For a plane tree $\mathbf{G}$ in Example \ref{ex:collections}(2), we consider a flip $\mathbf{H}:=\mu_1(\mathbf{G})$ with respect to the edge $1$:
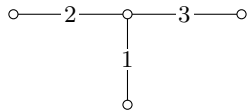
\[
   \begin{tikzpicture}
   \coordinate(a) at (3,-1.2);
   \coordinate(b) at (1.5,0);
   \coordinate(c) at (3,0);
   \coordinate(d) at (4.5,0);

   \draw (a)--node[fill=white,inner sep=1]{\small$1$}(c) (b)--node[fill=white,inner sep=1]{\small$2$}(c)--node[fill=white,inner sep=1]{\small$3$}(d);
   \draw[fill=white](a)circle(0.6mm);
   \draw[fill=white](b)circle(0.6mm);
   \draw[fill=white](c)circle(0.6mm);
   \draw[fill=white](d)circle(0.6mm);
   \end{tikzpicture} 
\]

All complete collection on this plane tree are described in \cite[Example 2.10]{Adachi16b} for example. 
The following figure describes the bijection in Proposition \ref{prop:flip}. Here, each black point in the left (resp., right) figure corresponds to an element of $\mathcal{A}(\mathbf{G})_1^{>0}$ (resp., $\mathcal{A}(\mathbf{H})_1^{<0}$): 

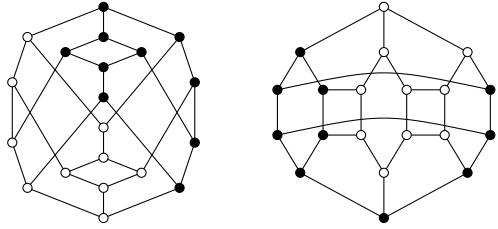
\[
\renewcommand{\arraystretch}{1}
\setlength{\tabcolsep}{3mm}
\begin{tabular}{ccccccccc}
\begin{tikzpicture}
\coordinate(v1) at (0,0.2);
\coordinate(v2) at (0,0.6);
\coordinate(v3) at (0.5,0.8);
\coordinate(v4) at (-0.5,0.8);
\coordinate(v5) at (0,1);
\coordinate(v6) at (0,1.4);
\coordinate(v7) at (1.2,0.4);
\coordinate(v8) at (1,1);
\coordinate(v9) at (-1.2,0.4);
\coordinate(v10) at (-1,1);

\coordinate(-v1) at ($-1*(v1)$);
\coordinate(-v2) at ($-1*(v2)$);
\coordinate(-v3) at ($-1*(v3)$);
\coordinate(-v4) at ($-1*(v4)$);
\coordinate(-v5) at ($-1*(v5)$) ;
\coordinate(-v6) at ($-1*(v6)$);
\coordinate(-v7) at ($-1*(v7)$);
\coordinate(-v8) at ($-1*(v8)$);
\coordinate(-v9) at ($-1*(v9)$);
\coordinate(-v10) at ($-1*(v10)$);

\draw[fill=black](v1)circle(0.6mm);
\draw[fill=black](v2)circle(0.6mm);
\draw[fill=black](v3)circle(0.6mm);
\draw[fill=black](v4)circle(0.6mm);
\draw[fill=black](v5)circle(0.6mm);
\draw[fill=black](v6)circle(0.6mm);
\draw[fill=black](v7)circle(0.6mm);
\draw[fill=black](v8)circle(0.6mm);

\draw (v1)--(v2) (v2)--(v3)--(v5)--(v4)--cycle (v5)--(v6) (v7)--(v8)--(v6)--(v10)--(v9);
\draw (-v1)--(-v2) (-v2)--(-v3)--(-v5)--(-v4)--cycle (-v5)--(-v6) (-v7)--(-v8)--(-v6)--(-v10)--(-v9);
\draw (-v8)--(v1)--(-v10) (v4)--(-v7) (v3)--(-v9);
\draw (v8)--(-v1)--(v10) (-v4)--(v7) (-v3)--(v9);
\draw (v9)--(-v7) (-v9)--(v7);

\draw[fill=white](v9)circle(0.6mm);
\draw[fill=white](v10)circle(0.6mm);

\draw[fill=white](-v1)circle(0.6mm);
\draw[fill=white](-v2)circle(0.6mm);
\draw[fill=white](-v3)circle(0.6mm);
\draw[fill=white](-v4)circle(0.6mm);
\draw[fill=white](-v5)circle(0.6mm);
\draw[fill=white](-v6)circle(0.6mm);
\draw[fill=white](-v7)circle(0.6mm);
\draw[fill=white](-v8)circle(0.6mm);
\draw[fill=black](-v9)circle(0.6mm);
\draw[fill=black](-v10)circle(0.6mm);
\end{tikzpicture}
& \quad 
\begin{tikzpicture}
   \coordinate(v1) at (0,1.4);
   \coordinate(v2) at (-1.1,0.8);
   \coordinate(v3) at (0,0.8);
   \coordinate(v4) at (1.1,0.8);
   \coordinate(v5) at (-1.4,0.3);
   \coordinate(v6) at (-0.8,0.3);
   \coordinate(v7) at (-0.3,0.3);
   \coordinate(v8) at (0.3,0.3);
   \coordinate(v9) at (0.8,0.3);
   \coordinate(v10) at (1.4,0.3);
   
   \coordinate(-v1) at ($-1*(v1)$);
   \coordinate(-v2) at ($-1*(v2)$);
   \coordinate(-v3) at ($-1*(v3)$);
   \coordinate(-v4) at ($-1*(v4)$);
   \coordinate(-v5) at ($-1*(v5)$);
   \coordinate(-v6) at ($-1*(v6)$);
   \coordinate(-v7) at ($-1*(v7)$);
   \coordinate(-v8) at ($-1*(v8)$);
   \coordinate(-v9) at ($-1*(v9)$);
   \coordinate(-v10) at ($-1*(v10)$);

\draw (v1)--(v2) (v1)--(v3) (v1)--(v4) (v6)--(v7) (v8)--(v9); 
\draw (-v1)--(-v2) (-v1)--(-v3) (-v1)--(-v4) (-v6)--(-v7) (-v8)--(-v9); 
\draw (v2)--(v5)--(-v10)--(-v4)--(-v9)--(v6)--cycle (v3)--(v7)--(-v8)--(-v3)--(-v7)--(v8)--cycle (v4)--(v9)--(-v6)--(-v2)--(-v5)--(v10)--cycle;
\draw (v5)..controls(0,0.6)..(v10);
\draw (-v5)..controls(0,0)..(-v10);

\draw[fill=white](v1)circle(0.6mm);
\draw[fill=black](v2)circle(0.6mm);
\draw[fill=white](v3)circle(0.6mm);
\draw[fill=white](v4)circle(0.6mm);
\draw[fill=black](v5)circle(0.6mm);
\draw[fill=black](v6)circle(0.6mm);
\draw[fill=white](v7)circle(0.6mm);
\draw[fill=white](v8)circle(0.6mm);
\draw[fill=white](v9)circle(0.6mm);
\draw[fill=black](v10)circle(0.6mm);

\draw[fill=black](-v1)circle(0.6mm);
\draw[fill=black](-v2)circle(0.6mm);
\draw[fill=white](-v3)circle(0.6mm);
\draw[fill=black](-v4)circle(0.6mm);
\draw[fill=black](-v5)circle(0.6mm);
\draw[fill=white](-v6)circle(0.6mm);
\draw[fill=white](-v7)circle(0.6mm);
\draw[fill=white](-v8)circle(0.6mm);
\draw[fill=black](-v9)circle(0.6mm);
\draw[fill=black](-v10)circle(0.6mm);

\end{tikzpicture}
\end{tabular}
\]
\end{example}

\subsection{The opposite plane tree} \label{seq:opp}
For a plane tree $\mathbf{G}=(G_0,G_1,s,\sigma)$, let ${\mathbf{G}}^{\rm op}$ be a plane tree 
whose underlying graph is the same as $\mathbf{G}$ and the cyclic permutation is given by $\sigma^{-1}:=(\sigma_v^{-1})_{v\in G_0}$.
We call it the \textit{opposite plane tree} of ${\mathbf{G}}$.
 
By symmetry, we have the following.

\begin{proposition} \label{prop:opp}
For any $e\in G_1$ and any $j\in \mathbb{Z}$, we have a commutative diagram 
\[
   \xymatrix{
      \mathcal{A}(\mathbf{G})_e^{j} \ar[r]_{(-)^{\rm op}}^{\sim} \ar[d] & \mathcal{A}(\mathbf{G}^{\rm op})_e^{-j}\ar[d] \\
      \twotilt_e^j (B_{\mathbf{G}}) \ar[r]^{\sim} & \twotilt_e^{-j} (B_{\mathbf{G}^{\rm op}})
   }
\] 
where two vertical arrows are bijections in Theorem \ref{thm:AAC} and two horizontal arrows are canonical ones.
\end{proposition}

\begin{proof}
   First, the algebra $B_{\mathbf{G}^{\rm op}}$ is naturally isomorphic to the opposite algebra of $B_{\mathbf{G}}$. So, the $k$-dual $\Hom_{k}(-,k)$ gives a bijection $\twotilt_e^j(B_{\mathbf{G}}) \overset{\sim}{\to} \twotilt_e^{-j} (B_{\mathbf{G}^{\rm op}})$ for any $e\in G_1$ and $j \in \mathbb{Z}$. 
   
   Second, we find that every complete collection of $\mathbf{G}$-arcs is canonically identified with a complete collection of $\mathbf{G}^{\rm op}$-arcs, vise versa. 
   However, a role of positive intersection points and negative intersection are swapped. This provides the bijection $(-)^{\rm op}$ in the statement. 
\end{proof}

\section{Gluing plane trees} \label{sec:glue}
Throughout this section, let $\mathbf{G}=(G_0,G_1,s, \sigma)$ be a plane tree and $\mathbf{L}$ its dual, which are embedded into the sphere $\mathbb{S}$. Assume that $n:=|\mathbf{G}|>1$.

Fix an edge $e$ of $\mathbf{G}$ having $a,b$ as its endpoints.
Since $\mathbf{G}$ is a plane tree, it determines a pair of plane subtrees ${\mathbf{G}}^{a},{\mathbf{G}}^{b}$ of ${\mathbf{G}}$ satisfying $G^{a}_1\cap G^{b}_1 = \{e\}$ and $G^{a}_1\cup G^{b}_1 = G_1$. 
Without loss of generality, we may assume that $e$ is an external edge of ${\mathbf{G}}^{a}$ (resp. ${\mathbf{G}}^{b})$ having external vertex $a$ (resp., $b$). 
For $v \in \{a,b\}$, the dual $\mathbf{L}^v$ of $\mathbf{G}^v$, embedded into the sphere $\mathbb{S}^v$, is canonically included in $\mathbf{L}$ as a ribbon subgraph.

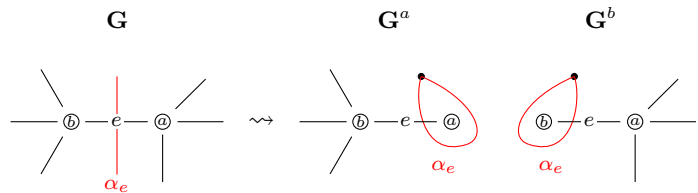
\[
\renewcommand{\arraystretch}{1}
\setlength{\tabcolsep}{1mm}
\begin{tabular}{ccccccccc}
   \small$\mathbf{G}$ & &\small$\mathbf{G}^a$ & \small$\mathbf{G}^b$ \\
   \begin{tikzpicture}
      \draw[red] (0,0.6)--(0,-0.7);
      \node(al) at (0,-0.85) {\color{red}\footnotesize$\alpha_e$};
      \node(a) at (0.6,0) {\tiny$a$};
      \node(b) at (-0.6,0) {\tiny$b$};
      \draw (a)--node[fill=white,inner sep=1]{\footnotesize$e$}(b); 
      \draw (a)--($(a)+(45:0.8)$) (a)--($(a)+(0:0.8)$) (a)--($(a)+(-90:0.8)$);
      \draw (b)--($(b)+(120:0.8)$) (b)--($(b)+(180:0.8)$) (b)--($(b)+(-120:0.8)$);
      
      \draw (a)circle(1mm);
      \draw (b)circle(1mm);

      \node(xx) at (0,-1) {};
      \node(xx) at (0,1) {};
      
   \end{tikzpicture}
   &
\begin{tikzpicture}
   \node(sq) at (0,0) {$\rightsquigarrow$};
   \node(xx) at (0,-1) {};
      \node(xx) at (0,1) {};
\end{tikzpicture}
   &
   \begin{tikzpicture}
      \coordinate(bl) at (0.2,0.6);
      \draw[fill=black] (bl)circle(0.4mm);
      \draw[red] (bl)..controls(0.9,0.3)and(1.1,-0.2)..(0.8,-0.3);
      \draw[red] (0.8,-0.3)..controls(0.2,-0.5)and(0.1,0.3)..(bl);
      \node at (0.5,-0.6){\color{red}\footnotesize$\alpha_e$};
      \node(a) at (0.6,0) {\tiny$a$};
      \node(b) at (-0.6,0) {\tiny$b$};
      \draw (a)--node[fill=white,inner sep=1]{\footnotesize$e$}(b); 
      %\draw (a)--($(a)+(45:0.8)$) (a)--($(a)+(0:0.8)$) (a)--($(a)+(-90:0.8)$);
      \draw (b)--($(b)+(120:0.8)$) (b)--($(b)+(180:0.8)$) (b)--($(b)+(-120:0.8)$);
      
      \draw (a)circle(1mm);
      \draw (b)circle(1mm);

      \node(xx) at (0,-1) {};
      \node(xx) at (0,1) {};
   
   \end{tikzpicture}
   & 
   \begin{tikzpicture}
      \coordinate(bl) at (-0.2,0.6);
      \draw[fill=black] (bl)circle(-0.4mm);
      \draw[red] (bl)..controls(-0.9,0.3)and(-1.1,-0.2)..(-0.8,-0.3);
      \draw[red] (-0.8,-0.3)..controls(-0.2,-0.5)and(-0.1,0.3)..(bl);
      \node at (-0.5,-0.6){\color{red}\footnotesize$\alpha_e$};

      \node(a) at (0.6,0) {\tiny$a$};
      \node(b) at (-0.6,0) {\tiny$b$};
      \draw (a)--node[fill=white,inner sep=1]{\footnotesize$e$}(b); 
      \draw (a)--($(a)+(45:0.8)$) (a)--($(a)+(0:0.8)$) (a)--($(a)+(-90:0.8)$);
      %\draw (b)--($(b)+(120:0.8)$) (b)--($(b)+(180:0.8)$) (b)--($(b)+(-120:0.8)$);
      
      \draw (a)circle(1mm);
      \draw (b)circle(1mm);
      \node(xx) at (0,-1) {};
      \node(xx) at (0,1) {};
   
   \end{tikzpicture}
\end{tabular}
\]

\begin{proposition}
For $v\in \{a,b\}$, the set of $\mathbf{G}^v$-arcs is precisely a set of all $\mathbf{G}$-arcs $\gamma$ satisfying $g_d(\gamma)=0$ for all $d\in G_1\setminus G_1^v$.
\end{proposition}

\begin{proof}
   Our claim is clear from the bijection between (a) and (b) in Proposition \ref{prop:sw}. More precisely, as $\mathbf{G}^v$ being a subtree of $\mathbf{G}$, 
   signed walks on $\mathbf{G}^v$ are precisely signed walks on $\mathbf{G}$ having only the edges in $\mathbf{G}^v$. 
\end{proof}

Next, the triple $(\mathbb{S},\mathbf{G}, \mathbf{L})$ of $\mathbf{G}$ is constructed from that of $\mathbf{G}^a$ and of $\mathbf{G}^b$ as follows:
For $v\in \{a,b\}$, let $\triangle_v^v$ be the polygon of $v$ in $\mathbb{S}^v$, which is a monogon with side $\alpha_e$ of $\mathbf{L}^v$. 
Then the surface $\mathbb{S}$ is obtained from $\mathbb{S}^a \setminus \triangle_a^a$ and $\mathbb{S}^b\setminus \triangle_b^b$ by gluing them along $\alpha_e$, with preserving the respective orientation of surfaces. 
Together, $\mathbf{G}$ (resp., ${\mathbf{L}}$) is constructed from ${\mathbf{G}}^a$ and ${\mathbf{G}}^b$ (resp., ${\mathbf{L}}^a$ and ${\mathbf{L}}^b$). 

Along this construction, we glue a pair of arcs on $\mathbf{G}^a$ and on $\mathbf{G}^b$ to get $\mathbf{G}$-arcs as follows. 
For $v\in \{a,b\}$, let $\gamma^v$ be a $\mathbf{G}^a$-arc satisfying $g_e(\gamma^v)=1$, and let $p^v$ be an intersection point of $\gamma^v$ and $\alpha_e$. In this case, $\gamma^v$ spirals around $v$ in $\mathbb{S}^v$ counterclockwise by Lemma \ref{lem:external}. Consider a restriction $\bar{\gamma}^v$ of $\gamma^v$ on $\mathbb{S}^v\setminus \triangle_v^v$, one of whose endpoint is $p^v$. 
We write $\gamma^a\underset{e}{\times}\gamma^b$ for the curve of $\mathbb{S}$ obtained from $\bar{\gamma}^a$ and $\bar{\gamma}^b$ by gluing them at $p:=p^a=p^b$ on $\alpha_e$. See Figure \ref{fig:glue_arc}.

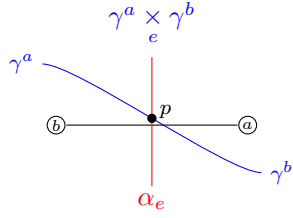
\begin{figure}[htp]   \centering
\begin{tikzpicture}[scale=0.9]
      \coordinate(u)at(0,0.8); \coordinate(d)at(0,-1);
      \coordinate(llu)at(-1.5,1); \coordinate(lu)at(-0.6,1); \coordinate(rru)at(1.5,1); \coordinate(ru)at(0.6,1);
      \coordinate(lld)at(-1.5,-1.2); \coordinate(ld)at(-0.6,-1.2); \coordinate(rrd)at(1.5,-1.2); \coordinate(rd)at(0.6,-1.2);
      %\draw(llu)--(lu)--(u)--(ru)--(rru) (lld)--(ld)--(d)--(rd)--(rrd) (u)--(d);
      \draw[blue,shift={(0,0.1)}](-1.6,0.8)..controls(-1.1,0.8)and(1.1,-0.8)..(1.6,-0.8);
      %\draw[blue,shift={(0,-0.4)}](-1.6,1.1)..controls(-1.1,1.1)and(1.1,-0.4)..(1.6,-0.4);
      %\draw[blue,shift={(0,-0.75)}](-1.6,1.1)..controls(-1.1,1.1)and(1.1,-0.4)..(1.6,-0.4);

      %\draw[blue,shift={(0,0)}](-1.5,0.5)..controls(-0.5,0.5)and(0.5,-0.5)..(1.5,-0.5);
      %\draw[blue,shift={(0,-0.25)}](-1.5,0.5)..controls(-0.5,0.5)and(0.5,-0.5)..(1.5,-0.5);
      %\fill(lu)circle(1mm); \fill(ru)circle(1mm); \fill(ld)circle(1mm); \fill(rd)circle(1mm);
      %\fill(u)circle(1mm); \fill(d)circle(1mm);
      \draw[fill=white](-1.4,0)circle(1.3mm); \draw[fill=white](1.4,0)circle(1.3mm);
      \node[fill=white, inner sep=0.05] at(-1.4,0){\tiny$b$};
      \node[fill=white, inner sep=0.05] at(1.4,0){\tiny$a$};
      \draw (-1.25,0)--(1.25,0);
      \draw[red] (0,1)--(0,-0.9); 
      \draw[fill=black] (0,0.1)circle(0.6mm);
      \node at (0.2,0.2) {\footnotesize$p$};
      \node[red] at (0,-1.12) {$\alpha_e$};
      %\node at(0.2,0.15){$p$};
      %\node[blue] at(-1.9,1.1) {\footnotesize$\gamma_3^a$};
      %\node[blue] at(-1.9,0.75) {\footnotesize$\gamma_2^a$};
      \node[blue] at(0,1.5) {$\gamma^a\underset{e}{\times} \gamma^b$};
      \node[blue] at(-1.9,0.9) {\footnotesize$\gamma^a$};
      \node[blue] at(1.9,-0.7) {\footnotesize$\gamma^b$};
      %\node[blue] at(1.9,-0.7) {\footnotesize$\gamma_2^b$};
      %\node[blue] at(1.9,-1.1) {\footnotesize$\gamma_1^b$};
\end{tikzpicture} 
\caption{A local figure for gluing $\gamma^a$ and $\gamma^b$ along $\alpha_e$.}
\label{fig:glue_arc}
\end{figure}

\begin{proposition}
In the above, $\gamma^a\underset{e}{\times} \gamma^b$ is an arc on $\mathbf{G}$ satisfying 
\begin{equation} \label{eq:glue-g}
   g_d(\gamma^a\underset{e}{\times} \gamma^b) = g_d(\gamma^v)
\end{equation}
for all $d\in G_1^v$ with $v\in \{a,b\}$. 
\end{proposition}

\begin{proof}
   It is easy to check that $\gamma^a\underset{e}{\times} \gamma^b$ provides a $\mathbf{G}$-arc.
   The equality (\ref{eq:glue-g}) is clear from our construction.
\end{proof}

More generally, we have the following. 
For $v\in \{a,b\}$, let $\mathcal{X}^v$ be an (not necessary reduced) admissible collection of $\mathbf{G}^v$-arcs such that $t_a:=g_{e}(\mathcal{X})>0$. By Lemma \ref{lem:abs}, we have a decomposition 
\begin{equation} \label{eq:pairX}
   \mathcal{X}^v=\{\gamma^v_1,\ldots, \gamma^v_{t_v}\} \cup \{\delta^v_{1},\ldots, \delta^v_{m_v}\},
\end{equation}
where $\gamma^v_i$ (resp., $\delta_i^v$) satisfy $g_e(\gamma_i^v)=1$ for all $i\in [1,t_v]$ (resp., $g_e(\delta_i^v)=0$ for all $v\in [1,m_v]$). In this setting, let $p_i^v$ be intersection points of $\gamma^v_i$ and $\alpha_e$ for $i\in [1,t_v]$. 
Reordering the numbering if necessary, we assume that $p^a_{1}, p^a_{2}, \ldots, p^a_{t_a}$ (resp., $p^b_{1},p^b_2, \ldots, p^b_{t_b}$) lie on the side $\alpha_e$ of $\triangle_a^a$ (resp., $\triangle_b^b$) in the counterclockwise (resp., clockwise) direction around $a$ (resp., $b$).

\begin{proposition} \label{prop:glue nat}
Assume that $t:=t_a=t_b$. We set 
\[
   \mathcal{X}^a\underset{e}{\times}\mathcal{X}^b := \{\gamma^a_i \underset{e}{\times}\gamma^b_i\}_{i=1}^{t} \cup \{\delta^a_i\}_{i=1}^{m_a} \cup \{\delta^b_i\}_{i=1}^{m_b}.  
\] 
Then it forms an admissible collection of $\mathbf{G}$-arcs satisfying
\[ 
   g_d(\mathcal{X}^a\underset{e}{\times}\mathcal{X}^b) = g_d(\mathcal{X}^v)
\] 
for all $d\in G^v$ with $v\in \{a,b\}$. 
\end{proposition}

\begin{proof}
   It is clear from our construction. See Figure \ref{fig:glue_natural}.
\end{proof}

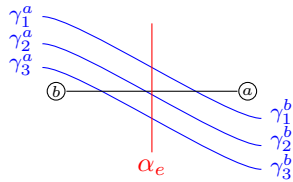
\begin{figure}[htp]   \centering
   \begin{tikzpicture}[scale=0.9]
      \coordinate(u)at(0,0.8); \coordinate(d)at(0,-1);
      \coordinate(llu)at(-1.5,1); \coordinate(lu)at(-0.6,1); \coordinate(rru)at(1.5,1); \coordinate(ru)at(0.6,1);
      \coordinate(lld)at(-1.5,-1.2); \coordinate(ld)at(-0.6,-1.2); \coordinate(rrd)at(1.5,-1.2); \coordinate(rd)at(0.6,-1.2);
      %\draw(llu)--(lu)--(u)--(ru)--(rru) (lld)--(ld)--(d)--(rd)--(rrd) (u)--(d);
      \draw[blue,shift={(0,0)}](-1.6,1.1)..controls(-1.1,1.1)and(1.1,-0.4)..(1.6,-0.4);
      \draw[blue,shift={(0,-0.4)}](-1.6,1.1)..controls(-1.1,1.1)and(1.1,-0.4)..(1.6,-0.4);
      \draw[blue,shift={(0,-0.75)}](-1.6,1.1)..controls(-1.1,1.1)and(1.1,-0.4)..(1.6,-0.4);

      %\draw[blue,shift={(0,0)}](-1.5,0.5)..controls(-0.5,0.5)and(0.5,-0.5)..(1.5,-0.5);
      %\draw[blue,shift={(0,-0.25)}](-1.5,0.5)..controls(-0.5,0.5)and(0.5,-0.5)..(1.5,-0.5);
      %\fill(lu)circle(1mm); \fill(ru)circle(1mm); \fill(ld)circle(1mm); \fill(rd)circle(1mm);
      %\fill(u)circle(1mm); \fill(d)circle(1mm);
      \draw[fill=white](-1.4,0)circle(1.3mm); \draw[fill=white](1.4,0)circle(1.3mm);
      \node[fill=white, inner sep=0.05] at(-1.4,0){\tiny$b$};
      \node[fill=white, inner sep=0.05] at(1.4,0){\tiny$a$};
      \draw (-1.25,0)--(1.25,0);
      \draw[red] (0,1)--(0,-0.9); 
      \node[red] at (0,-1.12) {$\alpha_e$};
      %\node at(0.2,0.15){$p$};
      \node[blue] at(-1.9,1.1) {\footnotesize$\gamma_1^a$};
      \node[blue] at(-1.9,0.75) {\footnotesize$\gamma_2^a$};
      \node[blue] at(-1.9,0.4) {\footnotesize$\gamma_3^a$};
      \node[blue] at(1.9,-0.3) {\footnotesize$\gamma_1^b$};
      \node[blue] at(1.9,-0.7) {\footnotesize$\gamma_2^b$};
      \node[blue] at(1.9,-1.1) {\footnotesize$\gamma_3^b$};
   \end{tikzpicture} 
\caption{A local figure for gluing $\{\gamma_1^a,\gamma_2^a,\gamma_3^a\}$ and $\{\gamma_1^b,\gamma_2^b,\gamma_3^b\}$ with $t_a = t_b = 3$.}
\label{fig:glue_natural}
\end{figure}

Next, we would like to construct a collection of pairwise distinct and pairwise admissible $\mathbf{G}$-arcs as many as possible from a given pair of reduced collections on $\mathbf{G}^v$. This is achieved by using combinatorial objects called lattice paths.

\begin{definition} \label{def:latticepath}
Let $s,t$ be positive integers.
For two elements $(j_1,k_1)$ and $(j_2,k_2)$ in $[1,s]\times[1,t]$, we say that they are \textit{compatible} if one of (i) $j_1\geq j_2$ and $k_1\geq k_2$ or (ii) $j_1\leq j_2$ and $k_2\leq k_1$ holds. A \textit{lattice path} of $[1,s]\times[1,t]$ is a maximal set of pairwise compatible elements of $[1,s]\times[1,t]$. 
\end{definition}
   
We denote by $\mathrm{P}(s,t)$ the set of lattice paths of $[1,s]\times[1,t]$. 
It is well-known that the cardinality of $\mathrm{P}(s,t)$ is $\binom{s+t-2}{s-1}=\binom{s+t-2}{t-1}$. In addition, every $P\in \mathrm{P}(s,t)$ consists of precisely $s+t-1$ elements and contains $(1,1)$ and $(s,t)$ from its maximality. We write 
\[
   P(j,-):=\#\{(j,b)\in P \mid b\in [1,t]\} \quad \text{and} \quad 
   P(-,k):=\#\{(a,k)\in P \mid a\in [1,s]\} 
\]
for $j\in [1,s]$ and $k\in [1,t]$ respectively.

\begin{proposition} \label{prop:gluepath}
Assume that $\mathcal{X}^v$ in (\ref{eq:pairX}) is reduced for $v\in \{a,b\}$. 
For each lattice path $P\in \mathrm{P}(t_a,t_b)$, we set
\[
   \mathcal{X}^a\underset{e_P}{\times} \mathcal{X}^b:= \{\gamma^a_j\underset{e}{\times} \gamma^b_k \mid (j,k)\in P\} \cup \{\delta^a_i\}_{i=1}^{m_a} \cup \{\delta^b_i\}_{i=1}^{m_b}. 
\]    
Then it forms a reduced collection of $\mathbf{G}$-arcs whose $g$-vector is given by
\begin{equation} \label{eq:gglue}
   g_d(\mathcal{X}^a\underset{e_P}{\times} \mathcal{X}^b) = 
   \begin{cases}
      \displaystyle \sum_{j=1}^{t_a}g_d(\gamma^a_j)P(j,-) + \sum_{i=1}^{m_a}g_d(\delta_i^a) & \text{if $d\in G_1^a$}, \\
      \displaystyle \sum_{k=1}^{t_b}g_d(\gamma^b_k)P(-,k) + \sum_{i=1}^{m_b}g_d(\delta_i^b) & \text{if $d\in G_1^b$}. 
   \end{cases}
\end{equation}
Furthermore, it is complete if both $\mathcal{X}^a,\mathcal{X}^b$ are complete on $\mathbf{G}^a, \mathbf{G}^b$ respectively.
\end{proposition}

\begin{proof}
   Consider collections 
   \begin{eqnarray}
      \mathcal{X}^{a(P)} &:=& \{\gamma_1^{a(1)},\ldots,\gamma_1^{a(P(1,-))},\ldots, \gamma_{t_a}^{a(1)},\ldots,\gamma_{t_a}^{a(P(t_a,-))}\} \cup \{\delta_1^a,\ldots, \delta_{m_a}^a\}, \\
      \mathcal{X}^{b(P)} &:=& \{\gamma_1^{b(1)},\ldots,\gamma_1^{b(P(-,1))},\ldots, \gamma_{t_b}^{b(1)},\ldots,\gamma_{t_b}^{b(P(-,t_b))}\}  \cup \{\delta_1^b,\ldots, \delta_{m_b}^b\},
   \end{eqnarray}
   where $\gamma_j^{a(x)}$ (resp., $\gamma_k^{b(y)}$) are copies of $\gamma_j^a$ for all $j\in [1,t_a]$ and $x\in [1,P(j,-)]$ (resp., $\gamma_k^b$ for all $k\in [1,t_b]$ and $y\in [1,P(-,k)]$).  
   Since $g_e(\mathcal{X}^{a(P)})= g_e(\mathcal{X}^{b(P)})=\#P=t_a+t_b-1$ holds, they provide a collection $\mathcal{X}^{a(P)} \underset{e}{\times} \mathcal{X}^{b(P)}$ by Proposition \ref{prop:glue nat}. 
   On the other hand, from the ordering of intersection points on $\alpha_e$, one can see that $\gamma_j^a\underset{e}{\times}\gamma_k^b$ lies in $\mathcal{X}^{a(P)} \underset{e}{\times} \mathcal{X}^{b(P)}$ if and only if $(j,k)\in P$. That is, 
   \[
      \mathcal{X}^{a(P)}\underset{e}{\times}\mathcal{X}^{b(P)} = \mathcal{X}^a\underset{e_P}{\times}\mathcal{X}^b.
   \]
   By our construction, this is a reduced collection of arcs whose $g$-vector is given by (\ref{eq:gglue}). 

   Now, we assume that $\mathcal{X}^v$ is complete on $\mathbf{G}^v$ for $v\in \{a,b\}$. 
   Then a collection $\mathcal{X}^a\underset{e_P}{\times}\mathcal{X}^b$ is complete by Proposition \ref{prop:partialarc} since 
   \[
      \#(\mathcal{X}^a\underset{e_P}{\times}\mathcal{X}^b) = (\#\mathcal{X}^a-t_a) + (\#\mathcal{X}^b-t_b) + (t_a+t_b-1) = |\mathbf{G}^a| + |\mathbf{G}^b| - 1 =  |\mathbf{G}|. 
   \]
   Here, we use equations $\#\mathcal{X}^v = |\mathbf{G}^v|$ for $v\in \{a,b\}$ by Proposition \ref{prop:partialarc}. It finishes a proof.
\end{proof}

\begin{figure}[htp]   \centering
   \renewcommand{\arraystretch}{1.5}
   \setlength{\tabcolsep}{1mm} 
   \begin{tabular}{cccccccc}
      \begin{tikzpicture}[baseline=0mm,scale=0.85]
         \coordinate(u)at(0,0.8); \coordinate(d)at(0,-1);
         \coordinate(llu)at(-1.5,1); \coordinate(lu)at(-0.6,1); \coordinate(rru)at(1.5,1); \coordinate(ru)at(0.6,1);
         \coordinate(lld)at(-1.5,-1.2); \coordinate(ld)at(-0.6,-1.2); \coordinate(rrd)at(1.5,-1.2); \coordinate(rd)at(0.6,-1.2);
         %\draw(llu)--(lu)--(u)--(ru)--(rru) (lld)--(ld)--(d)--(rd)--(rrd) (u)--(d);
         \draw[blue,shift={(0,0)}](-1.6,1.1)..controls(-1.1,1.1)and(1.1,-0.75)..(1.6,-0.75);
         \draw[blue,shift={(0,0.35)}](-1.6,1.1)..controls(-1.1,1.1)and(1.1,-0.75)..(1.6,-0.75);
         \draw[blue,shift={(0,-0.35)}](-1.6,1.1)..controls(-1.1,1.1)and(1.1,-0.75)..(1.6,-0.75);
         \draw[blue,shift={(0,-0.7)}](-1.6,1.1)..controls(-1.1,1.1)and(1.1,-0.75)..(1.6,-0.75);
         
         %\draw[blue,shift={(0,-0.35)}](-1.6,1.1)..controls(-1,1)and(-0.2,0.3)..(0,0.2);
         %\draw[blue,shift={(0,-0.35)}](1.6,-0.4)..controls(1,-0.3)and(0.2,0.1)..(0,0.2);
         %\draw[blue,shift={(0,-0.7)}](-1.6,1.1)..controls(-1,1)and(-0.2,0.3)..(0,0.2);
         %\draw[blue,shift={(0,-0.7)}](1.6,-0.4)..controls(1,-0.3)and(0.2,0.1)..(0,0.2);

         %\draw[blue,shift={(0,0)}](-1.5,0.5)..controls(-0.5,0.5)and(0.5,-0.5)..(1.5,-0.5);
         %\draw[blue,shift={(0,-0.25)}](-1.5,0.5)..controls(-0.5,0.5)and(0.5,-0.5)..(1.5,-0.5);
         %\fill(lu)circle(1mm); \fill(ru)circle(1mm); \fill(ld)circle(1mm); \fill(rd)circle(1mm);
         %\fill(u)circle(1mm); \fill(d)circle(1mm);
         \draw[fill=white](-1.4,0)circle(1.3mm); \draw[fill=white](1.4,0)circle(1.3mm);
         \node[fill=white, inner sep=0.05] at(-1.4,0){\tiny$b$};
         \node[fill=white, inner sep=0.05] at(1.4,0){\tiny$a$};
         \draw (-1.25,0)--(1.25,0);
         \draw[red] (0,1)--(0,-0.9); 
         \node[red] at (0,-1.12) {\footnotesize$\alpha_e$};
         %\node at(0.2,0.15){$p$};

         \node[blue] at(-1.9,1.45) {\tiny$\gamma_1^a$};
         \node[blue] at(-1.9,1.1) {\tiny$\gamma_1^a$};
         \node[blue] at(-1.9,0.75) {\tiny$\gamma_2^a$};
         \node[blue] at(-1.9,0.4) {\tiny$\gamma_3^a$};
      
         \node[blue] at(1.9,-0.3) {\tiny$\gamma_1^b$};
         \node[blue] at(1.9,-0.7) {\tiny$\gamma_2^b$};
         \node[blue] at(1.9,-1.1) {\tiny$\gamma_2^b$};
         \node[blue] at(1.9,-1.5) {\tiny$\gamma_2^b$};
         
      \end{tikzpicture} 
      &
      \begin{tikzpicture}[baseline=0mm,scale=0.85]
         \coordinate(u)at(0,0.8); \coordinate(d)at(0,-1);
         \coordinate(llu)at(-1.5,1); \coordinate(lu)at(-0.6,1); \coordinate(rru)at(1.5,1); \coordinate(ru)at(0.6,1);
         \coordinate(lld)at(-1.5,-1.2); \coordinate(ld)at(-0.6,-1.2); \coordinate(rrd)at(1.5,-1.2); \coordinate(rd)at(0.6,-1.2);
         %\draw(llu)--(lu)--(u)--(ru)--(rru) (lld)--(ld)--(d)--(rd)--(rrd) (u)--(d);
         \draw[blue,shift={(0,0)}](-1.6,1.1)..controls(-1.1,1.1)and(1.1,-0.75)..(1.6,-0.75);
         \draw[blue,shift={(0,0.35)}](-1.6,1.1)..controls(-1.1,1.1)and(1.1,-0.75)..(1.6,-0.75);
         \draw[blue,shift={(0,-0.35)}](-1.6,1.1)..controls(-1.1,1.1)and(1.1,-0.75)..(1.6,-0.75);
         \draw[blue,shift={(0,-0.7)}](-1.6,1.1)..controls(-1.1,1.1)and(1.1,-0.75)..(1.6,-0.75);
         
         %\draw[blue,shift={(0,-0.35)}](-1.6,1.1)..controls(-1,1)and(-0.2,0.3)..(0,0.2);
         %\draw[blue,shift={(0,-0.35)}](1.6,-0.4)..controls(1,-0.3)and(0.2,0.1)..(0,0.2);
         %\draw[blue,shift={(0,-0.7)}](-1.6,1.1)..controls(-1,1)and(-0.2,0.3)..(0,0.2);
         %\draw[blue,shift={(0,-0.7)}](1.6,-0.4)..controls(1,-0.3)and(0.2,0.1)..(0,0.2);

         %\draw[blue,shift={(0,0)}](-1.5,0.5)..controls(-0.5,0.5)and(0.5,-0.5)..(1.5,-0.5);
         %\draw[blue,shift={(0,-0.25)}](-1.5,0.5)..controls(-0.5,0.5)and(0.5,-0.5)..(1.5,-0.5);
         %\fill(lu)circle(1mm); \fill(ru)circle(1mm); \fill(ld)circle(1mm); \fill(rd)circle(1mm);
         %\fill(u)circle(1mm); \fill(d)circle(1mm);
         \draw[fill=white](-1.4,0)circle(1.3mm); \draw[fill=white](1.4,0)circle(1.3mm);
         \node[fill=white, inner sep=0.05] at(-1.4,0){\tiny$b$};
         \node[fill=white, inner sep=0.05] at(1.4,0){\tiny$a$};
         \draw (-1.25,0)--(1.25,0);
         \draw[red] (0,1)--(0,-0.9); 
         \node[red] at (0,-1.12) {\footnotesize$\alpha_e$};
         %\node at(0.2,0.15){$p$};

         \node[blue] at(-1.9,1.45) {\tiny$\gamma_1^a$};
         \node[blue] at(-1.9,1.1) {\tiny$\gamma_2^a$};
         \node[blue] at(-1.9,0.75) {\tiny$\gamma_2^a$};
         \node[blue] at(-1.9,0.4) {\tiny$\gamma_3^a$};
      
         \node[blue] at(1.9,-0.3) {\tiny$\gamma_1^b$};
         \node[blue] at(1.9,-0.7) {\tiny$\gamma_1^b$};
         \node[blue] at(1.9,-1.1) {\tiny$\gamma_2^b$};
         \node[blue] at(1.9,-1.5) {\tiny$\gamma_2^b$};
         
      \end{tikzpicture}  &
      \begin{tikzpicture}[baseline=0mm,scale=0.85]
         \coordinate(u)at(0,0.8); \coordinate(d)at(0,-1);
         \coordinate(llu)at(-1.5,1); \coordinate(lu)at(-0.6,1); \coordinate(rru)at(1.5,1); \coordinate(ru)at(0.6,1);
         \coordinate(lld)at(-1.5,-1.2); \coordinate(ld)at(-0.6,-1.2); \coordinate(rrd)at(1.5,-1.2); \coordinate(rd)at(0.6,-1.2);
         %\draw(llu)--(lu)--(u)--(ru)--(rru) (lld)--(ld)--(d)--(rd)--(rrd) (u)--(d);
         \draw[blue,shift={(0,0)}](-1.6,1.1)..controls(-1.1,1.1)and(1.1,-0.75)..(1.6,-0.75);
         \draw[blue,shift={(0,0.35)}](-1.6,1.1)..controls(-1.1,1.1)and(1.1,-0.75)..(1.6,-0.75);
         \draw[blue,shift={(0,-0.35)}](-1.6,1.1)..controls(-1.1,1.1)and(1.1,-0.75)..(1.6,-0.75);
         \draw[blue,shift={(0,-0.7)}](-1.6,1.1)..controls(-1.1,1.1)and(1.1,-0.75)..(1.6,-0.75);
         
         %\draw[blue,shift={(0,-0.35)}](-1.6,1.1)..controls(-1,1)and(-0.2,0.3)..(0,0.2);
         %\draw[blue,shift={(0,-0.35)}](1.6,-0.4)..controls(1,-0.3)and(0.2,0.1)..(0,0.2);
         %\draw[blue,shift={(0,-0.7)}](-1.6,1.1)..controls(-1,1)and(-0.2,0.3)..(0,0.2);
         %\draw[blue,shift={(0,-0.7)}](1.6,-0.4)..controls(1,-0.3)and(0.2,0.1)..(0,0.2);

         %\draw[blue,shift={(0,0)}](-1.5,0.5)..controls(-0.5,0.5)and(0.5,-0.5)..(1.5,-0.5);
         %\draw[blue,shift={(0,-0.25)}](-1.5,0.5)..controls(-0.5,0.5)and(0.5,-0.5)..(1.5,-0.5);
         %\fill(lu)circle(1mm); \fill(ru)circle(1mm); \fill(ld)circle(1mm); \fill(rd)circle(1mm);
         %\fill(u)circle(1mm); \fill(d)circle(1mm);
         \draw[fill=white](-1.4,0)circle(1.3mm); \draw[fill=white](1.4,0)circle(1.3mm);
         \node[fill=white, inner sep=0.05] at(-1.4,0){\tiny$b$};
         \node[fill=white, inner sep=0.05] at(1.4,0){\tiny$a$};
         \draw (-1.25,0)--(1.25,0);
         \draw[red] (0,1)--(0,-0.9); 
         \node[red] at (0,-1.12) {\footnotesize$\alpha_e$};
         %\node at(0.2,0.15){$p$};

         \node[blue] at(-1.9,1.45) {\tiny$\gamma_1^a$};
         \node[blue] at(-1.9,1.1) {\tiny$\gamma_2^a$};
         \node[blue] at(-1.9,0.75) {\tiny$\gamma_3^a$};
         \node[blue] at(-1.9,0.4) {\tiny$\gamma_3^a$};
      
         \node[blue] at(1.9,-0.3) {\tiny$\gamma_1^b$};
         \node[blue] at(1.9,-0.7) {\tiny$\gamma_1^b$};
         \node[blue] at(1.9,-1.1) {\tiny$\gamma_1^b$};
         \node[blue] at(1.9,-1.5) {\tiny$\gamma_2^b$};
         
      \end{tikzpicture}  \\ 
      \begin{tikzpicture}[baseline=0mm,scale=0.85]
         \draw[dotted] (-1.5,0)--(-1.5,-1.5)--(0,-1.5)--(1.5,-1.5)--(1.5,0)--(-1.5,0) (0,0)--(0,-1.5);
         \draw[fill=white](-1.5,0)circle(1mm); 
         \draw[fill=white](0,0)circle(1mm);
         \draw[fill=white](1.5,0)circle(1mm);
         \draw[fill=white](-1.5,-1.5)circle(1mm);
         \draw[fill=white](0,-1.5)circle(1mm);
         \draw[fill=white](1.5,-1.5)circle(1mm); 
         \draw[very thick] (-1.5,-1.5)--(-1.5,0)--(0,0)--(1.5,0);
         \node at(-1.5,-1.9){\small{$1$}};
         \node at(0,-1.9){\small{$2$}};
         \node at(1.5,-1.9){\small{$3$}};
         \node at(-1.9,-1.5){\small{$1$}};
         \node at(-1.9,0){\small{$2$}};
      \end{tikzpicture} &
      \begin{tikzpicture}[baseline=0mm,scale=0.85]
         \draw[dotted] (-1.5,0)--(-1.5,-1.5)--(0,-1.5)--(1.5,-1.5)--(1.5,0)--(-1.5,0) (0,0)--(0,-1.5);
         \draw[fill=white](-1.5,0)circle(1mm); 
         \draw[fill=white](0,0)circle(1mm);
         \draw[fill=white](1.5,0)circle(1mm);
         \draw[fill=white](-1.5,-1.5)circle(1mm);
         \draw[fill=white](0,-1.5)circle(1mm);
         \draw[fill=white](1.5,-1.5)circle(1mm); 
         \draw[very thick] (-1.5,-1.5)--(0,-1.5)--(0,0)--(1.5,0);
         \node at(-1.5,-1.9){\small{$1$}};
         \node at(0,-1.9){\small{$2$}};
         \node at(1.5,-1.9){\small{$3$}};
         \node at(-1.9,-1.5){\small{$1$}};
         \node at(-1.9,0){\small{$2$}};
      \end{tikzpicture} &
      \begin{tikzpicture}[baseline=0mm,scale=0.85]
         \draw[dotted] (-1.5,0)--(-1.5,-1.5)--(0,-1.5)--(1.5,-1.5)--(1.5,0)--(-1.5,0) (0,0)--(0,-1.5);
         \draw[fill=white](-1.5,0)circle(1mm); 
         \draw[fill=white](0,0)circle(1mm);
         \draw[fill=white](1.5,0)circle(1mm);
         \draw[fill=white](-1.5,-1.5)circle(1mm);
         \draw[fill=white](0,-1.5)circle(1mm);
         \draw[fill=white](1.5,-1.5)circle(1mm); 
         \draw[very thick] (-1.5,-1.5)--(0,-1.5)--(1.5,-1.5)--(1.5,0);
         \node at(-1.5,-1.9){\small{$1$}};
         \node at(0,-1.9){\small{$2$}};
         \node at(1.5,-1.9){\small{$3$}};
         \node at(-1.9,-1.5){\small{$1$}};
         \node at(-1.9,0){\small{$2$}};
      \end{tikzpicture} &
   
   \end{tabular}
   \caption{Gluing $\{\gamma_1^a,\gamma_2^a,\gamma_3^a\}$ and $\{\gamma_1^b,\gamma_2^b\}$ via lattice paths of $\{1,2,3\}\times\{1,2\}$.}
   \label{Part4:f_gluing example}
   \end{figure}
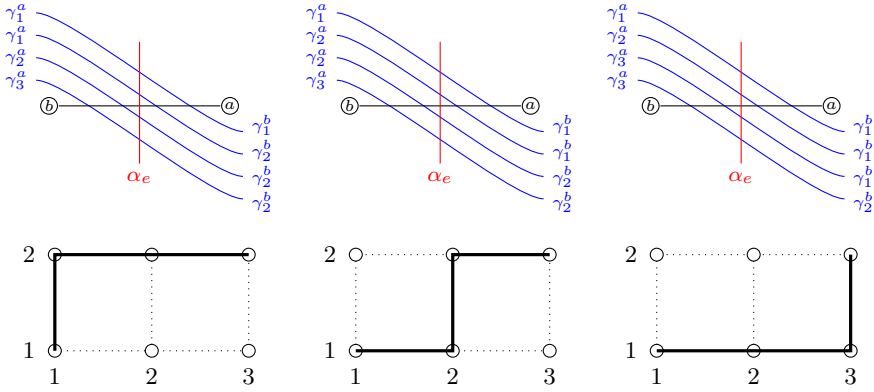

By the bijection $(-)^{\rm op}$ in  Proposition \ref{prop:opp}, $\mathcal{X}^a\underset{e_P}{\times} \mathcal{X}^b$ provides the complete collection on $\mathbf{G}^{\rm op}$ whose $g$-vector is given as $-g(\mathcal{X}^a\underset{e_P}{\times} \mathcal{X}^b)$. It commutes the following diagram: 
\[
   \xymatrix{
      (\mathcal{X}^a,\mathcal{X}^b,P) \ar@{|->}[r] \ar@{|->}[d]& \mathcal{X}^a\underset{e_P}{\times} \mathcal{X}^b \ar@{|->}[d] \\ 
      ((\mathcal{X}^a)^{\rm op}, (\mathcal{X}^b)^{\rm op},P) \ar@{|->}[r] & (\mathcal{X}^a)^{\rm op}\underset{e_P}{\times}(\mathcal{X}^b)^{\rm op}:= (\mathcal{X}^a\underset{e_P}{\times} \mathcal{X}^b)^{\rm op}. 
   }
\]

\begin{theorem} \label{thm:gluing}
   Let $n_v:=|\mathbf{G}^v|$ for $v\in \{a,b\}$.
   For any $s\in [1,n_a]$ and any $t\in [1,n_b]$, we have injective maps 
   \begin{eqnarray} 
      \rho_e^{s,t} \colon &\mathcal{A}(\mathbf{G}^a)_e^s\times \mathcal{A}(\mathbf{G}^b)_e^t \times \mathrm{P}(s,t) &\to \mathcal{A}(\mathbf{G})_e^{s+t-1} \quad \text{and} \label{eq:pos}\\ 
      \rho_e^{-s,-t} \colon &\mathcal{A}(\mathbf{G}^a)_e^{-s}\times \mathcal{A}(\mathbf{G}^b)_e^{-t} \times \mathrm{P}(s,t) &\to \mathcal{A}(\mathbf{G})_e^{-s-t+1} \label{eq:neg}
   \end{eqnarray}
   mapping $(\mathcal{X},\mathcal{Y},P)\mapsto \mathcal{X}\underset{e_P}{\times} \mathcal{Y}$. 
   Moreover, for each $j\in [1,n]$, we have decompositions
   \begin{equation} \label{eq:decomp}
      \mathcal{A}(\mathbf{G})_e^{j} = \bigsqcup_{\substack{s\in [1,n_a] \\ t\in [1,n_b]\\ j=s+t-1}} \Image \rho_e^{s,t} \quad \text{and} \quad 
      \mathcal{A}(\mathbf{G})_e^{-j} = \bigsqcup_{\substack{ s \in [1,n_a] \\ t \in [1,n_b]\\ j=s+t-1}} \Image \rho_e^{-s,-t}
   \end{equation}
\end{theorem}

\begin{proof}
   By Proposition \ref{prop:gluepath}, we have maps (\ref{eq:pos}) and (\ref{eq:neg}), both of which are injective by Lemma \ref{lem:g-invarc}.  
   For (\ref{eq:decomp}), it is enough to show that every complete collection $\mathcal{X}$ of $\mathbf{G}$-arcs can be written as $\mathcal{X}^a\underset{e_P}{\times}\mathcal{X}^b$. 

   Let $v\in \{a,b\}$. For a $\mathbf{G}$-arc $\gamma$, we define a curve $\gamma|_{\mathbf{G}_v}$ as follows: If $g_d(\gamma)\neq 0$ for some $d\in G_1^v$, let $w_{\gamma}:=(w=(e_1,\ldots,e_k),\epsilon)$ be the signed walk corresponding to $\gamma$. We have the maximum subsequence $e_i,e_{i+1}, \ldots, e_{i+j}$ of $w$ consisting of edges of $\mathbf{G}^v$, and it gives rise to the signed walk on $\mathbf{G}^v$ by a signature $\epsilon$. 
   Then, let $\gamma|_{\mathbf{G}_v}$ be the arc corresponding to this signed walk. Namely, it is characterized by its $g$-vector 
   \[
      g_d(\gamma|_{\mathbf{G}^v}) = 
      \begin{cases} 
         g_d(\gamma) & \text{if $d\in G_1^v$} \\
         0 & \text{else.} 
      \end{cases}
   \]

   For a complete collection $\mathcal{X}$ of $\mathbf{G}$-arcs, it is easy to see that the collection
   \[
      \{\gamma|_{\mathbf{G}_v} \mid \text{$\gamma\in \mathcal{X}$, $g_d(\gamma)\neq0$ for some $d\in G^v_1$}\}
   \]
   forms an admissible collection of $\mathbf{G}^v$-arcs. Furthermore, it clearly provides a complete collection $\mathcal{X}|_{\mathbf{G}_v}$ on $\mathbf{G}^v$ from its maximality. 
   Under this notation, one can find a lattice path $P\in \mathrm{P}(t_a,t_b)$ such that    
   \[
      \mathcal{X}= (\mathcal{X}|_{\mathbf{G}_v})\underset{e_P}{\times}(\mathcal{X}|_{\mathbf{G}_v}),
   \]
   where $t_v:=g_e(\mathcal{X}|_{\mathbf{G}_v})$ for $v\in \{a,b\}$. It finishes a proof.
\end{proof}

\section{Proof of a main theorem} \label{sec:proof}

In this section, we prove the following main result in this paper. 

\begin{theorem} \label{thm:main}
Let $\mathbf{G}$ be a plane tree and $n:=|\mathbf{G}|$ the number of edges of $\mathbf{G}$. 
\begin{enumerate}
   \item For any external edge $e$ of $\mathbf{G}$ and non-negative integer $j$, we have 
   \begin{equation} \label{eq:formula}
      \#\mathcal{A}(\mathbf{G})_e^j = \#\mathcal{A}(\mathbf{G})_e^{-j} = 
      \begin{cases}
         \binom{2n-j+1}{n-1} & \text{if $j\in [1,n]$,} \\ 
         0 & \text{otherwise.}
      \end{cases}
   \end{equation}
   \item The following equation holds: 
   \begin{equation}
      \#\mathcal{A}(\mathbf{G}) = \binom{2n}{n}.
   \end{equation}
\end{enumerate} 
\end{theorem}

\subsection{Proof of Theorem \ref{thm:main}}
The claim (2) of Theorem \ref{thm:main} is obvious from (1). 
In the following, we show (1) by induction on the number of edges.  

Let $\mathbf{G}=(G_0,G_1,s,\sigma)$ be a plane tree having $n=|\mathbf{G}|$ edges. 
Let $\mathbf{L}$ be the dual of $\mathbf{G}$, embedded into the sphere $\mathbb{S}$. 

If $n=1$, then the assertion is clear. So, we assume $n>1$. We prepare the following induction hypothesis: 
\begin{center}
($\ast$) The equation (\ref{eq:formula}) holds for any plane tree having less than $n$ edges.
\end{center}
Fix an external edge $e$ of $\mathbf{G}$. Suppose that $e$ has $a,b$ as its endpoints with external vertex $a$. 
Let $f:=\sigma_b(e)$ and $c$ its another endpoint.  
The edge $f$ determines two plane subtrees $\mathbf{G}^b, \mathbf{G}^c$ of $\mathbf{G}$ satisfying $G_1^b\cap G_1^c = \{f\}$ and $G_1^b\cup G_1^c = G_1$. Without loss of generality, for $v\in \{b,c\}$, we may assume that $\mathbf{G}^v$ has the edge $f$ as its external edge with external vertex $v$. In this case, $\mathbf{G}^c$ has the edge $e$ as its external edge with external vertex $a$. 

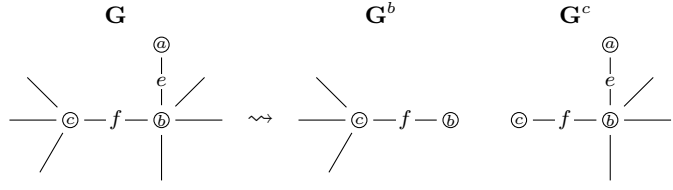
\[
   \renewcommand{\arraystretch}{1}
   \setlength{\tabcolsep}{1mm}
   \begin{tabular}{ccccccccc}
      \small$\mathbf{G}$ & &\small$\mathbf{G}^b$ & \small$\mathbf{G}^c$ \\
      \begin{tikzpicture}
         %\draw[red] (0,0.6)--(0,-0.7);
         %\node(al) at (0,-0.85) {\color{red}\footnotesize$\alpha_e$};
         \node(b) at (0.6,0) {\tiny$b$};
         \node(c) at (-0.6,0) {\tiny$c$};
         \node(a) at (0.6,1) {\tiny$a$};
         \draw (b)--node[fill=white,inner sep=1]{\footnotesize$f$}(c); 
         \draw (b)--($(b)+(45:0.8)$) (b)--($(b)+(0:0.8)$) (b)--($(b)+(-90:0.8)$);
         \draw (b)--node[fill=white,inner sep=1]{\footnotesize$e$}(a);
         \draw (c)--($(c)+(135:0.8)$) (c)--($(c)+(180:0.8)$) (c)--($(c)+(-120:0.8)$);

         \draw (b)circle(1mm);
         \draw (c)circle(1mm);
         \draw (a)circle(1mm);
   
         \node(xx) at (0,-0.8) {};
         \node(xx) at (0,0.8) {};
         
      \end{tikzpicture}
      &
   \begin{tikzpicture}
      \node(sq) at (0,0) {$\rightsquigarrow$};
      \node(xx) at (0,-0.8) {};
         \node(xx) at (0,0.8) {};
   \end{tikzpicture}
      &
      \begin{tikzpicture}
         %\draw[red] (0,0.6)--(0,-0.7);
         %\node(al) at (0,-0.85) {\color{red}\footnotesize$\alpha_e$};
         \node(b) at (0.6,0) {\tiny$b$};
         \node(c) at (-0.6,0) {\tiny$c$};
         %\node(a) at (-1.8,0) {\tiny$a$};
         \draw (b)--node[fill=white,inner sep=1]{\footnotesize$f$}(c);
         %\draw (c)--node[fill=white,inner sep=1]{\footnotesize$e$}(a);
         \draw (c)--($(c)+(135:0.8)$) (c)--($(c)+(180:0.8)$) (c)--($(c)+(-120:0.8)$);

         \draw (b)circle(1mm);
         \draw (c)circle(1mm);
         %\draw (a)circle(1mm);
   
         \node(xx) at (0,-0.8) {};
         \node(xx) at (0,0.8) {};
         
      \end{tikzpicture}
      & \quad
      \begin{tikzpicture}
         \coordinate(bl) at (-0.2,0.6);
         \node(a) at (0.6,0) {\tiny$b$};
         \node(b) at (-0.6,0) {\tiny$c$};
         \draw (a)--node[fill=white,inner sep=1]{\footnotesize$f$}(b); 
         \draw (a)--($(a)+(45:0.8)$) (a)--($(a)+(0:0.8)$) (a)--($(a)+(-90:0.8)$);
         %\draw (b)--($(b)+(120:0.8)$) (b)--($(b)+(180:0.8)$) (b)--($(b)+(-120:0.8)$);
         \node(aa) at (0.6,1) {\tiny$a$};
         \draw (a)--node[fill=white,inner sep=1]{\footnotesize$e$}(aa); 
         \draw (a)circle(1mm);
         \draw (b)circle(1mm);
         \draw (aa)circle(1mm);
         \node(xx) at (0,-0.8) {};
         \node(xx) at (0,0.8) {};
      
      \end{tikzpicture}
   \end{tabular}
   \]

We begin with the following observation.
Now, let $\mathbf{L}^v$ be the dual of $\mathbf{G}^v$, embedded into the sphere $\mathbb{S}^v$. Let $n_v:=|\mathbf{G}^v|$ for $v\in \{b,c\}$.

\begin{lemma} \label{lem:gamma0}
   Suppose that $\gamma_0$ is the $\mathbf{G}^c$-arc such that $g_e(\gamma_0)=-g_f(\gamma_0)=1$ and $g_d(\gamma_0)=0$ for all $d \neq e,f$. The following hold. 
   \begin{enumerate}
      \item There are no $\mathbf{G}^c$-arcs $\gamma\neq \gamma_0$ such that $g_e(\gamma)>0$ and $g_f(\gamma)\neq 0$.
      \item If $\mathcal{X}$ is a complete collection of $\mathbf{G}^c$-arcs satisfying $g_e(\mathcal{X})>0$ and $g_f(\mathcal{X})<0$, then $\gamma_0\in \mathcal{X}$. 
   \end{enumerate} 
\end{lemma}

\begin{proof}
   There is a walk $w:=(e,f)$ on $\mathbf{G}^c$ and this is only the walk having both $e$ and $f$. The walk $w$ gives rise to precisely two arcs $\gamma_0$ and $\delta_0$ with $g(\gamma_0)=-g(\delta_0)$ by Proposition \ref{prop:sw}. Thus, we have (1). 
   On the other hand, since $f=\sigma_b(e)$, the arc $\gamma_0$ does not intersect with any arc $\gamma$ whose $g$-vector is such that $g_e(\gamma)\geq 0$ and $g_f(\gamma)\leq0$. This means that every complete collection $\mathcal{X}$ in (2) contains $\gamma_0$ from its maximality. 
\end{proof}

Now, for integers $t_1,t_2\in \mathbb{Z}$, let 
\[
   \mathcal{A}(\mathbf{G}^c)_{e,f}^{t_1,t_2}:=\mathcal{A}(\mathbf{G}^c)_{e}^{t_1}\cap \mathcal{A}(\mathbf{G}^c)_{f}^{t_2}. 
\]
Similarly, let 
\[
   \mathcal{A}(\mathbf{G}^c)_{e,f}^{>0,>0}:=\mathcal{A}(\mathbf{G}^c)_{e}^{>0}\cap \mathcal{A}(\mathbf{G}^c)_{f}^{>0}   \quad \text{and} \quad
   \mathcal{A}(\mathbf{G}^c)_{e,f}^{>0,<0}:=\mathcal{A}(\mathbf{G}^c)_{e}^{>0}\cap \mathcal{A}(\mathbf{G}^c)_{f}^{<0}. 
\]

\begin{lemma} \label{lem:Gc}
The set $\mathcal{A}(\mathbf{G}^c)_e^{>0}$ is a disjoint union of subsets of the following forms:  
\begin{enumerate}
   \item[(a)] $\mathcal{A}(\mathbf{G}^c)_{e,f}^{t+1,-1}$; 
   \item[(b)] $\mathcal{A}(\mathbf{G}^c)_{e,f}^{1,-t-1}$;   
   \item[(c)] $\mathcal{A}(\mathbf{G}^c)_{e,f}^{t-u+1,u}$. 
\end{enumerate}
Here, $t,u$ run over all integers $t,u\in [1,n_c-1]$ with $u\leq t$.
\end{lemma}

\begin{proof}
   Trivially, we have 
   \[
      \mathcal{A}(\mathbf{G}^c)_e^{>0} = \mathcal{A}(\mathbf{G}^c)_{e,f}^{>0,<0} \sqcup \mathcal{A}(\mathbf{G}^c)_{e,f}^{>0,>0}. 
   \]
   First, we show that $\mathcal{A}(\mathbf{G}^c)_{e,f}^{>0,<0}$ is decomposed in to subsets of the forms (a) and (b). 
   Let $\mathcal{X}$ be its element. By Lemma \ref{lem:gamma0}, it contains the arc $\gamma_0$ in Lemma \ref{lem:gamma0}. On the other hand, if $\mathcal{X}$ contains an arc $\gamma\neq \gamma_0$ such that $g_e(\gamma)=1$ (resp., $g_f(\gamma)=-1$), then all arcs $\delta$ in $\mathcal{X}\setminus\{\gamma_0\}$ must satisfy $g_f(\delta)=0$ (resp., $g_e(\delta)=0$). That is, $\mathcal{X}$ lies in the subset $\mathcal{A}(\mathbf{G}^c)_{e,f}^{t+1,-1}$ (resp., $\mathcal{A}(\mathbf{G}^c)_{e,f}^{1,-t-1}$) for some $t \in [1,n_c-1]$.
   Second, it is clear from the definition that $\mathbf{A}(\mathbf{G}^c)_{e,f}^{>0,>0}$ is decomposed into subsets of the form (c). 
   We finish a proof of Lemma \ref{lem:Gc}.
\end{proof}

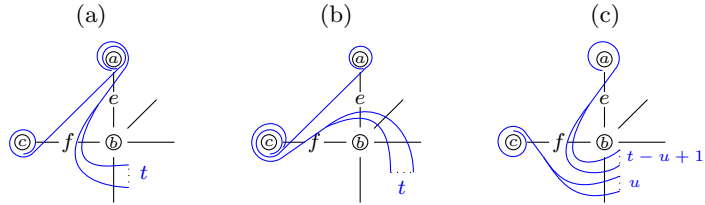
\[
   \renewcommand{\arraystretch}{2}
   \setlength{\tabcolsep}{5mm}
   \begin{tabular}{ccccccccc}
      \small (a) & \small(b) & \small(c) \\
      \begin{tikzpicture}
         \coordinate(bl) at (-0.2,0.6);
         \node(a) at (0.6,0) {\tiny$b$};
         \node(b) at (-0.6,0) {\tiny$c$};
         \draw (a)--node[fill=white,inner sep=1]{\footnotesize$f$}(b); 
         \draw (a)--($(a)+(45:0.8)$) (a)--($(a)+(0:0.8)$) (a)--($(a)+(-90:0.8)$);
         %\draw (b)--($(b)+(120:0.8)$) (b)--($(b)+(180:0.8)$) (b)--($(b)+(-120:0.8)$);
         \node(aa) at (0.6,1.1) {\tiny$a$};
         \draw (a)--node[fill=white,inner sep=1]{\footnotesize$e$}(aa); 
         \draw (a)circle(1mm);
         \draw (b)circle(1mm);
         \draw (aa)circle(1mm);
         \node(xx) at (0,-0.8) {};
         \node(xx) at (0,0.8) {};
         \draw[blue] ($(aa)+(-70:0.13)$)arc(-75:-270:0.15)arc(-270:-420:0.15)--($(c)+(310:0.2)$)arc(310:90:0.18)arc(90:-90:0.16);
         \draw[blue] ($(aa)+(-100:0.15)$)arc(-90:-270:0.19)arc(-270:-400:0.22)--(0.4,0.5);
         \draw[blue] (0.4,0.5)..controls(-0.2,-0.5)and(0.6,-0.3)..(0.8,-0.3);
         \draw[blue] (0.4,0.5)..controls(-0.4,-0.5)and(0.6,-0.6)..(0.8,-0.6);
         \draw[dotted] (0.8,-0.35)--(0.8,-0.6); 
         \node[blue] at (1,-0.4) {\footnotesize$t$};
      \end{tikzpicture}
      &
      \begin{tikzpicture}
         \coordinate(bl) at (-0.2,0.6);
         \node(a) at (0.6,0) {\tiny$b$};
         \node(b) at (-0.6,0) {\tiny$c$};
         \draw (a)--node[fill=white,inner sep=1]{\footnotesize$f$}(b); 
         \draw (a)--($(a)+(45:0.8)$) (a)--($(a)+(0:0.8)$) (a)--($(a)+(-90:0.8)$);
         %\draw (b)--($(b)+(120:0.8)$) (b)--($(b)+(180:0.8)$) (b)--($(b)+(-120:0.8)$);
         \node(aa) at (0.6,1.1) {\tiny$a$};
         \draw (a)--node[fill=white,inner sep=1]{\footnotesize$e$}(aa); 
         \draw (a)circle(1mm);
         \draw (b)circle(1mm);
         \draw (aa)circle(1mm);
         \draw[blue] ($(aa)+(-70:0.13)$)arc(-75:-270:0.15)arc(-270:-420:0.15)--($(c)+(310:0.2)$)arc(310:90:0.18)arc(90:-90:0.16);
         \draw[blue] ($(c)+(-20:0.2)$)arc(-35:90:0.2)arc(90:300:0.25)--(0.15,0.2);
         \draw[blue] (0.15,0.2)..controls(0.8,0.5)and(1,0.2)..(1,-0.4);
         \draw[blue] (0.15,0.2)..controls(0.8,0.7)and(1.3,0.2)..(1.3,-0.4);
         \draw[dotted](1,-0.4)--(1.3,-0.4);
         \node[blue] at (1.15,-0.6) {\footnotesize$t$};

         \node(xx) at (0,-0.8) {};
         \node(xx) at (0,0.8) {};
      
      \end{tikzpicture}
      & 
      \begin{tikzpicture}
         \coordinate(bl) at (-0.2,0.6);
         \node(a) at (0.6,0) {\tiny$b$};
         \node(b) at (-0.6,0) {\tiny$c$};
         \draw (a)--node[fill=white,inner sep=1]{\footnotesize$f$}(b); 
         \draw (a)--($(a)+(45:0.8)$) (a)--($(a)+(0:0.8)$) (a)--($(a)+(-90:0.8)$);
         %\draw (b)--($(b)+(120:0.8)$) (b)--($(b)+(180:0.8)$) (b)--($(b)+(-120:0.8)$);
         \node(aa) at (0.6,1.1) {\tiny$a$};
         \draw (a)--node[fill=white,inner sep=1]{\footnotesize$e$}(aa); 
         \draw (a)circle(1mm);
         \draw (b)circle(1mm);
         \draw (aa)circle(1mm);
         \node(xx) at (0,-0.8) {};
         \node(xx) at (0,0.8) {};
         %\draw[blue] ($(aa)+(-70:0.13)$)arc(-75:-270:0.15)arc(-270:-420:0.15)--($(c)+(310:0.2)$)arc(310:90:0.18)arc(90:-90:0.16);
         \draw[blue] ($(aa)+(-100:0.15)$)arc(-90:-270:0.19)arc(-270:-400:0.22)--(0.4,0.5);
         \draw[blue] (0.4,0.5)..controls(-0.2,-0.5)and(0.6,-0.3)..(0.8,-0.1);
         \draw[blue] (0.4,0.5)..controls(-0.4,-0.5)and(0.6,-0.5)..(0.8,-0.3);
         \draw[dotted] (0.8,-0.1)--(0.8,-0.3); 
         \node[blue] at (1.4,-0.2) {\tiny$t-u+1$};

         \draw[blue] ($(c)+(90:0.15)$)arc(90:-90:0.17)arc(-90:-331:0.2)--(-0.2,-0.2);
         \draw[blue] (-0.2,-0.2)..controls(0,-0.5)and(0.2,-0.7)..(0.8,-0.45);
         \draw[blue] (-0.2,-0.2)..controls(0,-0.5)and(0.2,-0.85)..(0.8,-0.65);

         \draw[dotted](0.8,-0.45)--(0.8,-0.65);
         \node[blue] at (1,-0.55) {\tiny$u$};

      \end{tikzpicture}

   \end{tabular}
   \]

Our proof of (1) is achieved by the following two lemmas. 
The first one follows from the previous result. Now, we denote by $\mathrm{P}(s,t)|_j$ a subset of $\mathrm{P}(s,t)$ consisting of all lattice paths $P$ satisfying $P(-,1)=j$.

\begin{lemma} \label{lem:1} We have the following.
   \begin{enumerate}
   \item The set $\mathcal{A}(\mathbf{G})_e^{>0}$ is a disjoint union of the images of the following injective maps of three types: 
            \begin{equation} \label{eq:A}
               \rho_{f}^{-s,-1} \colon \mathcal{A}(\mathbf{G}^b)_{f}^{-s} \times \mathcal{A}(\mathbf{G}^c)_{e,f}^{t+1,-1} \times \mathrm{P}(s,1) \hookrightarrow \mathcal{A}(\mathbf{G})_e^{s+t},
            \end{equation} 
            where $s,t$ run over all integers $s\in [1,n_b]$, $t \in [1,n_c-1]$;
   
            \begin{equation} \label{eq:B}
               \rho_{f}^{-s,-t-1} \colon \mathcal{A}(\mathbf{G}^b)_{f}^{-s} \times \mathcal{A}(\mathbf{G}^c)_{e,f}^{1,-t-1} \times \mathrm{P}(s,t+1)|_{l} \hookrightarrow \mathcal{A}(\mathbf{G})_e^{l},
            \end{equation}
            where $s,t,l$ run over all integers $s\in [1,n_b]$, $t \in [1,n_c-1]$ and $l\in [1,t+1]$;
            \begin{equation} \label{eq:C}
               \rho_{f}^{s,u} \colon \mathcal{A}(\mathbf{G}^b)_{f}^{s} \times \mathcal{A}(\mathbf{G}^c)_{e,f}^{t-u+1,u} \times \mathrm{P}(s,u) \hookrightarrow \mathcal{A}(\mathbf{G})_e^{t-u+1}, 
            \end{equation}
            where $s,t,u$ run over all integers $s\in [1,n_b]$, $t,u\in [1,n_c-1]$ with $u\leq t$.
   \item For each $j\in [1,n]$, the set $\mathcal{A}(\mathbf{G})_e^j$ is a disjoint union of the following subsets: 
            \begin{eqnarray} 
            &\displaystyle\bigsqcup_{\substack{s\in [1,n_b] \\ t\in [1,n_c-1] \\ s+t=j}}& \rho_f^{-s,-1}(\mathcal{A}(\mathbf{G}^b)_{f}^{-s} \times \mathcal{A}(\mathbf{G}^c)_{e,f}^{t+1,-1} \times \mathrm{P}(s,1));  \label{eq:AA} \\ 
            &\displaystyle\bigsqcup_{\substack{s\in [1,n_b] \\ t\in [1,n_c-1]}}& \rho_f^{-s,-t-1} 
            (\mathcal{A}(\mathbf{G}^b)_{f}^{-s} \times \mathcal{A}(\mathbf{G}^c)_{e,f}^{1,-t-1} \times \mathrm{P}(s,t+1)|_j); \label{eq:BB}\\
            &\displaystyle\bigsqcup_{\substack{s\in [1,n_b] \\ t\in [1,n_c-1]\\ u\in [1,t]\\ t-u+1=j}}& \rho_f^{s,u}(\mathcal{A}(\mathbf{G}^b)_{f}^{s} \times \mathcal{A}(\mathbf{G}^c)_{e,f}^{t-u+1,u} \times \mathrm{P}(s,u)).  \label{eq:CC}
            \end{eqnarray}
   Under the induction hypothesis ($\ast$), the cardinality of (\ref{eq:AA}), (\ref{eq:BB}) and (\ref{eq:CC}) are given by
            \begin{eqnarray} 
               A_{n_b,n_c-1} &:=& \displaystyle\sum_{\substack{s\in [1,n_b] \\ t\in [1,n_c-1] \\ s+t=j}} \binom{2n_b-s-1}{n_b-1}\binom{2n_c-t-3}{n_c-2}, \label{eq:AAA}\\ 
               B_{n_b,n_c-1} &:=& \displaystyle\sum_{\substack{s\in [1,n_b] \\ t\in [1,n_c-1]}} \binom{2n_b-s-1}{n_b-1}\binom{2n_c-t-3}{n_c-2}\binom{s+t-j-1}{s-j}\quad \text{and} \label{eq:BBB}\\ 
               C_{n_b,n_c-1} &:=& \displaystyle\sum_{\substack{s\in [1,n_b] \\ t\in [1,n_b-1]\\ u\in [1,t]\\ t-u+1=j}} \binom{2n_b-s-1}{n_b-1}\binom{2n_c-t-3}{n_c-2}\binom{s+t-j-1}{t-j}. \label{eq:CCC}
            \end{eqnarray}
      respectively.
   \end{enumerate}
\end{lemma}
The second one is an equation on binomial coefficients. 

\begin{lemma} \label{lem:2}
   Let $m=p+q-1$, where $p,q$ are integers with $p>0$ and $q>1$.
   For any $j\in [1,m]$, the following identity holds: 
   \[
      \binom{2m-j-1}{m-1} = A_{p,q-1}(j) + B_{p,q-1}(j) + C_{p,q-1}(j).
   \]
   where $A,B,C$ are in Lemma \ref{lem:1}(2). 
\end{lemma}

Using Lemmas \ref{lem:1} and \ref{lem:2}, we prove Theorem \ref{thm:main}(1). 
Proofs of Lemma \ref{lem:1} and \ref{lem:2} are given in Sections \ref{sec:lem:1} and \ref{sec:lem:2} respectively.

\begin{proof}[Proof of Theorem \ref{thm:main}(1)]
   Let $j$ be a non-negative integer. 
   For $j\notin [1,n]$, the set $\mathcal{A}(\mathbf{G})_e^j$ is empty by Proposition \ref{prop:greater}. 
   Suppose that $j\in [1,n]$. We denote by $\mathbf{G}^{\rm op}$ the opposite plane tree of $\mathbf{G}$ (see Section \ref{seq:opp}).
   Applying Lemmas \ref{lem:1} and \ref{lem:2} to both $\mathbf{G}$ and $\mathbf{G}^{\rm op}$, we get the identities 
   \[
      \#\mathcal{A}(\mathbf{G})_e^{j} = \binom{2n-j-1}{n-1} = \#\mathcal{A}(\mathbf{G}^{\rm op})_e^{j} .
   \]
   On the other hand, by Proposition \ref{prop:opp}, we have 
   \[
      \#\mathcal{A}(\mathbf{G})_e^{-j} = \#\mathcal{A}(\mathbf{G}^{\rm op})_e^{j}.
   \]
   Therefore, we get the desired equation (\ref{eq:formula}) for $\mathbf{G}$. 
\end{proof}

We finish a proof of Theorem \ref{thm:main}.

\subsection{Proof of Lemma \ref{lem:1}} \label{sec:lem:1}

In this subsection, we prove Lemma \ref{lem:1}.

\begin{proof}[Proof of Lemma \ref{lem:1}(1)]
   We compute the $g$-vector of the images of maps in (\ref{eq:AA})-(\ref{eq:CC}).
\begin{enumerate}
\item[(i)] Let
\[
   (\mathcal{X}^b,\mathcal{X}^c,P_0)\in \mathcal{A}(\mathbf{G}^b)_f^{-s}\times \mathcal{A}(\mathbf{G}^c)_{e,f}^{t+1,-1}\times \mathrm{P}(s,1),
\]
where $P_0:=\{(j,1)\mid j \in [1,s]\}$ is a unique element of $\mathrm{P}(s,1)$. 
By Lemma \ref{lem:gamma0}, $\mathcal{X}^c$ contains the arc $\gamma_0$ in Lemma \ref{lem:gamma0}, and there are no arcs $\gamma$ in $\mathcal{X}^c\setminus \{\gamma_0\}$ satisfying $(g_e(\gamma),g_f(\gamma)) \neq (0,0)$. By Theorem \ref{thm:gluing}, we have 
\[
   g_e(\mathcal{X}^b \underset{e_{P_0}}{\times}\mathcal{X}^c) = g_e(\gamma_0)\cdot s + g_e(\mathcal{X}^c\setminus \{\gamma_0\}) = s+t. 
\]

\item[(ii)] Let
\[
   (\mathcal{X}^b,\mathcal{X}^c,P)\in \mathcal{A}(\mathbf{G}^b)_f^{s}\times \mathcal{A}(\mathbf{G}^c)_{e,f}^{1,-t-1}\times \mathrm{P}(s,t+1)|_l.
\]
By Lemma \ref{lem:gamma0}, $\mathcal{X}^c$ contains the arc $\gamma_0$, and an intersection point $p$ of $\gamma_0$ and $\alpha_e$ lies in the first position in clockwise ordering around $c$. In addition, there are no arcs $\gamma$ in $\mathcal{X}^c\setminus \{\gamma_0\}$ satisfying $(g_e(\gamma),g_f(\gamma)) \neq (0,0)$. Therefore, 
\[
   g_e(\mathcal{X}^b\underset{e_P}{\times}\mathcal{X}^c) = g_e(\gamma_0)\cdot P(-,1) = l. 
\]
\item[(iii)] Let 
\[
   (\mathcal{X}^b,\mathcal{X}^c,P)\in \mathcal{A}(\mathbf{G}^b)_f^{s}\times \mathcal{A}(\mathbf{G}^c)_{e,f}^{t-u+1,u}\times \mathrm{P}(s,u).
\]
By Lemma \ref{lem:gamma0}, there are no arcs $\gamma\in \mathcal{X}^c$ satisfying $(g_e(\gamma),g_f(\gamma))\neq (0,0)$.
Therefore, 
\[
   g_e(\mathcal{X}^b \underset{e_P}{\times}\mathcal{X}^c) = g_e(\mathcal{X}^c) = t-u+1.
\]
\end{enumerate}
Consequently, we get (\ref{eq:AA})-(\ref{eq:CC}). In addition, the set $\mathcal{A}(\mathbf{G})_e^{>0}$ is a disjoint union of these images by Theorem \ref{thm:gluing} with Lemma \ref{lem:Gc}.
\end{proof}

\begin{proof}[Proof of Lemma \ref{lem:1}(2)]
   By Lemma \ref{lem:1}(1), we clearly have (\ref{eq:AA})-(\ref{eq:CC}). In the following, we determine their cardinality under the induction hypothesis ($\ast$). 
   First, we have 
   \[
      \#\mathcal{A}(\mathbf{G}^b)_f^s = \#\mathcal{A}(\mathbf{G}^b)_f^{-s} = \binom{2n_b-s-1}{n_b-1}
   \]
   by induction hypothesis ($\ast$). Second, we have 
   \[
      \#\mathrm{P}(s,t) = \binom{s+t-2}{s-1} = \binom{s+t-2}{t-1} \quad \text{and} \quad \#\mathrm{P}(s,t+1)|_l = \#\mathrm{P}(s-l+1,t)
   \]
   for any positive integers $s,t$ and an integer $l\in [1,s]$.

   Finally, we show that
   \begin{equation} \label{eq:tu}
      \#\mathcal{A}(\mathbf{G})_{e,f}^{t+1,-1} = \#\mathcal{A}(\mathbf{G})_{e,f}^{1,-t-1} = \#\mathcal{A}(\mathbf{G})_{e,f}^{t-u+1,u} = \binom{2n_c-t-3}{n_c-2}.
   \end{equation}
   Consider a flip $\mathbf{H}:=\mu_e(\mathbf{G}^c)$ at $e$ (see Definition \ref{def:flip}), it is equivalent to see that 
   \[
      \#\mathcal{A}(\mathbf{H})_{e,f}^{-t-1,t} = \#\mathcal{A}(\mathbf{H})_{e,f}^{-1,-t} = \#\mathcal{A}(\mathbf{H})_{e,f}^{-t+u-1,t+1} = \binom{2n_c-t-3}{n_c-2}.
   \]
   by Proposition \ref{prop:flip}.

   Now, the edge $f$ determines plane subtrees $\mathbf{H}^b, \mathbf{H}^c$ of $\mathbf{H}$ satisfying $H_1^b\cap H_1^c =\{f\}$ and $H_1^b\cup H_1^c=H_1$. We may assume that, for $v\in \{b,c\}$, $\mathbf{H}^v$ contains the edge $f$ as its external edge with external vertex $v$. Namely, $H_1^b=\{e,f\}$ and $H_1^c= H_1\setminus \{e\}$. 
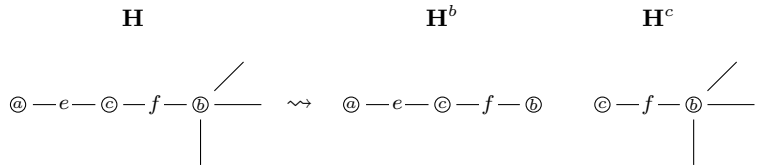
\[
   \renewcommand{\arraystretch}{1}
   \setlength{\tabcolsep}{1mm}
   \begin{tabular}{ccccccccc}
      \small$\mathbf{H}$ & &\small$\mathbf{H}^b$ & \small$\mathbf{H}^c$ \\
      \begin{tikzpicture}
         %\draw[red] (0,0.6)--(0,-0.7);
         %\node(al) at (0,-0.85) {\color{red}\footnotesize$\alpha_e$};
         \node(b) at (0.6,0) {\tiny$b$};
         \node(c) at (-0.6,0) {\tiny$c$};
         \node(a) at (-1.8,0) {\tiny$a$};
         \draw (b)--node[fill=white,inner sep=1]{\footnotesize$f$}(c); 
         \draw (b)--($(b)+(45:0.8)$) (b)--($(b)+(0:0.8)$) (b)--($(b)+(-90:0.8)$);
         \draw (c)--node[fill=white,inner sep=1]{\footnotesize$e$}(a);
         
         \draw (b)circle(1mm);
         \draw (c)circle(1mm);
         \draw (a)circle(1mm);
   
         \node(xx) at (0,-0.8) {};
         \node(xx) at (0,0.8) {};
         
      \end{tikzpicture}
      &
   \begin{tikzpicture}
      \node(sq) at (0,0) {$\rightsquigarrow$};
      \node(xx) at (0,-0.8) {};
         \node(xx) at (0,0.8) {};
   \end{tikzpicture}
      &
      \begin{tikzpicture}
         %\draw[red] (0,0.6)--(0,-0.7);
         %\node(al) at (0,-0.85) {\color{red}\footnotesize$\alpha_e$};
         \node(b) at (0.6,0) {\tiny$b$};
         \node(c) at (-0.6,0) {\tiny$c$};
         \node(a) at (-1.8,0) {\tiny$a$};
         \draw (b)--node[fill=white,inner sep=1]{\footnotesize$f$}(c);
         \draw (c)--node[fill=white,inner sep=1]{\footnotesize$e$}(a);
         
         \draw (b)circle(1mm);
         \draw (c)circle(1mm);
         \draw (a)circle(1mm);
   
         \node(xx) at (0,-0.8) {};
         \node(xx) at (0,0.8) {};
         
      \end{tikzpicture}
      & \quad
      \begin{tikzpicture}
         \coordinate(bl) at (-0.2,0.6);
         \node(a) at (0.6,0) {\tiny$b$};
         \node(b) at (-0.6,0) {\tiny$c$};
         \draw (a)--node[fill=white,inner sep=1]{\footnotesize$f$}(b); 
         \draw (a)--($(a)+(45:0.8)$) (a)--($(a)+(0:0.8)$) (a)--($(a)+(-90:0.8)$);
         %\draw (b)--($(b)+(120:0.8)$) (b)--($(b)+(180:0.8)$) (b)--($(b)+(-120:0.8)$);
         
         \draw (a)circle(1mm);
         \draw (b)circle(1mm);
         \node(xx) at (0,-0.8) {};
         \node(xx) at (0,0.8) {};
      
      \end{tikzpicture}
   \end{tabular}
   \]

   For the plane tree $\mathbf{H}^b$ having $2$ edges, we have
   \[
      \mathcal{A}(\mathbf{H}^b)_{e}^{<0} = \mathcal{A}(\mathbf{H}^b)_{e,f}^{-2,1}\sqcup \mathcal{A}(\mathbf{H}^b)_{e,f}^{-1,-1} \sqcup \mathcal{A}(\mathbf{H}^b)_{e,f}^{-1,2} = \{A_1\} \sqcup \{A_2\} \sqcup \{A_3\}. 
   \]
   Here, $A_i$ are complete collections in Example \ref{ex:collections} satisfying
   \[
      (g_e(A_1),g_f(A_1))=(-2,1), \quad (g_e(A_2),g_f(A_2))=(-1,-1) \quad \text{and} \quad (g_e(A_3),g_f(A_3))=(-1,2).
   \] 
   From this decomposition, we have  
   \begin{eqnarray}
      \mathcal{A}(\mathbf{H})_{e,f}^{-t-1,t}  &=& \rho_f^{1,t}(\{A_1\}\times \mathcal{A}(\mathbf{H}^c)_{f}^t, \{P_0\}), \label{eq:aa} \\ 
      \mathcal{A}(\mathbf{H})_{e,f}^{-1,-t} &=& \rho_f^{-1,-t}(\{A_2\}\times \mathcal{A}(\mathbf{H}^c)_{f}^{-t}, \{P_0\}) \ \text{and} \label{eq:bb}\\ 
      \mathcal{A}(\mathbf{H})_{e,f}^{-t+u-1,t+1} &=& \rho_f^{2,u}(\{A_3\}\times \mathcal{A}(\mathbf{H}^c)_{f}^{t}, \{P_u\}). \label{eq:cc}
   \end{eqnarray}
   by Theorem \ref{thm:gluing}.
   Here, $P_0:=\{(1,k)\mid k\in [1,t]\}$ is a unique element of $\mathrm{P}(1,t)$ and 
   \[
      \mathrm{P}(2,t) = \{P_u:= \{(1,1),\ldots, (1,t-u+1),(2,t-u+1),\ldots, (2,t)\} \mid u\in [1,t]\}. 
   \]
   Since the maps $\rho_f^{-,-}$ are injective, the cardinality of (\ref{eq:aa})-(\ref{eq:cc}) are precisely
   \[
      \#\mathcal{A}(\mathbf{H}^c)_{f}^t = \#\mathcal{A}(\mathbf{H}^c)_{f}^{-t} = \binom{2n_c-t-3}{n_c-2} 
   \]
   by induction hypothesis ($\ast$). Consequently, we obtain the desired equation (\ref{eq:tu}). 

   It finishes a proof of Theorem \ref{lem:1}(2).
\end{proof}

\subsection{Proof of Lemma \ref{lem:2}} \label{sec:lem:2} 
We give a proof of Lemma \ref{lem:2}.

\begin{proof}[Proof of Lemma \ref{lem:2}]
Fix $m=p+q-1$ for integers $p>0$ and $q>1$. 
We show the following desired equation by induction on $q$:
\begin{equation} \label{Part4:eq_desireddd}
   F_{p+q-1}(j)= A_{q-1}(j)+ B_{q-1}(j)+ C_{q-1}(j) 
\end{equation}
for all $j \in [1,p+q-1]$. Here, we set 
\begin{eqnarray} 
   A_{q-1}(j) := A_{p,q-1}(j) &=& \displaystyle\sum_{\substack{s\in [1,p] \\ t\in [1,q-1] \\ s+t=j}} F_p(s)F_{q-1}(t), \nonumber\\ 
   B_{q-1}(j) := B_{p,q-1}(j) &=& \displaystyle\sum_{\substack{s\in [1,p] \\ t\in [1,q-1]}} F_{p}(s)F_{q-1}(t)P(s-j+1,t), \nonumber\\ 
   C_{q-1}(j) := C_{p,q-1}(j) &=& \displaystyle\sum_{\substack{s\in [1,p] \\ t\in [1,q-1]\\ u\in [1,t]\\ t-u+1=j}} F_p(s)F_{q-1}(t)P(s,t-j+1), \nonumber
\end{eqnarray}
where $F_p(s):=\binom{2p-s-1}{p-1}$ and $P(s,t):=\binom{s+t-2}{s-1}=\binom{s+t-2}{t-1}$.

First, we assume that $q=2$. 
By definition, we have 
\[
   A_{1}(j)=
   \begin{cases} 0 &\text{if $j=1$} \\ F_p(j-1) & \text{else,}\end{cases}\quad 
   B_{1}(j)= 
   \begin{cases} F_{p+1}(2) & \text{if $j=1$} \\ \sum_{s=j}^p F_{p}(s) & \text{else,} \end{cases} \quad  
   C_{1}(j)=
   \begin{cases} F_{p+1}(2) &\text{if $j=1$} \\ 0 & \text{else.}\end{cases}
\]
So, we have 
\[A_1(1) + B_1(1) + C_1(1) = 2F_{p+1}(2) = F_{p+1}(1)
\]
for $j=1$, and  
\[
   A_1(j) + B_1(j) + C_1(j) = \sum_{s=j-1}^p F_p(s) = F_{p+1}(j)
\]
for $j \in \{2, \ldots, p+1\}$. Therefore, the desired equation holds for $q=2$. 
   
Second, we assume that the equation (\ref{Part4:eq_desireddd}) holds for $q$. Under the induction hypothesis, we show that 
\begin{equation} \label{Part4:eq_desired}
      F_{p+q}(j)= A_{q}(j)+ B_{q}(j)+ C_{q}(j) 
\end{equation}
for $j\in [1,p+q]$.
We first consider $j>1$. From the induction hypothesis, we have 
\[
   F_{p+q}(j)  = \sum_{i=j-1}^{p+q-1} F_{p+q-1}(i)  = \sum_{i=j-1}^{p+q-1}A_{q-1}(i) +\sum_{i=j-1}^{p+q-1} B_{q-1}(i) +\sum_{i=j-1}^{p+q-1}C_{q-1}(i). 
\]
We calculate each summand in the right-hand side.
   
\begin{enumerate}
   
\item[(a)]
By definition, we have
\begin{eqnarray}
   \sum_{i=j-1}^{p+q-1}A_{q-1}(i) &=& 
   \sum_{i=j-1}^{p+q-1} \sum_{\substack{s \in [1,p] \\ t \in [1,q-1] \\ i=s+t}}F_p(s) F_{q-1}(t) 
   \nonumber
   \\
   &=& 
   \sum_{s=j-1}^{p} F_p(s) \sum_{t=1}^{q-1} F_{q-1}(t) + 
   \sum_{s=1}^{j-2}F_{p}(s) \sum_{t=j-s-1}^{q-1} F_{q-1}(t).  \nonumber \\
   &=& F_{p+1}(j)F_{q}(2) + \sum_{s=1}^{j-2} F_{p}(s)F_{q}(j-s) \nonumber \\ 
   &=& A_{q}(j) - F_p(j-1)F_q(1) + F_{p+1}(j) F_{q}(2). \label{Part4:eq_I}
\end{eqnarray}
   
\item[(b)] We continue our calculation.  
\begin{eqnarray}
   \sum_{i=j-1}^{p+q-1}B_{q-1}(i) &=& 
   \sum_{t=1}^{q-1} F_{q-1}(t) \Big\{ \sum_{i=j-1}^{p+q-1}\sum_{s=i}^{p} F_p(s) P(s-i+1, t) \Big\} \nonumber \\
   &=& \sum_{t=1}^{q-1} F_{q-1}(t)\Big\{\sum_{s=j-1}^p F_p(s) + \sum_{s=j}^pF_p(s) \sum_{\lambda=1}^{s-j+1} P(\lambda+1 , t)\Big\}. \nonumber \\
   &=& F_{p+1}(j)F_q(2) + \sum_{t=1}^{q-1} F_{q-1}(t) 
   \Big\{\sum_{s=j}^pF_p(s) \sum_{\lambda=1}^{s-j+1} P(\lambda+1, t)\Big\} \nonumber\\
   &=& F_{p+1}(j)F_q(2) + \sum_{t=1}^{q-1} F_{q-1}(t) 
   \Big\{\sum_{s=j}^pF_p(s) \sum_{\mu=1}^{t} P(s-j+1, \mu+1)\Big\} \nonumber\\ 
   &=& F_{p+1}(j)F_q(2) + \sum_{t=1}^{q-1}\sum_{u=t}^{q-1} F_{q-1}(u) \sum_{s=j}^pF_{p}(s)P(s-j+1,t+1) \nonumber\\
   &=& F_{p+1}(j)F_q(2) + \sum_{t=1}^{q-1}\sum_{s=j}^{p}F_{q}(t+1) F_p(s)P(s-j+1,t+1) \nonumber \\
   &=& B_{q}(j) - F_{p+1}(j+1)F_q(1) +F_{p+1}(j)F_{q}(2),  \label{Part4:eq_II}
\end{eqnarray}
where the last two equalities are obtained by replacing $t \rightarrow t+1$.

\item[(c)] Finally, we get 
\begin{eqnarray}
   \sum_{i=j-1}^{p+q-1}C_{q-1}(i) &=& 
   \sum_{s=1}^p F_p(s) \Big\{ \sum_{i=j-1}^{p+q-1} \sum_{t=1}^{q-i} F_{q-1}(t+i-1)P(s,t)\Big\} 
   \nonumber\\ 
   &=& \sum_{s=1}^p F_p(s) \sum_{t=1}^{q-j+1} F_q(t+j-1)P(s, t) \nonumber \\
   &=& C_q(j). \label{Part4:eq_III}
\end{eqnarray}
\end{enumerate}
   
Adding (\ref{Part4:eq_I})--(\ref{Part4:eq_III}), we get the desired equation (\ref{Part4:eq_desired})
for $j\in \{2,\ldots, p+q\}$ by
\[
   F_{p+1}(j)F_{q}(2)- F_p(j-1)F_q(1) - F_{p+1}(j+1)F_q(1) + F_{p+1}(j)F_q(2)=0.  
\]  
   On the other hand, we get the equation for $j=1$ from the previous result:
\[
   F_{p+q}(1) =2F_{p+q}(2)=2A_q(2) + 2B_q(2)+ 2C_{q}(2)= A_q(1) + B_q(1)+ C_{q}(1).
\]   
We finish a proof of Lemma \ref{lem:2}. 
\end{proof}

\medskip\noindent{\bf Acknowledgements}. 
The author would like to express his deep gratitude to his supervisor Osamu Iyama 
for his support and advice. He would also like to thank Aaron Chan for helpful comments and suggestions. He is a Research Fellows of Society for the Promotion of Science (JSPS). This work was supported by JSPS KAKENHI Grant Number JP19J11408.

\bibliographystyle{alpha}
\bibliography{bibitem}

\end{document}